\documentclass[11pt,leqno]{article}

\usepackage[T1]{fontenc}

\usepackage{amsmath}
\usepackage{amsfonts}
\usepackage{amsthm}
\usepackage{amssymb}
\usepackage{color}
\usepackage{graphicx}
\usepackage{lmodern}
\usepackage{dsfont}
\usepackage{mathrsfs}
\usepackage{appendix}

\usepackage{cancel}
\usepackage[normalem]{ulem}
\let\d=\delta

\let\G=\Gamma

\newcommand{\Z}{{\mathbb Z}}

\newcommand{\cPmzeroGammaLipCV}{\cP_{m_0}(\Gamma^{\textrm{Lip}}_{C,V})}
\newcommand{\GammamuoptC}  {\Gamma^{\mu,{\rm opt}}_C}

\newcommand{\I}{{\mathcal I}}

\newcommand{\R}{{\mathbb R}}

\newcommand{\N}{{\mathbb N}}

\newcommand{\cG}{{\mathcal G}}

\newcommand{\cP}{{\mathcal P}}

\newcommand{\car}{{\mathds{1}}}
\newcommand{\half}{\frac{1}{2}}

\let\ds=\displaystyle

\def\R{{\mathbb R}}
\def\N{{\mathbb  N}}
\def\N{{\mathbb N}}

\let\d=\delta

\let\o=\omega

\def\qed{\hspace*{\fill} $\Box$\par\medskip}

\def\co{\overline{\textrm{co}}}
\numberwithin{equation}{section}

\newtheorem{theorem}{Theorem}[section]
\newtheorem{lemma}[theorem]{Lemma}
\newtheorem{proposition}[theorem]{Proposition}
\newtheorem{definition}{Definition}[section]
\newtheorem{corollary}[theorem]{Corollary}
\newtheorem{remark}[theorem]{Remark}
\newtheorem{example}{Example}[section]

  {\hfill$\Box$\bigskip\par}

\def\@maybe{\ifx[\testchar \let\next\@Opt
          \else \let\next\@NoOpt \fi \next}
\def\@Opt[#1]{{\it Proof of #1.\ }}\def\@NoOpt{{\it Proof.\ }}
\begin{document}

\title{\Large \bf First order Mean Field Games on networks}
\author{{\large \sc Yves Achdou\thanks{Universit\'e de Paris Cit\'e and Sorbonne Universit\'e, CNRS, Laboratoire Jacques-Louis Lions, (LJLL), F-75006 Paris, France, achdou@ljll-univ-paris-diderot.fr}, Paola Mannucci\thanks{Dipartimento di Matematica ``Tullio Levi-Civita'', Universit\`a di Padova, mannucci@math.unipd.it}, Claudio Marchi \thanks{Dipartimento di Matematica ``Tullio Levi-Civita'', Universit\`a di Padova, claudio.marchi@unipd.it}, Nicoletta Tchou \thanks{Univ Rennes, CNRS, IRMAR - UMR 6625, F-35000 Rennes, France, nicoletta.tchou@univ-rennes1.fr}}} 

%\author{{\large \sc Paola Mannucci\thanks{Dipartimento di Matematica ``Tullio Levi-Civita'', Universit\`a di Padova, mannucci@math.unipd.it}, Claudio Marchi \thanks{Dipartimento di Ingegneria dell'Informazione, Universit\`a di Padova, claudio.marchi@unipd.it},}\\
% {\large \sc Carlo Mariconda\thanks{Dipartimento di Matematica ``Tullio Levi-Civita'', Universit\`a di Padova, carlo.mariconda@unipd.it},   Nicoletta Tchou\thanks{Univ Rennes, CNRS, IRMAR - UMR 6625, F-35000 Rennes, France, nicoletta.tchou@univ-rennes1.fr}} France, nicoletta.tchou@univ-rennes1.fr}}%\\
%% \rm Universit\`a degli Studi di Padova, Universit\'e de Rennes 1
%}

%%%%%%%%%%%%%%%%%%%%%%%%%%%%%%%%%%%%%%%%%%%%%%%%%%%%%%%%%%%%%%%%%
\maketitle
%
%
%\date{\today}
%\date{28 March, 2023}
\begin{abstract}
This paper is devoted to  finite horizon deterministic mean field games in which the state space is a network. The  agents control their velocity, and when they occupy a vertex, they can enter into any incident edge. The running and terminal costs are assumed to be continuous in each edge but not necessarily globally continuous on the network. A Lagrangian formulation is proposed and studied. It leads to relaxed equilibria consisting of probability measures on admissible trajectories. 
The existence of such relaxed equilibria is obtained. The proof requires the existence of optimal trajectories and a closed graph property for the map which associates to each point the set of optimal trajectories starting from that point.\\
To any relaxed equilibrium corresponds a mild solution of the mean field game, i.e. a pair $(u,m)$ made of the value function $u$  of a related optimal control problem, and a family $m= (m(t))_t$ of
probability measures on the network. Given $m$, the value function $u$ is characterized by a Hamilton-Jacobi problem on the network. Regularity properties of $u$ and a weak form of a Fokker-Planck equation satisfied by $m$ are investigated.
\end{abstract}


\tableofcontents

%
%Introduction
%
\section{Introduction}
\label{intro}
The theory of Mean Field Games (MFGs in short) introduced in the pioneering articles  of Lasry and Lions \cite{LL1,LL2,LL3}, deals with the asymptotic behaviour of differential games, either deterministic or stochastic, as the number of players tends to infinity. The major part of the literature on deterministic MFGs addresses situations in which the state space is either $\R^d$ or the flat torus $\R^d/\Z^d$, and in which the dynamics of the players is strongly controllable. In such cases, the mean field game is determined by the pair made of the distribution of states at all times and  the optimal value of a representative agent. The latter quantities satisfy a system of PDEs coupling a continuity equation (forward in time) and a Hamilton--Jacobi (HJ) equation (backward in time), see~\cite{C}.

Assuming that the dynamics are strongly controllable,
Cannarsa et al, \cite{CC,CCC1,CCC2}, have studied MFGs in which the agents are constrained to remain in the closure of a regular bounded open domain of $\R^d$.  With such state constraints, the distribution of states may become singular, as it was first observed in \cite{AHLLM}, and it becomes difficult to write boundary conditions for the continuity equation  (see also Section~\ref{sect:MFGequil} below for some examples of formation and propagation of Dirac masses). For this reason,
 Cannarsa et al, following ideas contained in \cite{BB00,BC15,CM16}, introduce a  notion of relaxed equilibrium which is defined in a Lagrangian setting rather than with PDEs. The evolution of the game is described in terms of probability measures defined on a set of admissible trajectories, instead of time-dependent probability measures defined on the state space. In the same vein, Mazanti and Santambrogio,~\cite{MS19}, obtain the existence of relaxed equilibria for minimal time MFGs, in which each  agent aims at exiting a given closed subset of a general compact metric space in minimal time and faces congestion effects (her speed cannot exceed a bound depending on the density of players). See also \cite{DM} for similar models in the Euclidean setting. In \cite{AMMT2}, the authors of the present paper  prove the existence of relaxed equilibria for deterministic state constrained MFGs in which the agents control their acceleration. This is an example of state constrained MFGs in which  the strong controllability property does not hold.

The present paper aims at studying relaxed equilibria for deterministic MFGs in which the state space is a network, i.e. a subset of $\R^d$ made of a finite number of edges and vertices. Optimal control problem on junctions, networks or stratified sets is a rather recent field which contains a number of interesting open problems (see \cite{ACCT,IMZ,IM,BBC,LS, LS17, Mor,BC2023}).  The aforementioned paper~\cite{MS19} on minimal time MFGs also applies to networks. Stochastic MFGs on networks (each agent is subject to an independent noise) have been studied in \cite{ADLT2020} (see also \cite{CM16,ADLT2018} for infinite horizon problems). Finally, in the recent preprint \cite{GRA}, Gomes et al study a class of stationary MFG on networks and their relationship with Wardrop equilibria. 
The present paper can be considered as a first step of a more general research  project on deterministic MFGs on networks that we intend to pursue.

 For simplicity, we hereafter focus  on a junction, i.e. $N$ half-lines in $\R^d$ glued together at a single vertex, say the origin. Yet, all the results below may be generalized for general networks with more than one vertices and edges of possibly finite lengths. Given the time evolution of the distribution of the players, each agent solves an optimal problem with finite time horizon. We assume that  the agents control their velocity. In particular, when an agent is at the vertex, she can choose either to remain still  or to enter any edge.  The running and terminal costs depend on the distribution of agents in a non local, regularizing manner, but are not supposed to be continuous across the vertex (the costs may change from one edge to the next). We also restrict ourselves to running costs which depend quadratically on the velocity. Finally, there is a distinct running cost for staying at the vertex.  

The first part of the present paper is devoted to optimal control problems on the network, (which arise if the distribution of states in the MFG is given).
The main results concerning optimal control are as follows:
the existence of an optimal trajectory for any initial state, a closed graph property for the map which associates to each point on the network the set of optimal trajectories starting from that point, 
Euler-Lagrange conditions for the optimal control, the characterization of the value function of the  optimal control problem as the generalized viscosity solution of an Hamilton-Jacobi problem posed on the network with suitable conditions at the vertex (the definition of generalized viscosity solution will be recalled),  the local or global Lipschitz regularity  of the value function. The second part of the paper deals with relaxed equilibria for MFGs on the network. The existence of the latter is proved using  Kakutani's fixed point theorem applied to a suitable multivalued map, which requires in particular a  closed graph property. 
To any relaxed equilibrium, it is then possible to associate a family of time-dependent probability measures  on the state space $(m(t))_t$ and the value function $u$ of a suitable optimal control problem involving $m$. All the results of the first part of the paper apply to the latter optimal control problem. In particular, some regularity properties of $u$ can be deduced. It is also possible to prove that $m$ solves a continuity equation in a  weak sense and to give information on the propagation of its singularities.  The pair $(u,m)$ is named a {\sl mild solution} of the MFG, see \cite{CC}.

This paper is organized as follows. The remaining part of Section \ref{intro} contains the description of the geometry  and the definition of some notations. Section~\ref{sec:OC} is devoted to optimal control problems. In particular, we obtain the existence of an optimal trajectory for every starting point, a closed graph property for the map that associates to each point the set of optimal trajectories, and study the value function (mainly, its characterization as the viscosity solution of a HJ problem on the network and some regularity properties). Section~\ref{sect:MFGequil} concerns deterministic MFGs on the junction. Relying on the results of Section~\ref{sec:OC}, we prove the existence of a relaxed MFG equilibrium and study the related mild solutions.
\subsection{Notations}\label{sec:notations}

Throughout this paper, the notation $C_b$ means {\sl continuous and bounded}.
\paragraph{The junction.}
We adopt the notations of ~\cite{AOT}. In the whole paper, the state space is  a {\sl junction} in $\R^d$ with $N$ ($N>1$) semi-infinite straight {\sl edges}, denoted by $(J_i)_{i=1,\dots, N}$. Let the edge $J_i$ be the closed half-line $\R^+ e_i$, and 
the vectors $e_i$ be two by two distinct unit vectors in $\R^d$.
The {\sl junction} $\cG$ is obtained by gluing the half-lines $J_i$ at the origin $O$:
  \begin{displaymath}
  \cG=\bigcup_{i=1}^N J_i.
\end{displaymath}
For a vector $\xi$ aligned with a given $e_i$, we set $\bar \xi=\xi\cdot e_i$. \\
The geodetic distance $d(x,y)$ between two points $x,y$ of $\cG$ is 
\begin{displaymath}
    d(x,y)=\left\{  \begin{array}[c]{ll}
|x-y|\quad \hbox{if } x,y \hbox{ belong to the same edge } J_i \\
|x|+|y|  \quad \hbox{if } x,y \hbox{ belong to different edges } J_i \hbox{ and } J_j.
  \end{array}\right.
\end{displaymath}

If $\varphi$ is a function defined on $J_i$, 
we will sometimes use the same notation $\varphi$ for the function $\R^+\ni \bar x\mapsto \varphi(\bar x e_i)$.
\paragraph{Gradient of a function.}
Let  $C^1(\cG)$ be  the set of continuous functions $\varphi\in C(\cG)$ such that, 
for every $i=1,\dots,N$, the restriction of $\varphi$ to the edge~$J_i$, $\varphi_{\mid J_i}$ belongs to $C^1(\R^+)$; moreover, for $\varphi\in C^1(\cG)$, we set
\begin{equation}\label{eq:def_deriv}
D\varphi(x)=\left\{\begin{array}{ll}
D \varphi_{\mid J_i}&\qquad\textrm{if }x\in J_i\setminus\{O\},\\
\left(D \varphi_{\mid J_1},\dots,D \varphi_{\mid J_N}\right)&\qquad\textrm{if }x=O.
\end{array}\right.
\end{equation}
Observe that $D\varphi(x)$ is $1$-dimensional when the point~$x$ lies in the interior of a given edge while it is $N$-dimensional when $x$ coincides with the vertex~$O$.

In a similar manner, let $C^1(\cG\times [0,T])$ be  the set of continuous functions $\varphi\in C(\cG\times [0,T])$ such that for any $1\le i\le N$, the restriction  $\varphi_{\mid J_i\times [0,T]}$ belongs to $C^1(J_i\times[0,T])$.

\section{Deterministic optimal control on networks}\label{sec:OC}
We consider  optimal control problems on $\cG$ with horizon~$T>0$ and different running costs in the edges and at the vertex. The set of controls, the dynamics and the running cost associated to a given edge $J_i$ are respectively denoted by $A_i$,  $\widetilde F_i$ and  $\tilde \ell_i$. For the sake of simplicity, we shall focus on the case where $\tilde f_i=\alpha $, i.e. the agent directly chooses  its velocity,   and where the running cost is $\tilde \ell_i(x,t,\alpha)=\ell_i(x,t)+|\alpha|^2/2$ (it depends separately on the control and on the state variable). However, what follows may be easily extended to a more general setting, namely
\begin{itemize}
\item a  network instead of a simple junction
\item  functions $\tilde f_i$ with a linear or sublinear growth at  infinity and such that $\tilde f_i(A_i)$ contains a neighborhood of $0$ (strong controllability assumption)
\item running costs which depend separately on the control and the state variable and are strongly convex in the control.
\end{itemize}
More precisely, we make the following assumptions:
\begin{description}
\item{[H0]} In order to avoid confusion between the control sets, we set $A_i=\{i\}\times \R$  for $i=0,\dots,N$.  Hence, the sets $A_i$ are disjoint. We set $A=\bigcup_{i=0}^N A_i$. For $a=(i,\bar a)\in A$, we set $|a|=|\bar a|$ and, with an abuse of notations, we shall write indifferently $\bar a e_i$ and  $(i,\bar a)$.
%Note that what follows is actually does not depend on the set $A_0$, which has been %introduced for convenience only.
\\ %so is a non empty compact subset of $A$
Let $f_i: J_i\times A_i \to \R$ be defined by $f_i(x,a)=\bar a$ for $a=(i,\bar a)$. We will use the notation $F_i(x)$ for the set $\{f_i(x,a) e_i, a\in A_i\}=\R e_i$ for $x\in J_i$ ($i=1,\dots,N$). We also set $F_0(O)=\{0_{\R^d}\}$.
\item{[H1]} For $i=1,\dots,N$, the running costs $\ell_i$ belong to $C_b(J_i\times[0,T])$. Let us also introduce a specific cost for staying at the origin, namely $\ell_*:[0,T]\rightarrow \R$, continuous and bounded.\\
For $i=1,\dots,N$, the terminal costs $g_i$ belong to $C_b(J_i)$. Let $g_*$ be a fixed number.
\end{description}

In the remaining part of Section \ref{sec:OC}, we will always assume that
the costs satisfy the minimal hypotheses made in this paragraph and will not repeat them. We will specify when additional hypotheses are needed.

Let us now recall a general version of Filippov implicit function lemma, which will be useful to prove Theorem~\ref{sec:optim-contr-probl} below. For the proof, we refer the reader to \cite{MW}.
\begin{theorem}\label{sec:optim-contr-probl-1}
Let $I \subset \R $ be an interval and $\gamma: I\to \R^d \times \R^d$ be a measurable function. Let $A$ be a metric space. Let $K$ be a closed subset of $\R^d\times  A$ and $\Psi: K\to  \R^d \times \R^d$ be continuous.
Assume that $\gamma(I)\subset \Psi(K)$, then there is a measurable function $\Phi: I\to K$ such that
\begin{displaymath}
\Psi\circ \Phi(t)=\gamma(t) \quad  \hbox{for a.a. } t\in I.
\end{displaymath}
\end{theorem}
Let us introduce the set
\begin{equation}
\label{eq:1}
M=\left\{(x,a):\; x\in \cG;\quad  a\in A_i  \hbox{ if } x\in J_i\backslash \{O\},  \hbox{ and } a\in A   \hbox{ if } x =O\right \}.
\end{equation}
Note that $M$ is closed. Moreover, since the sets $A_i$ are disjoint, for each $(x,a)\in M$, there exist a unique $i\in\{1,\dots,N\}$ and a unique $\overline a\in\R$ such that $(x,a)=\left(x,(i,\bar a)\right)$. Let the function $f$ be defined on $M$ by 
\begin{displaymath}
  f(x,a)=\left\{
\begin{array}[c]{ll}
f_i(x,a) e_i,\quad &\hbox{ if } x\in J_i\backslash \{O\}, \\
f_i(O,a) e_i, \quad &\hbox{ if } x=O \hbox{ and } a\in A_i,\, i\ne 0, \\
0_{\R^d}, \quad &\hbox{ if } x=O \hbox{ and } a\in A_0,
\end{array}
\right.
\end{displaymath}
for $(x,a)\in M$.
Since the sets $A_i$ are disjoint,  $f $ is continuous on $M$. Let $\widetilde F(x)$ be defined by
\begin{displaymath}
\widetilde F(x)=\left\{ 
\begin{array}{ll}
F_i(x)\quad  &\hbox{if } x \in J_i\backslash\{O\} ,\\
\cup_{i=0}^N F_i(O) \quad  &\hbox{if }x =O.
\end{array}
\right. 
\end{displaymath}
For $x\in \cG$, let the set of admissible paths starting from $x$ be 
\begin{equation}
\label{eq:2}
Y_{x,0}=\left\{ y_x\in W^{1,2}([0,T]; \cG)\;:\;\left|
\begin{array}[c]{ll}
\dot y_x(t)   \in \widetilde F(y_x(t))  ,\quad& \hbox{for a.e. } t\in [0,T],\\
y_x(0)=x.
\end{array}\right.  \right\}.
\end{equation}

\begin{theorem}\label{sec:optim-contr-probl}
If [H0] and [H1] hold, then
\begin{enumerate}
\item For any $x\in \cG$, $Y_{x,0}$ is nonempty
\item For any $x\in \cG$, for any $y_x\in Y_{x,0}$, there exists a measurable function $\Phi:[0,T]\to M$, $\Phi=(\phi_1,\phi_2)$ such that
\begin{displaymath}
\begin{array}{c}
\phi_2=(i,\bar \phi_2),\, \textrm{with }\bar \phi_2\in\R ,\, \textrm{when } \phi_1\in J_i\setminus\{O\} \\(y_x(s),\dot y_x(s))= ( \phi_1(s), f (\phi_1(s),\phi_2(s))) ,\quad \hbox{for a.e. }s,
\end{array}
\end{displaymath}
which means in particular that $y_x$ is a continuous representation of $\phi_1$
\item Almost everywhere in $[0,T]$,
\begin{displaymath}
\dot y_x(s)= \sum_{i=1}^N 1_{\{ y_x(s)\in J_i\backslash \{O\} \}}  \bar \phi_2(s) e_i
\end{displaymath}
\item  Almost everywhere on $\{s: y_x(s)=O\}$, $f(O, \phi_2(s))=0$.
\end{enumerate}
\end{theorem}

\begin{proof}
The proof of point 1 is easy, because $0\in \widetilde F(x)$ for every $x\in\cG$.\\
The proof of point 2 is a consequence of Theorem \ref{sec:optim-contr-probl-1}, with $K=M$, $I=[0,T]$, $\gamma(s)=(y_x(s),\dot y_x(s))$ and 
$\Psi(x,a)=(x,f(x,a))$. \\
Point 2 implies
\begin{displaymath}
\dot y_x(s)= \sum_{i=1}^N \car_{\{ y_x(s)\in J_i\backslash \{O\} \}} \bar \phi_2(s) e_i +  \car_{\{ y_x(s) = O\}}  f( O, \phi_2(s)),
\end{displaymath}
and from Stampacchia's theorem, $f( O, \phi_2(s)) =0$ almost everywhere in $\{s: y_x(s) = O\}$. This yields points 3 and 4.
\end{proof}
\begin{remark}\label{rmk:2.2}
%The measurable function~$\phi_2$ of point (2) of Theorem~\ref{sec:optim-contr-probl} can be written as $\phi_2(s)=(i(s),\bar \phi_2(s))$ for $s\in[0,T]$; moreover, points (3) and (4) mean that $\bar \phi_2$ fulfills
%\begin{equation*}
%\begin{array}{ll}
%\dot y_x(s)= \sum_{i=1}^N 1_{\{ y_x(s)\in J_i\setminus \{O\} \}} \bar \phi_2(s) e_i & \quad \textrm{a.e. }s\in[0,T]\\
%\bar\phi_2(s)=0 &\quad \textrm{a.e. } s\in\{s: y_x(t)=O\}.
%\end{array}
%\end{equation*}
It is worth noticing that in Theorem \ref{sec:optim-contr-probl}, a solution $y_x$ can be associated with several control laws $\phi_2$ which may be different even on sets with positive measure. Actually, for a.e. $s\in\{s\in[0,T]\mid y_x(s)\in J_i\setminus\{O\}\}$, the control~$\phi_2(s)$ is uniquely defined as $\phi_2(s)=\dot y_x(s)$ and belongs to $\R e_i$ (for $i=1,\dots,N$). On the other hand, for a.e. $s\in\{s\in[0,T]\mid y_x(s)=O\}$,  the control~$\phi_2(s)$ is $0$ by Stampacchia theorem, and it can be arbitrarily chosen in any $A_i$, for $i=0,\dots, N$.
\end{remark}

For any $x\in\cG$ and $t_1,t_2\in[0,T]$ with $t_1<t_2$, consider the set of admissible {\it trajectories} (namely, pairs made of controls and paths) on the interval~$[t_1,t_2]$ which start from~$x$ at~$t_1$:
\begin{equation} \label{eq:3}
\Gamma_{t_1,t_2}[x]=\left\{
\begin{array}[c]{ll}
(y_x,\alpha)  \ds \in L^2([t_1,t_2],M): & y_x\in W^{1,2} ([t_1,t_2];\cG), \\
&\ds y_x(s)=x+\int_{t_1}^s f(y_x(\tau),\alpha(\tau)) d\tau \quad  \hbox{in }[t_1,t_2]
\end{array}\right\}.
\end{equation}
For simplicity, when $t_2=T$, we write $\Gamma_{t_1}[x]$ instead of~$\Gamma_{t_1,T}[x]$ and, when $t_2=T$ and $t_1=0$, we drop the subscript: $\Gamma[x]=\Gamma_{0,T}[x]$. 

Finally, the set of all admissible trajectories starting at time $t=0$ is defined as follows:
\begin{equation}
\label{eq:04}
\Gamma=\bigcup\limits_{x\in\cG}\Gamma[x].
\end{equation}

\begin{remark}[concatenation of two admissible trajectories]\label{rmk:concat}
For $0\leq t_1\leq t_2\leq t_3\leq T$ and $x\in\cG$, if $(y_1,\alpha_1)\in\Gamma_{t_1,t_2}[x]$ and $(y_2,\alpha_2)\in\Gamma_{t_2,t_3}[y_1(t_2)]$, the trajectory $(\tilde y,\tilde \alpha)$ defined by
\begin{equation*}
\tilde y(s)=\left\{\begin{array}{ll}
y_1(s)&\quad\textrm{for }s\in[t_1,t_2]\\y_2(s)&\quad\textrm{for }s\in[t_2,t_3]
\end{array}\right.\qquad\textrm{and}\qquad
\tilde \alpha(s)=\left\{\begin{array}{ll}
\alpha_1(s)&\quad\textrm{for }s\in[t_1,t_2]\\\alpha_2(s)&\quad\textrm{for }s\in[t_2,t_3]
\end{array}\right.
\end{equation*}
belongs to~$\Gamma_{t_1,t_3}[x]$.
\end{remark}

\paragraph{The cost functional.}
For $t\in [0,T]$, the cost associated to the trajectory $ (y_x, \alpha)\in \Gamma_t[x]$ is 
\begin{multline}\label{eq:46}
J_t(x;(y_x,\alpha) )=\int_t^T\left[\sum_{i=1}^N \ell_i(y_x(\tau),\tau)\car_{y_x(\tau)\in J_i\setminus\{O\}}+\ell_O(\tau)\car_{y_x(\tau)=O}\right] d\tau\\+\int_t^T\frac{|\alpha(\tau)|^2}{2}\, d\tau+g(y_x(T))
\end{multline}
where 
\begin{eqnarray}\label{eq:460}
\ell_O(\tau)&=&\min\{\ell_*(\tau),\min_{i=1,\dots,N}\ell_i(O,\tau)\}\\ \label{eq:461}
g(y)&=& \sum_{i=1}^N g_i(y)\car_{y\in J_i\setminus\{O\}}+\min\{g_*,\min\limits_{i=1,\dots,N}g_i(O)\}\car_{y=O},
%\left\{\begin{array}{ll}
%g_i(y)&\qquad\textrm{if }y\ne O\\
%\min\{g_0,\min\limits_{i=1,\dots,N}g_i(O)\}&\qquad\textrm{if }y= O.
%\end{array}\right.
\end{eqnarray}
recalling that $g_*$ and $\ell_*$ are introduced in assumption ($H_1$).
For brevity, defining
\begin{equation}\label{eq:462}
L(x,t)=\sum_{i=1}^N \ell_i(x,t)\car_{x\in J_i\setminus\{O\}}+\ell_O(t)\car_{x=O}\qquad \forall (x,t)\in\cG\times [0,T],
\end{equation}
enables one to write
\[
J_t(x;(y_x,\alpha) )=\int_t^T \left(L(y_x(\tau),\tau)+\frac{|\alpha(\tau)|^2}{2}\right)\, d\tau+g(y_x(T)).
\]
\begin{remark}
The arguments below would also apply for  costs of the form
\begin{multline*}
J_t(x;(y_x,\alpha) )=\int_t^T\left[\sum_{i=1}^N \ell_i(y_x(\tau),\tau)\car_{y_x(\tau)\in J_i\setminus\{O\}}+\sum_{i=0}^N\ell_i(O,\tau)\car_{y_x(\tau)=O,\alpha(\tau)\in A_i}\right] d\tau\\+\int_t^T\frac{|\alpha(\tau)|^2}{2}\, d\tau+g(y_x(T)),
\end{multline*}
where we have set $\ell_0(O,\tau)=\ell_*(\tau)$.
\end{remark}
%We emphasize that the abridged notation~$L$ will always denote a cost of the form as in~\eqref{eq:462} and not a general cost depending on $(x,a,t)$.
\paragraph{The value function.}
The value function of the optimal control problem is 
\begin{equation}\label{eq:4}
u(x,t)= \inf_{( y, \alpha)\in \Gamma_t[x]}  J_t(x;( y, \alpha) ).
\end{equation}
Set
\begin{equation}\label{eq:gamma_opt}
\Gamma^{\rm{opt}}_t[x]=\Bigl\{(y, \alpha)\in\Gamma_t[x]\;:\; J_t(x;(y,\alpha))=\min_{(\hat y,\hat \alpha)\in \Gamma_t[x]} J_t(x;(\hat y,\hat \alpha))\Bigr\}.
\end{equation}
For simplicity, we drop the subscript  when $t=0$: $\Gamma^{\rm{opt}}[x]=\Gamma^{\rm{opt}}_t[x]$.

\begin{remark}\label{rmk:bound_contr}
The  value function~$u$ is bounded. Indeed, the trajectory associated to the control $\alpha\equiv 0$ is admissible and provides an upper bound for the value function,  because  the costs $\ell_i$ and  $g$ are bounded. From this, it stems that the optimal controls, if they exist, are uniformly bounded in $L^2(0, T)$.
\end{remark}
\begin{remark}[restriction of optimal trajectories]\label{rmk:restr_OC}
 For $(y,\alpha)\in \Gamma^{\rm{opt}}_t[x]$ and $\bar t \in [t,T]$, $(y_{\mid[\bar t,T]},\alpha_{\mid[\bar t,T]})\in \Gamma^{\rm{opt}}_{\bar t}[y(\bar t)]$. Indeed, assume by contradiction that there exists a trajectory $(\bar y,\bar \alpha)\in \Gamma^{\rm{opt}}_{\bar t}[y(\bar t)]$ such that $J_{\bar t}(y(\bar t);(\bar y,\bar \alpha))<J_{\bar t}(y(\bar t);(y_{\mid[\bar t,T]},\alpha_{\mid[\bar t,T]}))$. Then, by Remark~\ref{rmk:concat}, the concatenation $(\tilde y,\tilde\alpha)$ of~$(y,\alpha)$ with $(\bar y,\bar \alpha)$, defined by
\begin{equation*}
\tilde y(s)=\left\{\begin{array}{ll}
y(s)&\quad \textrm{for }s\in[t,\bar t]\\ \bar y(s)&\quad \textrm{for }s\in[\bar t,T]
\end{array}\right., \quad \hbox{ and} \quad
\tilde \alpha(s)=\left\{\begin{array}{ll}
\alpha(s)&\quad \textrm{for }s\in[t,\bar t]\\ \bar \alpha(s)&\quad \textrm{for }s\in[\bar t,T]
\end{array}\right.
\end{equation*}
belongs to $\Gamma_t[x]$ and consequently there holds
\begin{eqnarray*}
u(x,t)&=&J_t(x; (y,\alpha))=\int_t^{\bar t} \left(L(y(\tau),\tau)+\frac{|\alpha(\tau)|^2}{2}\right)\, d\tau+J_{\bar t}(y(\bar t); (y_{\mid[\bar t,T]},\alpha_{\mid[\bar t,T]}))\\
&>&\int_t^{\bar t} \left(L(y(\tau),\tau)+\frac{|\alpha(\tau)|^2}{2}\right)\, d\tau+J_{\bar t}(y(\bar t); (\bar y,\bar \alpha))=J_t(x; (\tilde y,\tilde \alpha)),
\end{eqnarray*}
which contradicts the optimality of~$(y,\alpha)$.
\end{remark}
\begin{remark}\label{rmk:dpp=}
From Remark~\ref{rmk:restr_OC}, we deduce that for any $(y,\alpha)\in \Gamma^{\rm{opt}}_t[x]$, there holds
\begin{equation*}
u(x,t)=u(y(\bar t),\bar t)+\int_t^{\bar t} \left(L(y(\tau),\tau)+\frac{|\alpha(\tau)|^2}{2}\right)d\tau\qquad \forall \bar t\in[t,T].
\end{equation*}
\end{remark}
\begin{remark}\label{rmk:concat_opt}
The concatenation of two optimal trajectories yields an optimal trajectory. More precisely, for any $(y,\alpha)\in \Gamma^{\rm{opt}}_t[x]$, $\hat t\in(t,T)$ and $(\hat y,\hat \alpha)\in \Gamma^{\rm{opt}}_{\hat t}[y(\hat t)]$, the concatenation~$(y_0,\alpha_0)$ of $(y,\alpha)$ and $(\hat y,\hat \alpha)$ belongs to $\Gamma^{\rm{opt}}_t[x]$.
Indeed, from Remark~\ref{rmk:dpp=}, 
\begin{eqnarray*}
u(x,t)&=&u(y(\hat t),\hat t)+\int_t^{\hat t} \left(L(y(\tau),\tau)+\frac{|\alpha(\tau)|^2}{2}\right)d\tau\\
&=&\int_{\hat t}^T \left(L(\hat y(\tau),\tau)+\frac{|\hat \alpha(\tau)|^2}{2}\right)d\tau+g(\hat y(T))+\int_t^{\hat t} \left(L(y(\tau),\tau)+\frac{|\alpha(\tau)|^2}{2}\right)d\tau\\
&=& J_t(x;(y_0,\alpha_0)),
\end{eqnarray*}
i.e. $(y_0,\alpha_0)$ is optimal for $u(x,t)$.
\end{remark}

\begin{lemma}\label{lemma:u_g}
If [H0] and [H1] hold, then for any $x\in \cG$: $\lim_{t\to T^-} u(x,t)=g(x)$.
\end{lemma}
\begin{proof}
Fix $x\in \cG$. Since $y$ corresponding to control $\alpha=0$ is admissible,
\begin{equation*}
u(x,t)\leq \int_t^T L(y(\tau),\tau)\, d\tau+g(x),
\end{equation*}
and because $L$ is bounded, this implies that $\limsup_{t\to T^-} u(x,t)\leq g(x)$.

On the other hand, for any $\epsilon\in(0,1)$, let $(y_t^\epsilon,\alpha_t^\epsilon)$ be an $\epsilon$-optimal trajectory for $u(x,t)$. The same arguments as in Remark~\ref{rmk:bound_contr} yield that there exists a constant $C$ (independent of~$t$ and of~$\epsilon$) such that: $\|\alpha_t^\epsilon\|_{L^2(t,T)}\leq C$ so, in particular, $y_t^\epsilon(\cdot)$ is $1/2$-H\"older continuous with constant~$C$. Hence,
\begin{equation*}
\liminf_{t\to T^-} u(x,t)\geq \liminf_{t\to T^-}\left(\int_t^T L(y_t^\epsilon(\tau),\tau)\, d\tau+g(y_t^\epsilon(T))\right)-\epsilon=g(x)-\epsilon.
\end{equation*}
Letting $\epsilon$ tend to $0$ yields the desired result.
\end{proof}

\subsection{Existence of optimal trajectories}

\begin{proposition}\label{prp:ex_OT}
For each point $(x,t)\in \cG\times[0,T]$, there exists an optimal trajectory, namely there exists  $(y_x, \alpha)\in \Gamma_t[x]$ such that $u(x,t)=J_t(x;( y_x, \alpha))$. In other words, $\Gamma^{\rm{opt}}[x]\ne \emptyset$.
\end{proposition}
\begin{proof}
Fix $(x,t)\in\cG\times[0,T]$ and consider a minimizing sequence $(y^n, \alpha^n)\in \Gamma_t[x]$, i.e. $u(x,t)=\lim_{n\to\infty} J_t(x;( y^n, \alpha^n) )$.
From Remark~\ref{rmk:bound_contr}, there exists a constant $C$, independent of $n$, such that
\begin{equation}\label{Nic1}
\int_t^T\frac{|\alpha^n(\tau)|^2}{2}\, d\tau\leq C.
\end{equation}
This implies that $y^n$ are uniformly bounded and uniformly $1/2$-H\"older continuous, because
\begin{equation}\label{Nic2}
y^n(s)=x+\int_t^s \sum_{i=1}^N \car_{\{ y^n(\tau)\in J_i\backslash \{O\} \}} \bar \alpha^n(\tau) e_i\, d\tau.
\end{equation}
%From \eqref{Nic2}
%\begin{equation*}%\label{Nic2}
%y^n(t_1) - y^n(t_2)=\int_{t_1}^{t_2} \sum_{i=1}^N \car_{\{ y^n(\tau)\in J_i\backslash \{O\} \}}  \bar\alpha^n(\tau) e_i\, d\tau
%\end{equation*}
%and, in particular

There exist  $\alpha\in L^2([t,T];\R^d)$ and $y_x\in C^{1/2}([t,T];\R^d)$ such that, possibly up to the  extraction of subsequences,  $\sum_{i=1}^N \car_{\{ y^n(\cdot)\in J_i\backslash \{O\} \}} \bar \alpha^n(\cdot) e_i$ converge to $\alpha$ in the weak topology of $L^2([t,T];\R^d)$ and $y_n$ uniformly converge to $y_x$. In particular, letting $n\to\infty$ in \eqref{Nic2} yields
\begin{equation}\label{Nic3}
y_x\in W^{1,2}([t,T];\R^d), \quad\textrm{with }\quad y_x(s)=x+\int_t^s \alpha(\tau)\, d\tau\in \cG,
\end{equation}
because $\cG$ is closed.
We now claim that
\begin{equation}\label{claim1}
(y_x, \alpha)\in \Gamma_t[x].
\end{equation}
To obtain \eqref{claim1}, it suffices to prove that $\alpha$ is an admissible control, i.e. that $(y_x(s),\alpha(s))\in M$ for a.e. $s\in(t,T)$. To this end, let us argue differently whether $y_x(s)$ coincides or not with~$O$.\\
Consider $s\in(t,T)$ such that $y_x(s)\in J_i\setminus\{O\}$ for some $i=1,\dots,N$.
Since the $y^n$ are uniformly $1/2$-Holder continuous and uniformly converge to~$y_x$, we deduce that, for $\varepsilon>0$ sufficiently small and for any~$n$ sufficiently large, there holds
\[
y^n(\tau)\in J_i\setminus\{O\}\qquad \forall \tau\in(s-\varepsilon,s+\varepsilon).
\]
In particular, for $n$ sufficiently large, $\alpha^n(\tau)=\bar \alpha^n(\tau) e_i$ for $\tau\in(s-\varepsilon,s+\varepsilon)$. Letting $n\to\infty$, we conclude that $\alpha(\tau)$ is aligned with $e_i$ for $\tau\in(s-\varepsilon,s+\varepsilon)$.\\
Define the compact set
\begin{equation}\label{eq:defE}
    E=\{s\in(t,T):\, y_x(s)=O\}.
\end{equation} 
From \eqref{Nic3}, Stampacchia's theorem yields that $\alpha(s)=0$ for a.a. $s\in E$. \\
Hence,  we may write for instance $\alpha=0e_1$ in $E$. The claim~\eqref{claim1} is proved.

Let us now check that $(y_x,\alpha)$ is an optimal trajectory, i.e. that 
\begin{equation}\label{claim2}
u(x,t)=J_t(x;( y_x, \alpha)).
\end{equation}
In order to prove \eqref{claim2}, it is useful to decompose $u(x,t)$ as follows
\begin{eqnarray}
u(x,t)
%&=&\lim_{n\to\infty} \left[\int_t^T\left[\sum_{i=1}^N \ell_i(y^n(\tau),\tau)\car_{y^n(\tau)\in J_i\setminus\{O\}}+\ell_O\car_{y^n(\tau)=O}\right] d\tau\right.\\ \notag &&\left.\qquad+\int_t^T\frac{|\alpha^n(\tau)|^2}{2}\, d\tau+g(y^n(T))\right]\notag\\
&=& \lim_{n\to\infty}\left[\int_t^T\frac{|\alpha^n(\tau)|^2}{2} d\tau +\sum_{i=1}^5 I_i\right], \label{eq:Jyn}
\end{eqnarray}
where
\begin{eqnarray*}
I_1&=& \int_t^T\sum_{i=1}^N \ell_i(y^n(\tau),\tau)\car_{y^n(\tau)\in J_i\setminus\{O\}} \car_{y_x(\tau)\in J_i\setminus\{O\}} d\tau,\\
I_2&=& \int_t^T\sum_{i=1}^N \ell_i(y^n(\tau),\tau)\car_{y^n(\tau)\in J_i\setminus\{O\}} \car_{y_x(\tau)\in \cG\setminus J_i} d\tau,\\
I_3&=& \int_t^T\sum_{i=1}^N \ell_i(y^n(\tau),\tau)\car_{y^n(\tau)\in J_i\setminus\{O\}} \car_{y_x(\tau)=O} d\tau,\\
I_4&=&\int_t^T\ell_O(\tau)\car_{y^n(\tau)=O} d\tau,\\
I_5&=&g(y^n(T)),
\end{eqnarray*}
and study separately the different contributions in the right hand side of~\eqref{eq:Jyn}.
It is well known that the convergence in the weak topology of $L^2([t,T];\R^d)$ entails
\begin{equation}\label{part1}
\int_t^T\frac{|\alpha(\tau)|^2}{2}d\tau\leq \liminf_{n\to\infty} \int_t^T\frac{|\alpha^n(\tau)|^2}{2}d\tau.
\end{equation}
Concerning $I_1$, the uniform convergence of $y^n$ to $y_x$ as $n\to\infty$ and the continuity of $\ell_i$ ensure that, for any $\tau\in [t,T]$,
\[
\ell_i(y^n(\tau),\tau)\car_{y^n(\tau)\in J_i\setminus\{O\}} \car_{y_x(\tau)\in J_i\setminus\{O\}}\to \ell_i(y_x(\tau),\tau)\car_{y_x(\tau)\in J_i\setminus\{O\}}\qquad \textrm{as }n\to\infty.
\]
Since the $\ell_i$'s are bounded, the dominated convergence theorem yields
\begin{equation}\label{part2_1}
I_1\to \int_t^T\sum_{i=1}^N \ell_i(y_x(\tau),\tau)\car_{y_x(\tau)\in J_i\setminus\{O\}} d\tau\qquad \textrm{as }n\to\infty.
\end{equation}
As for $I_2$, again the uniform convergence of $y^n$ to $y_x$ and the continuity of $\ell_i$ ensure that the integrand tends to zero as $n\to \infty$. Again the dominated convergence theorem yields 
\begin{equation}\label{part2_2}
I_2\to 0\qquad\textrm{as }n\to\infty.
\end{equation}
Let us now consider the term $I_5$ and  argue differently whether $y_x(T)$ coincides or not with $O$. If $y_x(T)\in J_i\setminus\{O\}$ for some $i\in\{1,\dots,N\}$ then, the uniform convergence of $y^n$ to $y_x$ and the continuity of~$g_i$ entail $g(y^n(T))=g_i(y^n(T))\to g_i(y_x(T))=g(y_x(T))$ as $n\to\infty$. If $y_x(T)=O$, again by the uniform convergence of $y^n$ to $y_x$ and by the definition of~$g$ in~\eqref{eq:461}, for any $\varepsilon>0$, we get $g(y^n(T))\geq g(O)-\varepsilon=g(y_x(T))-\varepsilon$ for $n$ sufficiently large. In both cases, 
\begin{equation}\label{part5}
\liminf_{n\to\infty}I_5\geq  g(y_x(T)).
\end{equation}
On the other hand, 
\begin{eqnarray*}
I_3+I_4&=&\int_t^T\left[\sum_{i=1}^N \ell_i(y^n(\tau),\tau)\car_{y^n(\tau)\in J_i\setminus\{O\}} +\ell_O(\tau)\car_{y^n(\tau)=O}\right]\car_{y_x(\tau)=O} d\tau\\
&&+\int_t^T\ell_O(\tau)\car_{y^n(\tau)=O}\car_{y_x(\tau)\ne O} d\tau.
\end{eqnarray*}
Observe that $\car_{y^n(\cdot)=O}\car_{y_x(\cdot)\ne O}\to 0$  as $n\to\infty$. Hence, from the dominated convergence theorem, 
\begin{equation*}
\int_t^T\ell_O(\tau)\car_{y^n(\tau)=O}\car_{y_x(\tau)\ne O} d\tau \to 0 \quad \textrm{as }n\to\infty.
\end{equation*}
Assume for a while that
\begin{multline}\label{claim_fin}
\liminf_{n\to\infty}\int_t^T\left[\sum_{i=1}^N \ell_i(y^n(\tau),\tau)\car_{y^n(\tau)\in J_i\setminus\{O\}} + \ell_O(\tau)\car_{y^n(\tau)=O}\right]\car_{y_x(\tau)=O} d\tau\\ \geq \int_t^T \ell_O(\tau)\car_{y_x(\tau)=O} d\tau.
\end{multline}
%where $\bar \ell_O(\tau)=\min_{i=0,\dots,N}\ell_i(O,\tau)$.
From \eqref{eq:Jyn}, \eqref{part1} and \eqref{claim_fin},
\begin{equation*}
u(x,t)\geq \int_t^T\left[\frac{|\alpha(\tau)|^2}{2}+ \sum_{i=1}^N \ell_i(y_x(\tau),\tau)\car_{y_x(\tau)\in J_i\setminus\{O\}}+ \ell_O(\tau)\car_{y_x(\tau)=O}\right] d\tau+g(y_x(T))
\end{equation*}
%Let $i_0\in\{0,\dots,N\}$ be such that $\ell_{i_0}(O,\tau)=\bar \ell_O(\tau)$ (namely, it minimizes the cost in $O$). By Stampacchia theorem (using again~\eqref{Nic3}), we have $\alpha(\tau)=0$ for a.e. $\tau\in E$ (recall that  $E=\{\tau\in[t,T]:\,y_x(\tau)=O\}$); in particular, without changing $y_x$ we can assume: $\alpha(\tau)\in A_{i_0}$  for a.e. $\tau\in E$. Therefore last relation reads
%\[
%u(x,t)\geq J_t(x;( y_x, \alpha))
%\]
which is equivalent to \eqref{claim2}.

There remains to prove \eqref{claim_fin}.  Recall that the set $E$ has been defined in \eqref{eq:defE}.  Since $y^n$ uniformly converge to $y_x$, $\|y^n-y_x\|_{L^\infty(E)}$ is arbitrary  small for $n$ sufficiently large. Then the continuity of $\ell_i$ implies that for any $\varepsilon>0$, 
\[
\ell_i(y^n(\tau),\tau)> \ell_O(\tau)-\varepsilon, \quad \quad\forall \tau\in E,
\]
for $n$ sufficiently large, 
which implies inequality~\eqref{claim_fin}.
\end{proof}
\begin{remark}\label{rmk:quasi_dpp}
The following statement can be seen as the ``converse'' of Remark~\ref{rmk:dpp=}. If there exist $t_1\in[0,T]$ and $(y_1,\alpha_1)\in \Gamma_{t,t_1}[x]$ such that
\begin{equation*}
u(x,t)=u(y_1(t_1),t_1)+\int_t^{t_1} \left(L( y_1(\tau),\tau)+\frac{| \alpha_1(\tau)|^2}{2}\right)\, d\tau,
\end{equation*}
then, there exists $(y,\alpha)\in \Gamma^{\rm{opt}}_t[x]$ with  $(y_1,\alpha_1)=(y,\alpha)$ on $(t,t_1)$. Indeed, consider the concatenation $(y,\alpha)$ of $(y_1,\alpha_1)$ with $(\bar y,\bar \alpha)$ where $(\bar y,\bar \alpha)$ is any trajectory in $\Gamma^{\rm{opt}}_{t_1}[y_1(t_1)]$. from Remark~\ref{rmk:concat}, $(y,\alpha)$ is admissible for $u(x,t)$. Since the above equality can be written 
\begin{eqnarray*}
u(x,t)&=&\int_{t_1}^T \left(L(\bar y(\tau),\tau)+\frac{|\bar \alpha(\tau)|^2}{2}\right)\, d\tau+g(\bar y(T))
+\int_t^{t_1} \left(L( y_1(\tau),\tau)+\frac{| \alpha_1(\tau)|^2}{2}\right)\, d\tau\\
&=&J_t(x;(y,\alpha)),
\end{eqnarray*}
 $(y,\alpha)$ is optimal for $u(x,t)$.
\end{remark}

%\begin{remark}\label{rmk:ex_OT}
%Proposition~\ref{prp:ex_OT} ensures: $\Gamma^{\rm{opt}}[x]\ne \emptyset$ for any $x\in\cG$.
%\end{remark}

%
%closed graph property
%

\subsection{First properties}
This paragraph is devoted to the dynamic programming principle and the continuity of~$u$.  Let us stress that the structure of the control set 
%(in particular, the fact that $\alpha=0$ belongs to every~$A_i$)
plays a crucial role in what follows.
\begin{proposition}[Dynamic programming principle]\label{prp:prop_vf}
Assume [H0] and [H1]. For any $(x,t)\in\cG\times[0,T]$ and $\bar t\in[t,T]$, there holds 
\begin{equation}\label{eq:DPP}
u(x,t)= \inf_{(y,\alpha)\in \Gamma_{t,\bar t}[x]}\left\{
u(y(\bar t),\bar t)+\int_t^{\bar t} \left(L(y(\tau),\tau)+\frac{|\alpha(\tau)|^2}{2}\right)d\tau
\right\}.
\end{equation}
\end{proposition}
\begin{proof}
$(i)$. For any $(y,\alpha)\in \Gamma_{t}[x]$, there holds
\begin{eqnarray*}
J_t(x,(y,\alpha))&=&\int_t^{\bar t}\left(L(y(\tau),\tau)+\frac{|\alpha(\tau)|^2}{2}\right)d\tau +J_{\bar t}(y(\bar t),(y_{\mid [\bar t,T]},\alpha_{\mid [\bar t,T]}))\\
&\geq &  
\int_t^{\bar t}\left(L(y(\tau),\tau)+\frac{|\alpha(\tau)|^2}{2}\right)d\tau +u(y(\bar t),\bar t)\\
&\geq & \inf_{(z,\beta)\in \Gamma_{t,\bar t}[x]}\left\{
u(z(\bar t),\bar t)+\int_t^{\bar t} \left(L(z(\tau),\tau)+\frac{|\beta(\tau)|^2}{2}\right)d\tau
\right\},
\end{eqnarray*}
where $(y_{\mid [\bar t,T]},\alpha_{\mid [\bar t,T]})$ is the restriction of the trajectory $(y,\alpha)$ in the interval~$[\bar t,T]$. Taking the infimum in $(y,\alpha)\in \Gamma_{t}[x]$ leads to \eqref{eq:DPP} with the $\ge$ sign  instead of $=$.

Let us now prove the reverse inequality. Consider $(y,\alpha)\in \Gamma_{t,\bar t}[x]$. For $(\bar y,\bar \alpha)\in\Gamma^{\rm{opt}}_{\bar t}[y(\bar t)]$ (whose existence is ensured by Proposition~\ref{prp:ex_OT}),
\begin{equation*}
u(y(\bar t),\bar t)+\int_t^{\bar t}\left(L(y(\tau),\tau)+\frac{|\alpha(\tau)|^2}{2}\right)d\tau =
J_t(x,(\tilde y,\tilde \alpha))\geq u(x,t)  ,
\end{equation*}
where $(\tilde y,\tilde \alpha)$ is the concatenation of the trajectory $(y,\alpha)$ on $[t,\bar t]$ and of the trajectory $(\bar y,\bar \alpha)$ on $[\bar t,T]$. Recall from Remark~\ref{rmk:concat} that $(\tilde y,\tilde \alpha)\in \Gamma_t[x]$.
The proof is completed by taking the infimum in $(y,\alpha)\in \Gamma_{t,\bar t}[x]$.
%(\bar y,\bar \alpha)\in\Gamma_{\bar t}[y(\bar t)]$, we achieve the proof.
\end{proof}
\begin{proposition}[Continuity of the value function]\label{prp:U_cont} If [H0] and [H1] hold, then the function~$u$ is continuous in~$\cG\times[0,T)$.
\end{proposition}
\begin{proof}
Consider $(x_1,t_1), (x_2,t_2)\in \cG\times[0,T_1)$, with $T_1<T$ and $\delta:=d(x_1,x_2)<(T-T_1)/2$. Without any loss of generality, we may assume that both~$x_1$ and~$x_2$ belong to the same edge, say~$e_1$. Consider $(y_2,\alpha_2)\in\Gamma^{\rm{opt}}_{t_2}[x_2]$. Consider the trajectory that starts in~$x_1$ at time~$t_1$ and corresponds to the control
%{\color{red} New
%\begin{equation*}    
%\alpha_1(s)=
%\left\{\begin{array}{ll}
%\frac {x_2-x_1}{\delta}
%&\quad \textrm{for }s\in[t_1,t_1+\delta],\\
%\frac{T-t_2}{T-t_1-\d}\alpha_2 \left(\frac{T-t_2}{T-t_1-\d}s-%\frac{T(\delta+t_1-t_2)}{T-t_1-\d}\right)&\quad \textrm{for }s\in(t_1+\delta,T],\end{array}\right.
%\end{equation*}
%}
%{\color{red} OLD}
\begin{equation*}    
\alpha_1(s)=
\left\{\begin{array}{ll}
\left\{\begin{array}{ll}
e_1&\quad\textrm{if } x_1\leq x_2\\
-e_1&\quad\textrm{if } x_1> x_2\\
\end{array}
\right\}&\quad \textrm{for }s\in[t_1,t_1+\delta],\\
\frac{T-t_2}{T-t_1-\d}\alpha_2 \left(\frac{T-t_2}{T-t_1-\d}s-\frac{T(\delta+t_1-t_2)}{T-t_1-\d}\right)&\quad \textrm{for }s\in(t_1+\delta,T],\end{array}\right.
\end{equation*}
thus
\begin{equation*}
y_1(s)=\left\{\begin{array}{ll}
x_1+\frac{x_2-x_1}{\d}(s-t_1) &\quad\textrm{for }s\in[t_1,t_1+\d],\\
%x_2&\quad\textrm{for }s=t_1+\d,\\
x_2+\int_{t_1+\d}^s\frac{T-t_2}{T-t_1-\d}\alpha_2 \left(\frac{T-t_2}{T-t_1-\d}\tau-\frac{T(\delta+t_1-t_2)}{T-t_1-\d}\right)\, d\tau &\quad\textrm{for }s\in[t_1+\d,T].
\end{array}\right.
\end{equation*}
Observe that, for $s\in[t_1+\d,T]$, there holds
\begin{equation*}
y_1(s)=x_2+\int_{t_2}^{\frac{T-t_2}{T-t_1-\d}s-\frac{T(\delta+t_1-t_2)}{T-t_1-\d}}\alpha_2 (\theta))\, d\theta = y_2\left(\frac{T-t_2}{T-t_1-\d}s-\frac{T(\delta+t_1-t_2)}{T-t_1-\d}\right),
\end{equation*}
and  that $(y_1,\alpha_1)\in \Gamma_{t_1}(x_1)$ with $y_1(T)=y_2(T)$. On the other hand,
\begin{eqnarray}\notag
\|\alpha_1\|^2_{L^2(t_1,T)}&=&  %{\color{red} \frac {(x_2-x_1)^2} {\d}}
\delta+\int_{t_1+\d}^T\left(\frac{T-t_2}{T-t_1-\d}\right)^2\left|\alpha_2 \left(\frac{T-t_2}{T-t_1-\d}s-\frac{T(\delta+t_1-t_2)}{T-t_1-\d}\right)\right|^2\, d\tau\\ \notag
&=&%{\color{red} \frac {(x_2-x_1)^2} {\d}}
\delta +\frac{T-t_2}{T-t_1-\d} \|\alpha_2\|^2_{L^2(t_2,T)}\\\label{eq:14giu_2}
&=&%{\color{red} \frac {(x_2-x_1)^2} {\d}}
\delta+\|\alpha_2\|^2_{L^2(t_2,T)}+\frac{t_1-t_2+\d}{T-t_1-\d} \|\alpha_2\|^2_{L^2(t_2,T)}.
\end{eqnarray}
Let us estimate 
\begin{equation*}
J_{t_1}(x_1;(y_1,\alpha_1))-J_{t_2}(x_2;(y_2,\alpha_2))=\sum_{i=1}^4 I_i,
\end{equation*}
where
\begin{equation*}
\begin{array}{ll}
\ds I_1=\int_{t_1}^{t_1+\d}L(y_1(\tau),\tau)\,  d\tau ,&\quad\quad
\ds I_2= \frac{\|\alpha_1\|^2_{L^2(t_1,T)}-\|\alpha_2\|^2_{L^2(t_2,T)}}{2},\\
\ds I_3=\int_{t_1+\d}^T L(y_1(\tau),\tau)d\tau ,&\quad\quad
\ds I_4=-\int_{t_2}^{T} L(y_2(\tau),\tau)\, d\tau,
\end{array}
\end{equation*}
recalling that $y_1(T)=y_2(T)$.
From the boundedness of the running cost and~\eqref{eq:14giu_2}, there holds $|I_1|\leq K\d$  for some constant~$K$, and
\[
|I_2|\leq \frac{\d}2 +\frac{\|\alpha_2\|^2_{L^2(t_2,T)}}2\frac{|t_1-t_2|+\d}{T-T_1}.
\]
On the other hand, after a change of variable, 
\begin{equation*}
\begin{array}{ll}
I_3&=\ds \int_{t_1+\d}^T L\left(y_2\left(\frac{T-t_2}{T-t_1-\d}\tau-\frac{T(\delta+t_1-t_2)}{T-t_1-\d}\right), \tau\right)\, d\tau\\
&\ds =\frac{T-t_1-\d}{T-t_2}\int_{t_2}^T L\left(y_2(\theta), \frac{T-t_1-\d}{T-t_2}\theta+\frac{T(\delta+t_1-t_2)}{T-t_2}\right)\, d\theta,
%&{\color{red} \xout{???? =\frac{T-t_1-\d}{T-t_2}\int_{t_2}^T L\left(y_2(\theta), %\frac{T-t_1-\d}{T-t_2}\theta+\frac{T(\delta+t_1-t_2)}{T-t_2}\right)\, d\theta}}
\end{array}
\end{equation*}
%\color{red}{\xout{?????where the last equality is due to the same arguments used in the proof of Lemma~\ref{lemma1}.}}
 which implies that
\begin{equation*}
\begin{array}{ll}
|I_3+I_4|&=\ds \left|
\frac{t_2-t_1-\d}{T-t_2}\int_{t_2}^T L\left(y_2(\theta), \frac{T-t_1-\d}{T-t_2}\theta+\frac{T(\delta+t_1-t_2)}{T-t_2}\right)\, d\theta\right.\\
&\ds \qquad\left.
+\int_{t_2}^T \left[ L\left(y_2(\theta), \frac{T-t_1-\d}{T-t_2}\theta+\frac{T(\delta+t_1-t_2)}{T-t_2}\right)-L(y_2(\theta),\theta)\right]\, d\theta
\right|\\
&\ds \leq 
\frac{|t_2-t_1|+\d}{T-T_1}K+T\o\left(\frac{|t_2-t_1|+\d}{T-T_1}2T \right),
\end{array}
\end{equation*}
 where~$\o$ is a  modulus of continuity common to all the costs $\ell_i$ in $B(0,R)$ with $R>|x_1|+|x_2| +\max_{s\in[0,T]}|y_2(s)|$. In conclusion, 
\[
|J_{t_1}(x_1;(y_1,\alpha_1))-J_{t_2}(x_2;(y_2,\alpha_2))|\leq
\tilde K(|t_2-t_1|+\d)+ \tilde \o(|t_2-t_1|+\d)
\]
for a suitable constant~$\tilde K$ (depending only on $T_1$) and a suitable modulus of continuity~$\tilde \o$ (depending on $T_1$, $|x_1|$ and $|x_2|$). From the optimality of~$(y_2,\alpha_2)$, 
\[
\begin{array}{rcl}
u(x_1,t_1)&\leq& J_{t_1}(x_1;(y_1,\alpha_1))\leq J_{t_2}(x_2;(y_2,\alpha_2)) +\tilde K(|t_2-t_1|+\d)+ \tilde \o(|t_2-t_1|+\d)\\
&\leq& u(x_2,t_2)+\tilde K(|t_2-t_1|+\d)+ \tilde \o(|t_2-t_1|+\d).
\end{array}
\]
Reversing the role of~$(x_1,t_1)$ and~$(x_2,t_2)$, we get
\[
|u(x_1,t_1)-u(x_2,t_2)|\leq \tilde K(|t_2-t_1|+\d)+ \tilde \o(|t_2-t_1|+\d),
\]
and the proof is done.
\end{proof}
\begin{remark}[H\"older/Lipschitz continuity]\label{rmk:U_holder}
If the running costs $\ell_i$ are $\theta$-H\"older continuous with respect to time for $\theta =(0, 1]$, the same arguments as above
can be used for proving  that the value function is locally $\theta$-H\"older continuous with respect to $(x,t)$ in $\cG\times [0,T)$.
\end{remark}
The following property will not be used in the remaining part  of the paper.
%QUESTO LEMMA NON SERVE PIU' PERO' POTREBBE TORNARE UTILE NEL FUTURO. PER IL MOMENTO NON LO CANCELLO.
\begin{lemma}\label{lemma1bis}
Fix $(x,t)\in\cG\times [0,T)$ and $(y,\alpha)\in\Gamma_t[x]$ and  consider a sequence $\{(x_n,t_n)\}_{n\in\N}$, such that $(x_n,t_n)\in\cG\times[0,T)$, $\d'_n=d(x_n,x)+|t_n-t|\to 0$ as $n\to \infty$.
There exists a sequence $\{(y_n,\alpha_n)\}_{n\in\N}$, such that $(y_n,\alpha_n)\in\Gamma_{t_n}[x_n]$
\begin{equation}\label{eq:lemma1bis}
\begin{array}{rl}
(i)&\quad \sup\limits_{[t_n\vee t,T]}d(y_n(\cdot),y(\cdot))\leq \d_n+|t_n-t|+\|\alpha\|_2\sqrt{\d'_n},\qquad y_n(T)=y(T)\\
(ii)&\quad\|\alpha_n\|_2^2\leq \|\alpha\|_2^2+\d'_n\left(1+\frac{\|\alpha\|_2^2}{T-\d'_n}\right)\\
(iii)&\quad \lim\limits_{n\to\infty}J_{t_n}(x_n;(y_n,\alpha_n)) = J_t(x;(y,\alpha)).
\end{array}
\end{equation}
\end{lemma}
\begin{proof}[Proof of Lemma~\ref{lemma1bis}]
We adapt the arguments in the proof of Lemma~\ref{lemma1}. It is enough to focus on the situation in which all the points $x_n$ and $x$ belong to the edge $e_1$,  and  all the $t_n$ are either smaller or larger than $t$. Set $\d_n=d(x_n,x)$.

Case 1: $t_n\leq t$, $\forall n\in\N$. Let us introduce the control
\begin{equation*}
\alpha_n(s)=
\left\{\begin{array}{ll}
0,&\quad \textrm{for }s\in[t_n,t],\\
\left\{\begin{array}{ll}
e_1,&\quad\textrm{if }\bar x_n\leq \bar x,\\
-e_1,&\quad\textrm{if }\bar x_n> \bar x,\\
\end{array}
\right.&\quad \textrm{for }s\in[t,t+\delta_n],\\
\frac{T-t}{T-t-\d_n}\alpha \left(s\frac{T-t}{T-t-\d_n}-\d_n\frac{T}{T-t-\d_n}\right),&\quad \textrm{for }s\in(t+\delta_n,T],\end{array}\right.
\end{equation*}
and let $y_n$ be the corresponding path starting from $x_n$ at time $t_n$. Clearly, $(y_n,\alpha_n)\in\Gamma_{t_n}[x_n]$, and
\begin{equation*}
y_n(s)=\left\{\begin{array}{ll}
x_n,&\qquad\textrm{for }s\in[t_n,t],\\
x_n +(x-x_n)\d_n^{-1}(s-t),&\qquad\textrm{for }s\in[t,t+\d_n],\\
y\left(s\frac{T-t}{T-t-\d_n}-\d_n\frac{T}{T-t-\d_n}\right),&\qquad\textrm{for }s\in[t+\d_n,T].
\end{array}\right.
\end{equation*}
 The bounds in~\eqref{eq:lemma1bis}-$(i)$ and~$(ii)$ are obtained with the same arguments as above. Moreover, 
\[
J_{t_n}(x_n;(y_n,\alpha_n))-J_{t}(x;(y,\alpha))=\sum_{i=1}^5I_i,
\]
where, for $i=1,\dots,4$, the terms $I_i$ are analogous to the corresponding ones in~\eqref{eq:rescal4}, while $I_5=\int_{t_n}^tL(x_n,0,\tau)\, d\tau$. Then $|I_5|\leq K|t-t_n|$ for a suitable constant $K$, since  the costs $\ell_i$ are bounded functions. The same calculations as in the proof of Lemma~\ref{lemma1} lead to the desired result.

Case 2: $t_n\geq t$, $\forall n\in\N$. It is clear that $t+\d'_n=t_n+\d_n$. Consider the control
\begin{equation*}
\alpha_n(s)=
\left\{\begin{array}{ll}
\left\{\begin{array}{ll}
e_1,&\quad\textrm{if }\bar x_n\leq \bar x,\\
-e_1,&\quad\textrm{if }\bar x_n> \bar x,\\
\end{array}
\right. &\quad \textrm{for }s\in[t_n,t+\delta'_n],\\
\frac{T-t}{T-t-\d'_n}\alpha \left(s\frac{T-t}{T-t-\d'_n}-\d'_n\frac{T}{T-t-\d'_n}\right),&\quad \textrm{for }s\in(t+\delta'_n,T],\end{array}\right.
\end{equation*}
and let $y_n$ be the corresponding path starting from $x_n$ at time $t_n$. Then $(y_n,\alpha_n)\in\Gamma_{t_n}[x_n]$ and
\begin{equation*}
y_n(s)=\left\{\begin{array}{ll}
x_n +(x-x_n)\d_n'^{-1}(s-t),&\qquad\textrm{for }s\in[t,t+\d'_n],\\
y\left(s\frac{T-t}{T-t-\d_n}-\d_n\frac{T}{T-t-\d_n}\right),&\qquad\textrm{for }s\in[t+\d'_n,T].
\end{array}\right.
\end{equation*}
The desired result is obtained with the same calculations as in the proof of Lemma~\ref{lemma1}.
\end{proof}

\subsection{Euler-Lagrange  conditions}
\label{sec:euler_lagrange}
Below, we address situations in which it is possible to write  the Euler-Lagrange conditions for an optimal trajectory.
They will consist of a family of  differential equations along with a  condition at the horizon.
The following lemma deals with the Euler-Lagrange condition 
in time intervals $[t_1,t_2]\subset (0,T)$ for which an optimal trajectory lies in the interior of a given edge. %In particular, we show that the related control $\alpha$ is Lipschitz continuous in $[t_1,t_2]$.

%
%Lemma: Euler-Lagrange
%

\begin{lemma}\label{lemma:EL}
Consider $i\in\{1,\dots,N\}$, and 
assume that the function $\ell_i$ is differentiable with respect to its first argument with $\partial_x\ell_i\in C(J_i\times [0,T])$. % with $\sup_{s\in[t,T]}\|\ell_i(\cdot,s)\|_{C^2}\leq K$.
Consider any $(x,t)\in \cG\times[0,T]$ and any $(y,\alpha)\in \Gamma^{\rm{opt}}_t[x]$ such that, for some $t_1,t_2\in(t,T)$, there holds
\begin{equation*}
y(s)\in J_i\setminus \{O\}, \qquad \forall s\in [t_1,t_2].
\end{equation*}
Then, the control $\alpha$ is $C^1$ in $(t_1,t_2)$ and 
\begin{equation}\label{EL_giugno}
\alpha'(s)= \partial_x\ell_i(y(s),s),\qquad \forall s\in (t_1,t_2).
\end{equation}
\end{lemma}
\begin{proof}
Fix $\tilde t\in(t_1,t_2)$ and %let $\delta=\min\{\tilde t-t_1,T-\tilde t\}$ and
consider $\alpha_1\in L^2(t,T)$, with $\alpha_1(s)\in \R e_i$ a.e. in $(t_1,t_2)$, $\alpha_1(s)=0$ for $s\notin (t_1,t_2)$  and $\int_{t_1}^{t_2}\alpha_1 ds=0$. 
In $(t_1,t_2)$, both $\alpha$ and $\alpha_1$ are aligned with $e_i$ and can be written  $\alpha(s)=\bar \alpha(s)e_i$ and $\alpha_1(s)=\bar \alpha_1(s)e_i$ with $\bar \alpha(s), \bar \alpha_1(s)\in \R$. 
For $h\in\R$, with $|h|$ sufficiently small, the control $\alpha_h(\cdot):=\alpha(\cdot)+h\alpha_1(\cdot)$ is admissible for $(x,t)$ because  $ \left|y|_{t_1,t_2]}\right|$ is bounded from below by a positive number. 
Let $y_h$ denote the trajectory corresponding to the control $\alpha_h$. It is clear that $y_h(T)=y(T)$. Then,  since   $(y,\alpha)$ is optimal,
\begin{equation}\label{10giu_1}
0\leq \frac{J_t(y_h,\alpha_h)-J_t(y,\alpha)}{h}
=\int_{t_1}^{t_2}\left(\bar \alpha(s)\bar \alpha_1(s)+\frac{h\bar \alpha_1(s)^2}2+ \frac{\ell_i(y_h(s),s)-\ell_i(y(s),s)}{h}\right)ds.
\end{equation}
Since  $y_h(s)=y(s)+h\int_{t_1}^s\alpha_1(\tau)d\tau$ for $s\in[t_1,t_2]$, we deduce from the regularity of~$\ell_i$ with respect to the state variable that
\begin{equation*}
\int_{t_1}^{t_2}\frac{\ell_i(y_h(s),s)-\ell_i(y(s),s)}{h}\, ds=
\int_{t_1}^{t_2} \partial_x\ell_i(y(s),s) \int_{t_1}^s\bar\alpha_1(\tau)d\tau ds +o(1),
\end{equation*}
where $o(1)$ is a function of $h$ that tends to $0$ as $h\to 0$. Integrating by parts the last integral  and observing that
$\int_{t_1}^{t_2}\bar \alpha_1 ds=0$ yields
\begin{equation*}
\int_{t_1}^{t_2}\frac{\ell_i(y_h(s),s)-\ell_i(y(s),s)}{h}\, ds=
- \int_{t_1}^{t_2}\left(\int_{t_1}^{s}\partial_x\ell_i(y(\theta),\theta)d\theta\right) \bar\alpha_1(s)ds.
\end{equation*}
Inserting the latter in~\eqref{10giu_1} and letting $h\to 0$ leads to
\begin{equation*}
0\leq \int_{t_1}^{t_2}\left[\bar \alpha(s)-\int_{t_1}^{s}\partial_x\ell_i(y(\theta),\theta)d\theta \right]\bar \alpha_1(s) ds,
\end{equation*}
for every $\alpha_1$ supported in $[t_1,t_2]$ with $\int_{t_1}^{t_2}\alpha_1 ds=0$. The linearity of the constraint then implies
\begin{equation*}
0=\int_{t_1}^{t_2}\left[\bar \alpha(s)-\int_{t_1}^{s}\partial_x\ell_i(y(\theta),\theta)d\theta \right]\bar \alpha_1(s) ds,
\end{equation*}
i.e. that $s\mapsto \bar \alpha(s)-\int_{t_1}^{s}\partial_x\ell_i(y(\theta),\theta)d\theta$ is orthogonal in $L^2(t_1,t_2) $ to
$V=\{f\in L^2(t_1,t_2):\; \int_{t_1}^{t_2}f=0\} = \R^{\perp_{L^2(t_1,t_2)}}$  . Hence, this function is constant and \eqref{EL_giugno} is proved.
%\begin{equation*}
%\alpha'(s)= \partial_x\ell_i(y(s),s)\qquad \forall s\in (t_1,t_2)
%\end{equation*}
%which is equivalent to relation~\eqref{EL_giugno}.
\end{proof}
\begin{remark}\label{rem:EL}
A consequence of \eqref{EL_giugno} is that $\alpha$ is Lipschitz continuous in each interval $[t_1,t_2]\subset [0,T]$ such that $y(t)\ne O$ for $t\in(t_1,t_2)$.
\end{remark}
\begin{remark}\label{rem:EL2}
If we only suppose that for some $p\in [1,\infty]$, $\ell_i(\cdot, t)$ is bounded in $ W^{1,p}_{{ \rm loc}}(J_i)$  uniformly with respect to $t\in [0,T]$,
% locally Lipschitz in $x$ uniformly with respect to $t\in [0,T]$,
then $y\in W^{2,p}(t_1,t_2)$ and
%then $y\in W^{2,\infty}(t_1,t_2)$ and
\eqref{EL_giugno} holds for almost all $s\in (t_1,t_2)$.
\end{remark}

The following lemma deals with the transversality condition for an  optimal trajectory which stays in the interior of a given edge near the horizon $T$.

%
%Lemma: trasversality condition
%

\begin{lemma}\label{lemma:trasver}
We keep the assumptions of Lemma~\ref{lemma:EL} and we also assume  that $g_i\in C^1(J_i)$. Consider any $(x,t)\in \cG\times[0,T]$ and any $(y,\alpha)\in \Gamma^{\rm{opt}}_t[x]$ such that $y(T)\in J_i\setminus \{O\}$. Then, there holds
\begin{equation}\label{trasver_giugno}
\alpha(T)=-\partial_x g_i(y(T)).
\end{equation}
\end{lemma}
\begin{proof}
The  arguments are similar to those in the proof of Lemma~\ref{lemma:EL}.
Since  $t\mapsto y(t)$ is continuous with $y(T)\in J_i\setminus\{O\}$, there exists $\delta>0$ such that $y(s)\in J_i\setminus\{O\}$ for $s\in [T-\delta,T]$. 
Consider $\alpha_1\in L^2(t,T)$ with $\alpha_1(s)\in \R e_i$  a.e. in $(T-\delta,T)$, $\alpha_1(s)=0$ a.e. in $(t,T-\delta)$.
In $(T-\delta,T)$, both $\alpha$ and $\alpha_1$ are aligned with $e_i$ and  we may write $\alpha(s)=\bar \alpha(s)e_i$ and $\alpha_1(s)=\bar \alpha_1(s)e_i$. 
As before, for $h\in\R$ with $|h|$ sufficiently small, the control $\alpha_h(\cdot):=\alpha(\cdot)+h\alpha_1(\cdot)$ is admissible for $(x,t)$. Let $y_h$ be the trajectory corresponding
to the control~$\alpha_h$. We deduce from the optimality of $(y,\alpha)$ that
\begin{equation*}%\label{10giu_2}
0\leq 
\int_{T-\delta}^{T}\left(\bar \alpha(s)\bar \alpha_1(s)+\frac{h\bar \alpha_1(s)^2}2+ \frac{\ell_i(y_h(s),s)-\ell_i(y(s),s)}{h}\right)ds + \frac{g_i(y_h(T))-g_i(y(T))}{h}.
\end{equation*}
 Since $y_h(s)=y(s)+h\int_{T-\delta}^s \alpha_1(\tau)d\tau$ for $s\in[T-\delta,T]$, arguing as in the proof of Lemma~\ref{lemma:EL} leads to
\begin{equation*}
\frac{g_i(y_h(T))-g_i(y(T)}{h}=
\partial_x g_i(y(T)) \int_{T-\delta}^T\bar\alpha_1(\tau)d\tau +O(h),
\end{equation*}
and
\begin{align*}
\int_{T-\delta}^{T}\frac{\ell_i(y_h(s),s)-\ell_i(y(s),s)}{h}\, ds=
\int_{T-\delta}^{T} \partial_x\ell_i(y(s),s) \int_{T-\delta}^s\bar\alpha_1(\tau)d\tau ds +O(h)\\
\qquad =
\int_{T-\delta}^{T}\partial_x\ell_i(y(\theta),\theta)d\theta \int_{T-\delta}^{T}\bar\alpha_1(\tau)d\tau
- \int_{T-\delta}^{T}\left(\int_{T-\delta}^{s}\partial_x\ell_i(y(\theta),\theta)d\theta\right) \bar\alpha_1(s)ds,
\end{align*}
where the last equality is obtained after an integration by parts.
Combining the  latter three inequalities and letting $h\to 0$ yield
\begin{multline*}
0\leq 
\left(\int_{T-\delta}^T \partial_x\ell_i(y(\theta),\theta)d\theta+\partial_x g_i(y(T))\right)\int_{T-\delta}^{T}\bar \alpha_1(s)ds+\\
\int_{T-\delta}^{T}\left[\bar \alpha(s) -\int_{T-\delta}^s \partial_x\ell_i(y(\theta),\theta)d\theta\right]\bar \alpha_1(s)ds.
\end{multline*}
Since $y(s)\in J_i\setminus\{O\}$ for $s\in [T-\delta,T]$,  we infer from~\eqref{EL_giugno} that
\begin{equation*}
0\leq \left(\bar \alpha(T)+\partial_x g_i(y(T))\right)\int_{T-\delta}^{T}\bar \alpha_1(s)ds.
\end{equation*}
This yields \eqref{trasver_giugno} since  $\alpha_1$ is arbitrary.
\end{proof}

\subsection{Lipschitz regularity of optimal trajectories}

We now aim at proving that for any $(x,t)\in \cG\times[0,T]$, any trajectory $(y,\alpha)\in \Gamma^{\rm{opt}}_t[x]$ is such that $\alpha$ is bounded in $(t,T)$, with a bound that depends locally uniformly on $x$. 
The essential arguments are  the Euler-Lagrange and the transversality conditions obtained in Section \ref{sec:euler_lagrange} and a key estimate on the initial velocity of an optimal trajectory, locally independent of the starting point, see Lemma \ref{lemma:casi_alphalim} below.

\begin{theorem}\label{optcur_lip}
Assume that for all $i=1,\dots, N$,  $g_i\in C^1(J_i)$ with $\partial_xg_i\in C_b(J_i)$, $\ell_i$ is differentiable with respect to its first argument with $\partial _x \ell_i \in C_b(J_i\times[0,T])$, and let $M_g$, $M_\ell$, $L_g$ and $L_\ell$ be defined by 
\begin{equation}
\label{eq:9999}
\begin{split}
\ds  M_g= \| g\|_{L^\infty(\mathcal G )}, \quad L_g=  \max_{i=1,\dots, N} \| \partial_x g_i\|_{L^\infty( J_i )}, \\
\ds  M_\ell= \| L\|_{L^\infty(\mathcal G \times [t,T] )}, \quad  L_\ell=  \max_{i=1,\dots, N} \| \partial_x \ell_i\|_{L^\infty( J_i\times [t,T] )}. 
\end{split}
\end{equation}
% Assume that for some $K>0$, for each $i=1,\dots,N$,
% \begin{equation}\label{H2}
% g_i \in C^2(\R e_i), \qquad \ell_i(\cdot,s) \in C^2(\R e_i)\quad \textrm{with }\|\ell_i(\cdot,s)\|_{C^2}\leq K \quad\forall s\in[t,T],
% \end{equation}
For any $(x,t)\in \cG\times[0,T]$ and for any trajectory $(y,\alpha)\in \Gamma^{\rm{opt}}_t[x]$, the control~$\alpha$ belongs to $L^\infty(t,T)$. Moreover, there exists a positive constant $V$ (depending only on~$L_g$, $M_\ell$, $L_\ell$, $d(x,O)$ and $(T-t)^{-1}$) such that
\begin{equation*}  
\|\alpha\|_\infty \leq V.
\end{equation*}
\end{theorem}
\begin{proof}
%The starting point~$x$ can either be inside some edge (say, $x\in J_i\setminus\{O\}$) or coincide with the junction~$O$. In the former case ($x\in J_i\setminus\{O\}$), the path~$y(\cdot)$ can
%\begin{itemize}
%\item{i)] always remain in J_i\setminus\{O\}$
%\item[ii)] arrive in~$O$ at some time $t_0\in (t,T]$.
%\end{itemize}
%In case $(ii)$, after time $t_0$, the path~$y(\cdot)$ can
%\begin{itemize}
%\item[a)] remain in~$O$ up to time~$T$
%\item[b)] remain in~$O$ up to some time~$t_1$, with $t_1\in [t_0,T)$ and after move according to points $(c)$ and $(d)$ below
%\item[c)] enter in some edge $J_j\setminus\{O\}$ and remain there up to time~$T$
%\item[d)] enter in some 
%\end{itemize}
Consider a trajectory $(y,\alpha)\in \Gamma^{\rm{opt}}_t[x]$.
Set
 %  \begin{equation}
%     \label{eq:9999}
%   M_g= \| g\|_{L^\infty(\mathcal G )}, \quad L_g=  \max_{i=1,\dots, N} \| \partial_x g_i\|_{L^\infty( J_i )}, \quad
%   M_\ell= \| L\|_{L^\infty(\mathcal G \times [t,T] )}, \quad  L_\ell=  \max_{i=1,\dots, N} \| \partial_x \ell_i\|_{L^\infty( J_i\times [t,T] )},
%   \end{equation}
% and
\begin{equation}\label{eq:Vstar}
V_*=L_g+(T-t)L_\ell.    
\end{equation}

Let us split the interval $[t,T]$ in order to distinguish the times $s$ for which $y(s)\in J_i\setminus\{O\}$, $i=1,\dots,N$, and $y(s)=O$. More precisely, set
\begin{equation*}
I_0=\{s\in[t,T]:\; y(s)=O\},\qquad I_i=\{s\in[t,T]:\; y(s)\in J_i\setminus\{O\}\}\quad \textrm{for }i=1,\dots,N.
\end{equation*}
Since $y(\cdot)$ is continuous, the set $I_0$ is closed and each $I_i$ can be written as the disjoint union of a (possibly infinite) family of
subintervals of $[t,T]$, open in $[t,T]$.

We aim at bounding $\|\alpha\|_\infty$. For that, we consider the following different cases:
\begin{enumerate}
\item From Stampacchia theorem, $\alpha(s)=0$ for a.e. $s\in I_0$.
\item  Assume that, for some $t_1\in (t,T)$ and for some $i\in\{1,\dots,N\}$, $(t_1, T]\subset I_i$.
This  implies in particular that $y(T)\in J_i\setminus\{O\}$. From the Euler-Lagrange condition \eqref{EL_giugno} 
and the transversality condition \eqref{trasver_giugno}, 
\begin{equation*}
\alpha(s)=-\partial_x g_i(y(T))+\int_{T}^{s}\partial_x\ell_i(y(\tau),\tau)\, d\tau.
\end{equation*}
From the assumptions made on $g_i$ and $\ell_i$, this implies that $\|\alpha\|_{L^\infty(t_1,T)}\le  V_*$.
\item Assume that for some $i\in\{1,\dots,N\}$, $[t,T]\subset I_i$. Then the same argument as in the previous point yield that  $\|\alpha\|_{L^\infty(t,T)}\le V_*$.
\item Assume that, for some $t_1,t_2\in [t,T]$,  $y(t_1)=y(t_2)=O$, and for some $i\in\{1,\dots,N\}$, $(t_1,t_2)\subset I_i$.
From Lemma~\ref{lemma:EL}, the function $s\mapsto \alpha(s)$ is continuous on $(t_1,t_2)$. Then, from standard calculus, we deduce that there exists $t_3\in(t_1,t_2)$ such that $\alpha(t_3)=0$.
Then \eqref{EL_giugno} implies that for $s\in (t_1,t_2)$,
\begin{equation*}
\alpha(s)=\int_{t_3}^{s}\partial_x\ell_i(y(\tau),\tau)\, d\tau,
\end{equation*}
therefore that $\|\alpha\|_{L^\infty(t_1,t_2)}\le V_*$.
\item Assume that, for some $t_1\in (t,T)$,  $y(t_1)=O$,   and for some $i\in\{1,\dots,N\}$, $[t,t_1)\subset I_i$.
 From Remark \ref{rem:EL}, the control $\alpha$ is Lipschitz continuous 
 and the bound \eqref{EL_giugno} holds in $(t,t_1)$.
In particular, $\alpha(t)$ is well defined. Take $\alpha(s)=\bar \alpha(s)e_i$ and $y(s)=\bar y(s)e_i$ for $s\in [t,t_1]$. It is clear that
\begin{equation}\label{eq:stima_EL}
\bar \alpha(t)- L_\ell (s-t)   \leq \bar \alpha(s)\leq \bar \alpha(t)+ L_\ell (s-t)
\qquad\forall s\in[t,t_1).
\end{equation}
We distinguish two subcases
\begin{enumerate}
\item  If $\bar \alpha(t)$ is nonnegative, then since $\bar y(t_1) =0< \bar y(t)$, there exists $t_2: t\le t_2<t_1$ such that $\bar \alpha(t_2)=0$.
  As above $ \bar \alpha(s) =  \int_{t_2} ^s  \partial_x \ell_i(y(\tau),\tau)  d\tau $, which yields $\|\alpha\|_{L^\infty(t,t_1)}\le V_*$
\item  If $\bar \alpha(t)$ is negative, then we can apply  Lemma~\ref{lemma:casi_alphalim} below, which yields the desired  bound on $\|\alpha\|_{L^\infty(t,t_1)}$. 
\end{enumerate}
\end{enumerate}
By using the fact that $[t,T]=\cup_{i=0}^N I_i$, the observations above on $I_0$ and $I_i$,
and by combining all the points above, we get the desired estimate on $\|\alpha\|_{L^\infty(0,T)}$.
\end{proof}

\begin{lemma}\label{lemma:casi_alphalim}
Keeping the assumptions of Theorem~\ref{optcur_lip}, we also assume  that, for some $t_1\in(t,T]$,  $y(t_1)=O$, and for  some $i\in\{1,\dots,N\}$, $y(s)\in J_i\setminus\{O\}$ for $s\in[t,t_1)$, and that $\alpha(t)\cdot e_i<0$. Then, for some positive constant $C$ (depending only on   $(T-t)^{-1}$, $d(x,O)$, $M_\ell$, $L_\ell$, $L_g$ defined in \eqref{eq:9999}), there holds
\[\alpha(s)\cdot e_i\geq -C\qquad\textrm{in }[t,t_1).\]
\end{lemma}

\begin{proof}[Proof of Lemma~\ref{lemma:casi_alphalim}]
Set $x=\bar x e_i$, $\alpha(t)=-\bar v e_i$ with $\bar v>0$, and for any $s\in [t,t_1)$, let $\bar\alpha(s),\bar y (s)$ be the real numbers such that
$\alpha(s)=\bar \alpha(s)e_i$ and $y(s)=\bar y(s)e_i$.

Hence, from Lemma~\ref{lemma:EL},  the claim  is equivalent to the existence of
some positive $C$ (depending only on   $M_\ell$, $L_\ell$, $L_g$, $d(x,O)$ and $(T-t)$), such that
\[\bar v\leq C.\]

%  Set $M_\ell:=\|\ell\|_{L^\infty([t,T])}$ and $L_\ell:=\|\partial_x\ell\|_{L^\infty([t,T])}$.
From \eqref{EL_giugno}, for $s\in[t,t_1)$ there holds
\begin{equation}\label{eq:stima_EL2}
\left\{
\begin{array}{rl}
(i)&\quad-\bar v-L_\ell(s-t)\leq \bar \alpha(s)\leq -\bar v+L_\ell(s-t),\\
(ii)&\quad\bar x-\bar v(s-t)-\frac{L_\ell(s-t)^2}{2}\leq \bar y(s)\leq \bar x-\bar v(s-t)+\frac{L_\ell(s-t)^2}{2}.
\end{array}\right.
\end{equation}
Let us start by some useful estimates. We claim that, for $\bar v\geq 2L_\ell T$ there holds
\begin{equation}\label{eq:est_t_1}
t+\frac{4\bar x}{5\bar v}\leq t_1\leq t+\frac{4\bar x}{3\bar v}.
\end{equation}
Indeed, the left inequality in~\eqref{eq:stima_EL2}-(ii) with $s=t_1$ yields
\begin{equation*}
\bar x\leq (t_1-t)\left[\bar v+\frac{L_\ell(t_1-t)}{2}\right]\leq (t_1-t)\left[\bar v+\frac{L_\ell T}{2}\right]\leq (t_1-t)\frac{5\bar v}{4}.
\end{equation*}
Analogously, the right inequality in~\eqref{eq:stima_EL2}-(ii) with $s=t_1$ yields
\begin{equation*}
-\bar x\leq (t_1-t)\left[-\bar v+\frac{L_\ell(t_1-t)}{2}\right]\leq (t_1-t)\left[-\bar v+\frac{L_\ell T}{2}\right] \leq -(t_1-t)\frac{3\bar v}{4}.
\end{equation*}
This concludes the proof of \eqref{eq:est_t_1}.

We now claim that, for $\bar v\geq \max\{2L_\ell T,4 \bar x /(3T)\}$, there holds
\begin{equation}\label{eq:est_baralpha}
\bar \alpha(s)\leq -\frac {\bar v}2 \quad \forall s\in[t,t_1)\qquad\textrm{and}\qquad\int_t^{t_1}\frac{\bar \alpha(s)^2}{2}\, ds \geq  \frac{\bar v \bar x}{10}.
\end{equation}
Indeed, observe first that  \eqref{eq:stima_EL2}-(i) entails
\begin{equation*}
\bar \alpha(s)\leq -\bar v+L_\ell(t_1-t) \qquad\forall s\in[t,t_1).
\end{equation*}
From estimate~\eqref{eq:est_t_1} and our choice of $\bar v$, 
\begin{equation*}
\bar \alpha(s)\leq -\bar v+\frac{2\bar x}{3  T}    \le  -\frac{\bar v}2  \qquad\forall s\in[t,t_1).
\end{equation*}
where we have  successively used that $\bar v\geq 2L_\ell T$ and that  $3T \bar v\geq 4 \bar x $. Next, we deduce from the first inequality in \eqref{eq:est_baralpha} and  \eqref{eq:est_t_1} that
\begin{equation*}
\int_t^{t_1}\frac{\bar \alpha(s)^2}{2}\, ds\geq \frac{\bar v^2(t_1-t)}{8}\geq \frac{\bar v \bar x}{10},
\end{equation*}
and \eqref{eq:est_baralpha} is  proved.

%{\color{red} Let $N$ be a fixed positive value. Here, there is a confusion between this N and the number of edges} 
We are now going to find estimates on $\bar v$ by proposing 
suitable competitors for the optimal control problem defining $u(x,t)$. Let $V_*$ be the constant defined in \eqref{eq:Vstar}.

If  $\bar v\le  \max\left\{ 2L_\ell T, 4 \bar x /(3T),    40 \bar x /(T-t) , 20 V_*   \right\}$, there is nothing to do. We are left with estimating $\bar v$ in the case when
\begin{equation}\label{eq:10000}
\bar v> \max\left(2L_\ell T, \frac {4 \bar x} {3T},  \frac {40 \bar x} {T-t}, 20 V_*  \right) .
\end{equation}
The arguments below differ  according to the behaviour of $(y,\alpha)$ after time $t_1$.

\noindent{\it Case A: $d(y(s),O)\leq \bar x$ for $s\in[t_1,T]$.}

Recall that the case under focus is when \eqref{eq:10000} holds. Consider the control
\begin{equation*}
\alpha_1(s)=-\bar v/20\, e_i\quad\textrm{in }[t,t+20 \bar x/\bar v ],\qquad
\alpha_1(s)=0\quad\textrm{in }[t+20 \bar x/\bar v,T].
\end{equation*}
Let  $(y_1,\alpha_1)$ be  the corresponding trajectory.  Observe that $(y_1,\alpha_1)$ is admissible for $(x,t)$, so the optimality of $(y,\alpha)$ entails
\begin{eqnarray*}
0\leq J_t(x;(y_1,\alpha_1))-J_t(x;(y,\alpha)) &\leq& 
\int_{t}^{t+ 20 \bar x/\bar v}\left(\frac{\bar v^2}{800}+\ell_i(y_1(s),s)-L(y(s),s)\right) ds\\
&& -\int_{t}^{t_1}\frac{\bar \alpha(s)^2}{2}ds+\int_{t+20 \bar x/\bar v}^T\left(\ell_O(s)-L(y(s),s)\right)ds\\
&& +g(O) -g(y(T)).
\end{eqnarray*}
Since $\bar v> \max\{2L_\ell  T, 4\bar x /(3T)\}$, \eqref{eq:est_baralpha} implies that
\begin{equation}\label{eq:case_A}
0\leq -\frac{3 \bar x \bar v }{40} +40 M_\ell \frac {\bar x}{\bar v} +\int_{t+20\bar x/\bar v}^T\left(\ell_O(s)-L(y(s),s)\right)ds +g(O) -g(y(T)).
\end{equation}
Denoting by $I$ the last integral, \eqref{eq:462} and \eqref{eq:460} yield 
\begin{eqnarray}\label{eq:case_Abis}
I&=&\sum_{i=1}^N\int_{t+20 \bar x/\bar v}^T\left[\ell_O(s)-\ell_i(y(s),s)\right]\car_{y(s)\in J_i\setminus\{O\}}ds\\\notag
&\leq &\sum_{i=1}^N\int_{t+20 \bar x/\bar v}^T\left[\ell_i(0,s)-\ell_i(y(s),s)\right]\car_{y(s)\in J_i\setminus\{O\}}ds\\\notag
&\leq & TL_\ell\bar x,
\end{eqnarray}
where the latter inequality follows from the definition of case $A$. Similarly, $g(O) -g(y(T))\le L_g \bar x$.
Injecting these estimates in~\eqref{eq:case_A}, we get
\begin{equation*}
0\leq  -\frac{3\bar v^ 2 }{40} +(TL_\ell+L_g) \bar v +40 M_\ell ,
\end{equation*}
which implies that $\bar v\le  \frac {20} 3 \left(   \sqrt{ ((TL_\ell+L_g))^2 +12 M_\ell} +(TL_\ell+L_g)\right)$.
We have proven that in case A,
\begin{equation}
\label{eq:10001}
\bar v\le \max\left(    2L_\ell T, \frac {4 \bar x} {3T},  \frac {40 \bar x} {T-t},   20 V_*,  \frac {20} 3 \left(   \sqrt{ ((TL_\ell+L_g))^2 +12 M_\ell} +(TL_\ell+L_g)\right)   \right).
\end{equation}

\noindent{\it Case B: $\exists \tau\in(t_1,T]$ such that $d(y(\tau),O)>  \bar x$.}
Recall that  \eqref{eq:10000} holds.
For later use, set
\begin{eqnarray*}
&&\tau_2=\inf\{\tau\in [t_1,T]:\; d(y(\tau),O)>\bar x\},\\
&&\textrm{$l\in\{1,\dots,N\}$ such that $y(\tau_2)\in J_l\setminus\{O\}$},\\
&&\tau_1=\inf\{s \in [t_1,T]:\; y(\tau)\in J_l\setminus\{O\}\quad \forall \tau\in(s,\tau_2]\}.
\end{eqnarray*}
In other words, $\tau_2$ is the first time larger than $t_1$ at which the trajectory
reaches a distance to the origin greater than $\bar x$ and  $\tau_1$ is the time at which the trajectory enters in $J_l\setminus\{O\}$ and remains there up to time $\tau_2$
(note that the trajectory can also visit $J_l\setminus\{O\}$ before $\tau_1$).

Let us distinguish three subcases.

\noindent{\it Subcase B1: $\tau_1\geq t+ 20 \bar x/\bar v$.}
Consider the control
\begin{equation*}
\alpha_1(s)=-\frac {\bar v}{20} e_i \textrm{ in }[t,t+20 \bar x/\bar v),\qquad 
\alpha_1(s)=0 \textrm{ in }[t+20 \bar x/\bar v, \tau_1),\qquad
\alpha_1(s)=\alpha(s)\textrm{ in }[\tau_1,T],
\end{equation*}
and let $(y_1,\alpha_1)$ be the corresponding trajectory, which is clearly admissible for $(x,t)$. The optimality of $(y,\alpha)$ entails
\begin{multline*}
0\leq J_t(x;(y_1,\alpha_1))-J_t(x;(y,\alpha))\leq 
\int_{t}^{t+20 \bar x/\bar v}\left(\frac{\bar v^2}{800}+\ell_i(y_1(s),s)-L(y(s),s)\right) ds\\
-\int_{t}^{t_1}\frac{\bar \alpha(s)^2}{2}ds+\int_{t+20 \bar x/\bar v}^{\tau_1}\left(\ell_O(s)-L(y(s),s)\right)ds.
\end{multline*}
Then, from \eqref{eq:est_baralpha},
\begin{equation*}
0\leq \left(\frac{\bar v}{40}-\frac{\bar v}{10}\right)\bar x +40\frac{M_\ell}{\bar v}\bar x+\int_{t+20 \bar x/\bar v}^{\tau_1}\left(\ell_O(s)-L(y(s),s)\right)ds.
\end{equation*}
As above, we deduce that
\begin{equation*}
  0\leq -\frac{3\bar v^2 }{40} + TL_\ell \bar v +40M_\ell,
\end{equation*}
which proves that in Subcase B1,
\begin{equation}
\label{eq:10002}
\bar v\le \max\left(    2L_\ell T, \frac {4 \bar x} {3T},  \frac {40 \bar x} {T-t},   20 V_* , \frac {20} 3 \left(   \sqrt{ (TL_\ell)^2 +12 M_\ell} +TL_\ell\right)   \right).
 \end{equation}

\noindent{\it Subcase B2: $\tau_1< t+20 \bar x/\bar v$ and $y(s)\in J_l\setminus\{O\}$ for $s\in (\tau_1,T]$.}
Consider the control
\begin{equation}\label{contr_B2}
\alpha_1(s)=-\frac {\bar v}{20} e_i \textrm{ in }[t,t+20 \bar x/\bar v],\qquad 
\alpha_1(s)=a \alpha(as+b) \textrm{ in }(t+20 \bar x/\bar v, T],
\end{equation}
with
\begin{equation*}
a=\frac{T-\tau_1}{T-t- 20\bar x/\bar v}, \quad \hbox{and} \quad b=T\frac{\tau_1-t-20 \bar x/\bar v}{T-t-20\bar x/\bar v},
\end{equation*}
(note that $a>1$).
Let $(y_1,\alpha_1)$ be the corresponding trajectory. There holds
\begin{equation*}
y_1(s)=(\bar x-\bar v(s-t)/20)e_i\quad \textrm{in }[t,t+20\bar x/\bar v],\qquad 
y_1(s)=y(as+b) \textrm{ in }[t+20 \bar x/\bar v, T].
\end{equation*}
In particular, $(y_1,\alpha_1)$ is admissible for $(x,t)$. 
The optimality of $(y,\alpha)$ entails
\begin{eqnarray}\notag
0&\leq& J_t(x;(y_1,\alpha_1))-J_t(x;(y,\alpha))\\ \notag
&\leq& 
\int_{t}^{t+20 \bar x/\bar v}\left(\frac{\bar v^2}{800}+\ell_i(y_1(s),s)\right) ds+\int_{t+20\bar x/\bar v}^T\left(\frac{a^2 | \alpha(as+b)|^2}{2} +
L(y_1(s),s)\right) ds\\ \notag
&&-\int_{t}^{t_1}\frac{\bar \alpha(s)^2}{2}ds-\int_{t_1}^{\tau_1}\frac{ |\alpha(s)|^2}{2}ds-\int_{\tau_1}^{T}\frac{ |\alpha(s)|^2}{2}ds -\int_{t}^{\tau_1}L(y(s),s)ds-\int_{\tau_1}^{T}L(y(s),s)ds\\ \notag
&\leq& 
\left(\frac{\bar v}{40}-\frac{\bar v}{10}\right)\bar x+
\int_{t}^{t+20\bar x/\bar v}\ell_i(y_1(s),s)ds+\int_{t+20\bar x/\bar v}^T\left(\frac{a^2 |\alpha(as+b)|^2}{2} + L(y_1(s),s)\right) ds\\ \label{stima_B2I}
&&-\int_{\tau_1}^{T}\frac{|\alpha(s)|^2}{2}ds -\int_{t}^{\tau_1}L(y(s),s)ds-\int_{\tau_1}^{T}L(y(s),s)ds.
\end{eqnarray}
Similarly as above,
\begin{equation}\label{eq:stimaB2I}
\int_{t}^{t+   20\bar x/\bar v}\ell_i(y_1(s),s)ds-\int_{t}^{\tau_1}L(y(s),s)ds\leq M_\ell\left(\frac{20 \bar x}{\bar v}+(\tau_1-t)\right)\leq 40 M_\ell\frac{\bar x}{\bar v}.
\end{equation}
On the other hand, \begin{equation*}
\int_{t+20 \bar x/\bar v}^T\frac{a^2 |\alpha(as+b)|^2}{2}ds-\int_{\tau_1}^{T}\frac{ |\alpha(s)|^2}{2}ds=(a-1) \int_{\tau_1}^{T}\frac{|\alpha(s)|^2}{2}ds\leq
\frac{40 \bar x}{\bar v(T-t)}\int_{\tau_1}^{T}\frac{ |\alpha(s)|^2}{2}ds,
\end{equation*}
where the latter inequality comes from the fact that  $a-1\le \frac {20 \bar x/\bar v}{T-t-20\bar x/\bar v} $ and that
 $\bar v(T-t)>40 \bar x$.

Recall that $V_*$ is the constant defined in \eqref{eq:Vstar}, and that $\|\alpha\|_\infty\leq V_*$ in $[\tau_1,T]$.
Then, from the latter inequality, we deduce
\begin{equation}\label{eq:stimaB2Ibis}
\int_{t+20\bar x/\bar v}^T\frac{a^2 |\alpha(as+b)|^2}{2}ds-\int_{\tau_1}^{T}\frac{|\alpha(s)|^2}{2}ds\leq  \frac{20 \bar x}{\bar v} V_*^2 . 
\end{equation}
On the other hand, 
\begin{equation*}
\int_{t+ 20\bar x/\bar v}^TL(y_1(s),s) ds-\int_{\tau_1}^{T}L(y(s),s)ds=
\int_{t+20\bar x/\bar v}^TL(y(as+b),s) ds-\int_{\tau_1}^{T}L(y(s),s)ds= I_1+I_2
\end{equation*}
for
\begin{equation*}
\begin{array}{rcl}
I_1&=& \ds -\int_{\tau_1}^{t+20\bar x/\bar v}L(y(s),s)ds\\
I_2&=& \ds \int_{t+20\bar x/\bar v}^T\left(L(y(as+b),s) - L(y(s),s) \right)ds.
\end{array}
\end{equation*}
Since  $   0<   t+20 \bar x /\bar v -\tau_1\le 20 \bar x /\bar v$,
\begin{equation*}
I_1\leq 20 M_\ell \frac{\bar x}{\bar v}.
\end{equation*}
On the other hand, since both $y(as+b)$ and $y(s)$ belong to $J_l\setminus\{O\}$ for $s\in(t+  20\bar x/\bar v,T]$, there holds
\begin{equation}\label{eq:stimaB2quater}
d(y(as+b),y(s))\leq \int_{as+b}^{s}|\alpha(\theta)|d\theta\leq  V_* (T-s)\frac{t+20 \bar x/\bar v-\tau_1}{T-t-20 \bar x/\bar v}\leq
\frac{V_* (T-t)}{T-t-20 \bar x/\bar v} \frac{20 \bar x}{\bar v}\leq 40V_*\frac{\bar x}{\bar v},
\end{equation}
where the last inequality comes from the fact that $\bar v(T-t)>40 \bar x$. This implies that
\begin{equation*}
I_2\leq 40V_* L_\ell T \frac{\bar x}{\bar v}.
\end{equation*}
Hence,
\begin{equation}\label{eq:stimaB2ter}
\int_{t+\bar x/K_0}^TL(y_1(s),s) ds-\int_{\tau_1}^{T}L(y(s),s)ds\leq 20 (M_\ell+2L_\ell  V_* T)\frac{\bar x}{\bar v}.
\end{equation}
Injecting \eqref{eq:stimaB2I}, \eqref{eq:stimaB2Ibis} and~\eqref{eq:stimaB2ter} in \eqref{stima_B2I}, we obtain 
\begin{equation*}
  0\leq  -\frac{3\bar v \bar x}{40} + 20 ( 3 M_\ell+V_*^2+ 2 V_* L_\ell T  ) \frac{\bar x}{\bar v},
\end{equation*}
thus
\begin{equation}
\label{eq:10003}
\bar v\le \max\left(2L_\ell T, \frac {4 \bar x} {3T},  \frac {40 \bar x} {T-t},  20 V_*, 20  \sqrt{  \frac 2 3  ( 3 M_\ell+ V_*^2+ 2 L_\ell V_*T)}\right) .
\end{equation}

\noindent{\it Subcase B3: $\tau_1< t+ 20\bar x/\bar v$ and $\exists \tau\in(\tau_1,T)$ such that $y(\tau)=O$.}
Set
\begin{equation*}
\tau_3=\min\{\tau\in(\tau_1,T]:\; y(\tau)=O\},
\end{equation*}
i.e. $\tau_3$ is the first time greater than  $\tau_1$ at which the trajectory $y(\cdot)$ reaches the vertex.
Clearly, from the  definition of $\tau_1$, $y(\tau)\in J_l\setminus\{O\}$ for $\tau\in(\tau_1,\tau_3)$ and $\tau_3>\tau_2$.
As in the  previous cases, 
\begin{equation*}
\tau_2>\tau_1+\frac{\bar x}{V_*}\geq t+\frac{\bar x}{V_*}.
\end{equation*}
Since $ \bar v> 20V_*$, we know that  $\tau_2>t+20 \bar x/\bar v$. Hence, 
\begin{equation*}
\tau_1<t+20\frac{\bar x}{\bar v}<\tau_2<\tau_3.
\end{equation*}
Consider the trajectory $(y_1,\alpha_1)$ defined in~\eqref{contr_B2}. Note that %Recall $y_1(s)=y(as+b)$ in $(t+20\bar x/v, T]$. Hence,
\begin{equation*}
y_1(s)= y(as+b) \in J_l\setminus\{O\} \qquad\forall s\in I_*:=\left[t+20\frac{\bar x}{\bar v}, \frac{\tau_3(T-t-20\bar x/\bar v)-T(\tau_1-t-20\bar x/\bar v)}{T-\tau_1}\right] .
\end{equation*}
Observe that $\tau_3\in I_*$ and $y(\cdot)-y_1(\cdot)=(\bar y(\cdot)-\bar y_1(\cdot))e_l$ in~$I_*$
with $\bar y(t+20\bar x/\bar v)-\bar y_1(t+20\bar x/\bar v)>0$ and $\bar y(\tau_3)-\bar y_1(\tau_3)<0$.
We deduce that there exists $\tau_4\in (t+20\bar x/\bar v,\tau_3)$ such that $y(\tau_4)=y_1(\tau_4)$.

We can now choose a competitor $(y_2,\alpha_2)$ as the trajectory corresponding to the control
\begin{equation*}%\label{contr_B2}
\alpha_2(s)=\alpha_1(s)\textrm{ in }[t,\tau_4],\qquad 
\alpha_2(s)=\alpha(s) \textrm{ in }(\tau_4, T].
\end{equation*}
Note that there holds: $y_2(s)\in J_i\setminus\{O\}$ for $s\in [t,t+20\bar x/\bar v)$, $y_2(t+20\bar x/\bar v)=O$, $y_2(s)\in J_l\setminus\{O\}$ and $y_2(s)=y(as+b)$ for $s\in [t+20\bar x/\bar v,\tau_4)$, $y_2(s)=y(s)$ for $s\in [\tau_4,T]$.

The optimality of $(y,\alpha)$ entails
\begin{eqnarray}\notag
0&\leq& J_t(x;(y_2,\alpha_2))-J_t(x;(y,\alpha))\\ \notag
&\leq& 
\int_{t}^{t+20\bar x/\bar v}\left(\frac{\bar v^2}{800}+\ell_i(y_2(s),s)\right) ds+\int_{t+20\bar x/\bar v}^{\tau_4}\left(\frac{a^2 |\alpha(as+b)|^2}{2}+ L(y_2(s),s)\right) ds\\ \notag
&&
%+\int_{\tau_4}^T\left(\frac{\bar \alpha(s)^2}{2}+ L(y_2(s),s)\right) ds
-\int_{t}^{t_1}\frac{\bar \alpha(s)^2}{2}ds-\int_{t_1}^{\tau_1}\frac{ |\alpha(s)|^2}{2}ds-\int_{\tau_1}^{\tau_4}\frac{|\alpha(s)|^2}{2}ds- \int_{t}^{\tau_1}L(y(s),s)ds\\\label{stima:B2II}
&&%- \int_{\tau_4}^T\frac{\bar \alpha(s)^2}{2}ds
-\int_{\tau_1}^{\tau_4}L(y(s),s)ds.% -\int_{\tau_4}^TL(y(s),s)ds.
\end{eqnarray}
As above,
\begin{equation*}
\int_{t}^{t+20\bar x/\bar v}\frac{\bar v^2}{800}ds-\int_{t}^{t_1}\frac{\bar \alpha(s)^2}{2}ds\leq -\frac{3\bar v \bar x}{40}.
\end{equation*}
The same arguments as those used for obtaining \eqref{eq:stimaB2I},\eqref{eq:stimaB2Ibis} lead to 
\begin{eqnarray*}
&&\int_{t}^{t+ 20\bar x/\bar v}\ell_i(y_2(s),s)ds-\int_{t}^{\tau_1}L(y(s),s)ds\leq 40M_\ell\frac{\bar x}{\bar v},\\
&&\int_{t+20\bar x/\bar v}^{\tau_4}\frac{a^2 |\alpha(as+b)|^2}{2}ds-\int_{\tau_1}^{\tau_4}\frac{ |\alpha(s)|^2}{2}ds\leq 20V_*^2\frac{\bar x}{\bar v}.%\\&&
%\int_{t+\bar x/K_0}^{\tau_4}L(y_2(s),s) ds-\int_{\tau_1}^{\tau_4}L(y(s),s)ds\leq (K+2K\bar KT)\frac{\bar x}{K_0}.
\end{eqnarray*}
On the other hand,
\begin{equation*}
\begin{array}{rcl}
\Lambda&:=&\ds \int_{t+20\bar x/\bar v}^{\tau_4}L(y_2(s),s)ds-\int_{\tau_1}^{\tau_4}L(y(s),s)ds\\
& =&
\ds \int_{t+20\bar x/\bar v}^{\tau_4}L(y(as+b),s)ds-\int_{\tau_1}^{\tau_4}L(y(s),s)ds \\
& =&
\ds -\int_{\tau_1}^{t+20\bar x/\bar v}L(y(s),s)ds+ \int_{t+ 20\bar x/\bar v}^{\tau_4}\left[L(y(as+b),s)-L(y(s),s)\right]ds\\
& \leq & \ds 20 M_\ell\frac{\bar x}{\bar v}+\int_{t+20\bar x/\bar v}^{\tau_4}\left[\ell_l(y(as+b),s)-\ell_l (y(s),s)\right]ds,
\end{array}
\end{equation*}
where the last inequality is due  to the fact that both $y(as+b)$ and $y(s)$ belong to $J_l\setminus\{O\}$ for $s\in (t+20\bar x/\bar v,\tau_4)$.
Observe that estimate~\eqref{eq:stimaB2quater} holds on $[t+20\bar x/\bar v, \tau_4]$, hence
\begin{equation*}
\Lambda\leq 20(M_\ell+2L_\ell V_* T)\frac{\bar x}{\bar v}.
\end{equation*}
Injecting all these estimates in~\eqref{stima:B2II}, we obtain
\begin{equation*}
0\leq -\frac{3\bar v\bar x}{40} +20 \left(3M_\ell+V_*^2+2L_\ell V_* T\right)\frac{\bar x}{\bar v},
\end{equation*}
thus, in Subcase B3,
\begin{equation}
\label{eq:10004}
\bar v\le \max\left(     2L_\ell T, \frac {4 \bar x} {3T},  \frac {40 \bar x} {T-t}, 20 V_*, 20  \sqrt{  \frac 2 3  ( 3 M_\ell+ V_*^2+ 2 L_\ell V_*  T  ) }   \right).
\end{equation}
Finally, in all cases, $\bar v$ is smaller than the maximal value of the  right hand sides in   \eqref{eq:10001},\eqref{eq:10002},\eqref{eq:10003},\eqref{eq:10004}.
\end{proof}

If, in addition to the assumptions made in Theorem~\ref{optcur_lip}, the final cost is continuous on the whole network $\cG$ (thus Lipschitz continuous on~$\cG$ because of the other assumptions), then it turns out that the optimal controls are uniformly bounded in the whole time interval $[0,T]$:
\begin{theorem}\label{optcur_lipT}
We keep the  assumptions of Theorem~\ref{optcur_lip} and also assume  that $g\in C^0(\cG)$. Then, the same result of Theorem~\ref{optcur_lip} holds true with a constant~$V$ independent of $(T-t)$, namely: there exists a constant $V_\#>0$ (dependent on $L_g$, $M_g$, $L_\ell$, $d(x,O)$ but independent of $(T-t)$) such that 
\begin{equation*}
\|\alpha\|_\infty \leq V_\#\qquad \forall (y,\alpha)\in \Gamma^{\rm{opt}}_t[x].
\end{equation*}
\end{theorem}
\begin{proof}
We consider the same cases as in the proof of Theorem~\ref{optcur_lip}. Cases $(1)$-$(4)$ and $(5)-(a)$ are dealt with using the same arguments as in the proof of Theorem~\ref{optcur_lip}. In Case $(5)-(b)$, we apply Lemma~\ref{lemma:lipT} below.
\end{proof}
\begin{lemma}\label{lemma:lipT}
Under the assumptions of Theorem~\ref{optcur_lipT}, the statement of Lemma~\ref{lemma:casi_alphalim} holds true with a constant~$V_\#$ independent of $(T-t)$.
\end{lemma}
\begin{proof}
  We borrow some notations of Lemma~\ref{lemma:casi_alphalim}. In particular, we set: $x=\bar x e_i$, $(y,\alpha)\in \Gamma^{\rm{opt}}_t[x]$, $\alpha(t)=-\bar v e_i$ with $\bar v>0$, $\alpha(s)=\bar \alpha(s) e_i$ for $s\in[t,t_1)$. (Recall: $y(t_1)=O$). By Lemma~\ref{lemma:EL}, without any loss of generality, we assume $\bar v$ so large to have $\bar \alpha(s)<0$ for $s\in [t,t_1)$. \\
Note that points $(1)$-$(4)$ in the proof of Theorem~\ref{optcur_lipT} ensure that there exists a positive constant $V_1$ (dependent on $L_g$, $M_g$, $L_\ell$, $d(x,O)$ but independent of $(T-t)$ and of $(T-t_1)$) such that: $|\alpha(s)| \leq V_1$ for $s\in[t_1,T]$.\\
We proceed constructing a competitor $(y_1,\alpha_1)$. For a constant $\mu\geq V_1$ which will suitably chosen later on, we introduce the trajectory $(y_1,\alpha_1)\in\Gamma_{t,t_1}[x]$ obeying to the control $\alpha_1(s)=\bar \alpha_1(s)e_i$ with 
\begin{equation*}
\bar \alpha_1(s)=\left\{\begin{array}{ll}
\bar \alpha(s)&\qquad\textrm{if }\bar\alpha(s)\geq -\mu\\
0&\qquad\textrm{otherwise. }
\end{array}\right.
\end{equation*}
Clearly, if $y_1(t_1)=y(t_1)$, then $\alpha(\cdot)=\bar \alpha(\cdot)$ a.e. in $[t,t_1]$ and there is nothing to prove. So we consider $y_1(t_1)\ne O$. We take $y_1(s)=\bar y_1(s)e_i$ for $s\in [t,t_1]$. Since $|\bar\alpha_1(\cdot)|\leq |\bar\alpha(\cdot)|$ in $[t,t_1]$, $y_1(t_1)\in J_i\setminus\{O\}$, namely $\bar y_1(t_1)>0$. Recalling $y(t_1)=O$ and $\bar \alpha(\cdot)<0$ in $[t,t_1)$, 
\begin{equation*}
\bar y_1(t_1)=[y_1(t_1)-y(t_1)]\cdot e_i =-\int_t^{t_1}\bar \alpha (s)\car_{\{\bar \alpha(s)<-\mu\}}ds=\int_t^{t_1}|\bar \alpha (s)|\car_{\{\bar \alpha(s)<-\mu\}}ds=: A
\end{equation*}
and also 
\begin{equation}\label{eq:22marzo_0}
d(y_1(s),y(s))\leq -\int_t^{s}\bar \alpha (\tau)\car_{\{\bar \alpha(\tau)<-\mu\}}d\tau\leq A \qquad\forall s\in[t,t_1].
\end{equation}
In order to construct our competitor after time $t_1$, we need an auxiliary trajectory. We consider the trajectory $(y_2,\alpha_2)$ starting at point~$y_1(t_1)=Ae_i$ at time~$t_1$ and obeying to the control $\alpha_2(s)=-V_1e_i$ for $s\in[t_1,t_1+A/V_1]$. Clearly, $y_2(s)\in e_i\setminus\{O\}$ for $s\in[t_1,t_1+A/V_1)$ with $y_2(t_1+A/V_1)=O$. We set
\begin{equation*}
t_2=\min\left\{T,\, t_1+A/V_1,\, \min\{s\in[t_1,T]\;:\; y_2(s)=y(s)\}\right\}
\end{equation*}
namely $t_2$ is the first moment among: the time horizon~$T$, the instant $t_1+A/V_1$ when $y_2$ reaches~$O$ and the first moment when the trajectories $y(\cdot)$ and $y_1(\cdot)$ intersect. On the interval $[t_1,t_2)$, we define our competitor $(y_1,\alpha_1)$ as: $y_1(s)=y_2(s)$. We note that, for $s\in[t_1,t_2)$, there holds
\begin{equation}\label{eq:22marzo_1}
d(y_1(s),y(s))\leq d(y_1(s),O)+d(O,y(s))\leq A-V_1(s-t_1)+V_1(s-t_1)\leq A.
\end{equation}
Let us now argue differently according to the different situations in the definition of time~$t_2$.\\
\texttt{Case (a): $t_2=T$.} In this case, our competitor is already completely constructed. By the optimality of $(y,\alpha)$,
\begin{equation}\label{eq:22marzo_2}
0\leq J_t(x;(y_1,\alpha_1))-J_t(x;(y,\alpha))=\sum_{i=1}^5 I_i
\end{equation}
where
\begin{eqnarray*}
I_1&=&\int_t^{t_1}\frac{|\alpha_1(s)|^2-|\alpha(s)|^2}{2}ds,\qquad I_2=\int_t^{t_1}\left(\ell_i(y_1(s),s)-\ell_i(y(s),s)\right)ds,\\
I_3&=&\int_{t_1}^{t_2}\frac{|\alpha_1(s)|^2-|\alpha(s)|^2}{2}ds,\qquad I_4=\int_{t_1}^{t_2}\left(\ell_i(y_1(s),s)-L(y(s),s)\right)ds, \\I_5&=&g(y_1(T))-g(y(T)).
\end{eqnarray*}
From our choice of $\alpha_1$ in $[t,t_1]$, the Lipschitz continuity of $\ell_i$ and \eqref{eq:22marzo_0},
\begin{equation*}
I_1=-\int_t^{t_1}\frac{|\alpha (s)|^2}{2}\car_{\{\bar \alpha(s)<-\mu\}}ds, \quad
I_2\leq L_\ell T \|d(y_1(s),y(s))\|_{L^\infty(t,t_1)}\leq L_\ell T A.
\end{equation*}
Moreover, we note $t_2-t_1\leq A/V_1$ because of $t_2=T$. From our choice of $\alpha_1$ in $[t_1,t_2]$, 
\begin{equation*}
I_3\leq \int_{t_1}^{t_2}\frac{|\alpha_1 (s)|^2}{2}ds= \frac{V_1^2}{2}(t_2-t_1)\leq \frac{V_1 A}{2}.
\end{equation*}
In order to estimate $I_4$ and $I_5$, observe that for $s\in [t_1,t_2]$ $y_1(s)$ and $y(s)$ may belong to different edges. For this reason, nothing better than
\begin{equation*}
I_4\leq 2M_\ell (t_2-t_1)\leq 2M_\ell A/V_1\quad\textrm{and}\quad
I_5\leq L_g d(y_1(T),y(T))\leq L_g A,
\end{equation*}
can be obtained,
where the latter estimate is due to the global Lipschitz continuity of $g$ and \eqref{eq:22marzo_1} (here, the continuity of $g$ in the vertex plays a crucial role).
Replacing all these estimates in~\eqref{eq:22marzo_2}, by the definition of~$A$, we get
\begin{equation*}
0\leq \int_t^{t_1}\left(-\frac{|\alpha (s)|}{2}+(L_\ell T+V_1/2+2M_\ell/V_1+L_g)\right)|\alpha (s)|\car_{\{\bar \alpha(s)<-\mu\}}ds.
\end{equation*}
Hence, if $\{\bar \alpha(s)<-\mu\}\cap[t,t_1]$ has positive measure and $\mu>2(L_\ell T+V_1/2+2M_\ell/V_1+L_g)$, then we get the desired contradiction.\\
\texttt{Case (b): $t_2=\min\{s\in[t_1,T]\;:\; y_2(s)=y(s)\}$.} We need to construct $(y_1,\alpha_1)$ also on~$(t_2,T]$; we choose: $(y_1(s),\alpha_1(s))=(y(s),\alpha(s))$ for $s\in(t_2,T]$. Note that also in this case,  $t_2-t_1\leq A/V_1$. Following the same calculations as those of case-$(a)$, we end the proof.\\
\texttt{Case (c): $t_2=t_1+A/V_1$ with $t_2<T$ and $t_2<\min\{s\in[t_1,T]\;:\; y_2(s)=y(s)\}\}$.} Observe that $y_1(t_2)=O$ and $y(t_2)\in J_j\setminus\{O\}$ for some $j\in\{1,\dots,N\}$ with $|y(t_2)|\leq A$. In this case, we need to construct our competitor also in the interval $[t_2,T]$. To this end, we need another auxiliary trajectory; let $y_3(\cdot)$ be the path that starts at~$O$ at time~$t_2$ and obeying to the control $\alpha_3(s)=|\alpha(s)|e_j$ for $s\in[t_2,T]$. We set
\begin{equation*}
t_3=\min\left\{T,\, \min\{s\in[t_2,T]\;:\; y_3(s)=y(s)\}\right\}
\end{equation*}
and we define
\begin{equation*}
(y_1(s),\alpha_1(s))=(y_3(s),\alpha_3(s))\qquad \textrm{for }s\in(t_2,t_3].
\end{equation*}
Note that, in the interval $[t_2,t_3]$ both $y_1(s)$ and $y(s)$ belong to the same edge~$J_j$; moreover, by $y_1(t_2)=O$ and $|y(t_2)|\leq A$, for $s\in[t_2,t_3]$ there holds
\begin{equation}\label{eq:22marzo_3}
d(y_1(s),y(s))\leq d\left(\int_{t_2}^s|\alpha(\tau)|d\tau e_j,\left(A+\int_{t_2}^s\alpha(\tau)d\tau\right) e_j\right)=A+\int_{t_2}^s\left(\alpha(\tau)-|\alpha(\tau)|\right)d\tau \leq A.
\end{equation}
Now, we  split our arguments according to the different situations in the definition of time~$t_3$.\\
\texttt{Case (c1): $t_3=T$.} From the optimality of $(y,\alpha)$, 
\begin{equation}\label{eq:22marzo_4}
0\leq J_t(x;(y_1,\alpha_1))-J_t(x;(y,\alpha))=\sum_{i=1}^7 I_i
\end{equation}
where: for $i=1,\dots,5$, the $I_i$'s are the same as those of case $(a)$ (in particular, the estimates obtained in case $(a)$ still hold true because $t_2-t_1=A/V_1$) and
\begin{equation*}
I_6=\int_{t_2}^{t_3}\left(\frac{|\alpha_1(s)|^2-|\alpha(s)|^2}{2}\right)ds,\quad 
I_7=\int_{t_2}^{t_3}\left(L(y_1(s),s)-\ell_j(y(s),s)\right)ds.
\end{equation*}
Our definition of $\alpha_3$ entails: $I_6=0$. Moreover, thanks to assumption~\eqref{eq:460} on the structure of $\ell_O$, the Lipschitz continuity of $\ell_i$ and~\eqref{eq:22marzo_3}, 
\begin{equation*}
I_7 \leq \int_{t_2}^{t_3}\left(\ell_j(y_1(s),s)-\ell_j(y(s),s)\right)ds \leq L_\ell TA.
\end{equation*}
Replacing all these estimates in~\eqref{eq:22marzo_4}, we get
\begin{equation*}
0\leq \int_t^{t_1}\left(-\frac{|\alpha (s)|}{2}+(2L_\ell T+V_1/2+2M_\ell/V_1+L_g)\right)|\alpha (s)|\car_{\{\bar \alpha(s)<-\mu\}}ds.
\end{equation*}
Hence, if $\{\bar \alpha(s)<-\mu\}\cap[t,t_1]$ has positive measure and $\mu>2(2L_\ell T+V_1/2+2M_\ell/V_1+L_g)$, then we get the desired contradiction.\\
\texttt{Case (c2): $t_3=\min\{s\in[t_2,T]\;:\; y_3(s)=y(s)\}$ with $t_3<T$.} We define our competitor on $[t_3,T]$ as the trajectory starting at $y_1(t_3)=y(t_3)$ at time~$t_3$ and obeying to the control~$\alpha_1(s)=\alpha(s)$ for $s\in[t_3,T]$. We end our proof using the same calculations of case $(c1)$.
\end{proof}

\subsection{Closed graph property}
Let us now investigate a closed graph property of the multi-valued map  $x  \rightrightarrows \Gamma^{\rm{opt}}[x]$ defined in~\eqref{eq:gamma_opt}. 
\begin{proposition}\label{prp:prop1}
Fix $x\in\cG$ and a sequence $\{x_n\}_{n\in\N}$ with $x_n\in\cG$ and $x_n\to x$ as $n\to\infty$. Consider $(y_n,\alpha_n)\in\Gamma^{\rm{opt}}[x_n]$ for any $n\in\N$. Assume that, as $n\to\infty$, $y_n$ uniformly converge to a path~$y$. Then, the $y$ belongs to $Y_{x,0}$ (defined in~\eqref{eq:2}). Moreover, there exists a measurable function~$\alpha$ such that~$(y,\alpha)$ belongs to~$\Gamma^{\rm{opt}}[x]$ defined in~\eqref{eq:gamma_opt}.
\end{proposition}
An intermediate step in the proof of Proposition \ref{prp:prop1} is  Lemma \ref{lemma1} below which deals with the approximation of admissible trajectories. The proof of Lemma \ref{lemma1} is postponed after that of Proposition \ref{prp:prop1}.

%
%VECCHIA VERSIONE DEL LEMMA \ref{lemma1}
%
%In questa vecchia versione si faceva una approssimazione differente della curva e si trovava solo una stima del costo cinetico. NON FUNZIONA pr il caso di costi finali discontinui.

%\begin{lemma}\label{lemma1}
%Fix $x\in\cG$ and $(y,\alpha)\in\Gamma[x]$; consider a sequence $\{x_n\}_{n\in\N}$ of points $x_n\in\cG$ such that $x_n\to x$ as $n\to \infty$.
%Then, there exists a sequence $\{(y_n,\alpha_n)\}_{n\in\N}$ such that, for any $n\in\N$, $(y_n,\alpha_n)\in\Gamma[x_n]$, 
%\begin{equation}\label{eq:lemma1}
%y_n \to y \textrm{ uniformly in $[0,T]$ as $n\to\infty$}\qquad\textrm{and}\qquad \|\alpha_n\|_{L^2}\leq \|\alpha\|_{L^2} +\sqrt{d(x,x_n)}.
%\end{equation}
%\end{lemma}

\begin{lemma}\label{lemma1}
Fix $x\in\cG$ and $(y,\alpha)\in\Gamma[x]$; consider a sequence $\{x_n\}_{n\in\N}$ of points $x_n\in\cG$ such that $\d_n:=d(x_n,x)\to 0$ as $n\to \infty$.
Then, there exists a sequence $\{(y_n,\alpha_n)\}_{n\in\N}$ such that, for any $n\in\N$, $(y_n,\alpha_n)\in\Gamma[x_n]$,
\begin{equation}\label{eq:lemma1}
\begin{array}{rl}
(i)&\qquad \sup_{[0,T]}d(y_n(\cdot),y(\cdot))\leq \d_n+\|\alpha\|_2\sqrt{\d_n}\qquad \textrm{with}\qquad y_n(T)=y(T)\\
(ii)&\qquad\|\alpha_n\|_2^2\leq \|\alpha\|_2^2+\d_n\left(1+\frac{\|\alpha\|_2^2}{T-\d_n}\right)\\
(iii)&\qquad \lim\limits_{n\to\infty}J_0(x_n;(y_n,\alpha_n)) = J_0(x;(y,\alpha)).
\end{array}
\end{equation}
\end{lemma}
  %y_n \to y \textrm{ uniformly in $[0,T]$ as $n\to\infty$}\qquad\textrm{and}\qquad \|\alpha_n\|_{L^2}\leq \|\alpha\|_{L^2} +\sqrt{d(x,x_n)}.
%\end{equation}
%\end{lemma}
\begin{proof}[Proof of Proposition~\ref{prp:prop1}]
Consider $x$, $x_n$, $(y_n,\alpha_n)$ and $y$ as in the statement. We wish to prove that there exists a control~$\alpha$ such that
\begin{itemize}
\item[i)] $(y,\alpha)$ belongs to~$\Gamma[x]$,
\item[ii)] $(y,\alpha)$ is optimal for $J_0$, i.e.$J_0(x,(y,\alpha))\leq J_0(x,(\hat y,\hat \alpha))$ for every $(\hat y,\hat \alpha)\in\Gamma[x]$.  
\end{itemize}
Fix any $(\hat y,\hat \alpha)\in\Gamma[x]$. Lemma~\ref{lemma1} ensures that there exists a sequence $\{(\hat y_n,\hat \alpha_n)\}_{n\in\N}$ such that $(\hat y_n,\hat \alpha_n)\in\Gamma[x_n]$ and
\begin{equation}\label{eq:prp1}
\begin{array}{c}
\hat y_n \to \hat y \textrm{ uniformly in $[0,T]$ as $n\to\infty$},\qquad\qquad \|\hat \alpha_n\|_{2}\leq \|\hat \alpha\|_{2} +o_n(1),\\
\limsup\limits_{n\to\infty}J_0(x_n;(\hat y_n,\hat \alpha_n))\leq J_0(x;(\hat y,\hat \alpha))
\end{array}
\end{equation}
where $o_n(1)$ is a sequence such that $\lim_n o_n(1)=0$.
%By the properties in~\eqref{eq:prp1}, assumption~$(H1)$ and the formula of the cost~$J$ in~\eqref{eq:46}, we deduce
%\begin{equation}\label{eq:prp2}
%\limsup_{n\to\infty}J_0(x_n;(\hat y_n,\hat \alpha_n))\leq J_0(x;(\hat y,\hat \alpha)).
%\end{equation}
On the other hand, the optimality of $(y_n,\alpha_n)$ yields
\begin{equation}\label{eq:prp3}
J_0(x_n;(y_n,\alpha_n))\leq J_0(x_n;(\hat y_n,\hat \alpha_n)).
\end{equation}
From the observations above,  we deduce that $J_0(x_n;(y_n,\alpha_n))$ are uniformly bounded and, in particular that there exists a constant~$C$, independent of~$n$, such that $\int_t^T|\alpha_n(\tau)|^2\, d\tau\leq C$. Hence, repeating the same arguments as those in the proof of Proposition~\ref{prp:ex_OT} (in particular, for obtaining ~\eqref{claim1}), we deduce that $\{\alpha_n\}_{n\in\N}$ converges to some control~$\alpha$ in the weak topology of~$L^2([0,T],\R^d)$ and $(y,\alpha)\in\Gamma[x]$. Hence, point~$i)$ is proved.\\
Taking the~$\liminf_n$ in~\eqref{eq:prp3} and using~\eqref{eq:prp1}, we also deduce $J_0(x,(y,\alpha))\leq J_0(x,(\hat y,\hat \alpha))$. Thanks to the arbitrariness of $(\hat y,\hat \alpha)\in\Gamma[x]$, we deduce point~$ii)$.
\end{proof}

\begin{proof}[Proof of Lemma~\ref{lemma1}]
Without any loss of generality, we may assume that, (possibly  after extracting a subsequence that we still denote $\{x_n\}$)  all the points~$x$ and~$x_n$ belong to the same edge (for simplicity, say $J_1$) for~$n$ sufficiently large, so  $x=\bar x e_1$, $x_n=\bar x_n e_1$ for $\bar x,\bar x_n\in\R^+$.
Indeed, if  $x=O$, we may argue edge by edge since
there are finitely many edges.
Set $\delta_n=d(x,x_n)=|\bar x-\bar x_n|$.
Let us now introduce a control $\alpha_n$ such that the corresponding path~$y_n$ is admissible (i.e. it takes its values on the network). 

Set
\begin{equation*}
\alpha_n(s)=
\left\{\begin{array}{ll}
\left\{\begin{array}{ll}
e_1&\quad\textrm{if }\bar x_n\leq \bar x\\
-e_1&\quad\textrm{if }\bar x_n> \bar x\\
\end{array}
\right\}&\quad \textrm{for }s\in[0,\delta_n]\\
\frac{T}{T-\d_n}\alpha \left((s-\d_n)\frac{T}{T-\d_n}\right)&\quad \textrm{for }s\in(\delta_n,T]\end{array}\right.
\end{equation*}
(note that here the structure $A_i=\{i\}\times \R$ plays a crucial role) and let $y_n$ start from $x_n$ and correspond to $\alpha_n$: 
\[
y_n(s)=x_n+\int_t^s \alpha_n(\tau)\, d\tau.
\]
Observe that for $s\in[0,\d_n]$,
\begin{equation*}
y_n(s)=
x_n+ (s/\delta_n) \;  (x-x_n)   %d(x,x_n)^{-1}s.
\end{equation*}
in particular, $y_n(\d_n)=x$.  From the definition of~$\alpha_n$,  we get after a change of variable that,  for $s\in[\d_n,T]$,
\begin{eqnarray}\notag
y_n(s)&=&x+\int_{\d_n}^s\frac{T}{T-\d_n}\alpha \left((\tau-\d_n)\frac{T}{T-\d_n}\right) \, d\tau
=x+\int_{0}^{(s-\d_n)\frac{T}{T-\d_n}}\alpha (\tau) \, d\tau\\ \label{eq:rescal}
&=&y\left((s-\d_n)\frac{T}{T-\d_n}\right).
\end{eqnarray}
The trajectory $(y_n,\alpha_n)$ is admissible  and
\begin{equation}\label{eq:rescal2}y_n(T)=y(T).
\end{equation}
The trajectory $y_n$ starts at $x_n$, moves with speed $1$ until it reaches the point $x$ at time~$\delta_n$ (clearly, in this time interval it always remains in the edge $J_1$)
and, from time~$\delta_n$,  becomes a time-rescaled version of  the trajectory~$y$  such that $y_n(T)=y(T)$.

Let us now estimate~$d(y(s),y_n(s))$.  
For $s\in[0,\delta_n]$, 
\begin{eqnarray*}
d(y(s),y_n(s))&\leq& d(y_n(s),x)+d(y(s),x)\leq (\delta_n-s)+\int_0^s|\bar \alpha(\tau)|\, d\tau \\
&\leq & (\delta_n-s)+\|\alpha\|_2\sqrt{s}
\end{eqnarray*}
(Cauchy-Schwarz inequality is used in the last line).\\
This and \eqref{eq:rescal} imply that for $s\in[\delta_n,T]$, 
\begin{eqnarray*}
d(y(s),y_n(s))&=& d\left(y(s),y\left((s-\d_n)\frac{T}{T-\d_n}\right)\right)
\leq \int_{(s-\d_n)\frac{T}{T-\d_n}}^s|\bar \alpha(\tau)|\, d\tau\\
&\leq& \|\alpha\|_2\sqrt{\d_n}\sqrt{\frac{T-s}{T-\d_n}}.
\end{eqnarray*}
The latter two inequalities easily imply the bound \eqref{eq:lemma1}-$(i)$.

%Let us now prove the bound in~\eqref{eq:lemma1}-$(ii)$.

Next, by definition of~$\alpha_n$, 
\begin{eqnarray*}%\label{eq:rescal3}
\|\alpha_n\|_2^2&=&\d_n+\int_{\d_n}^s\left(\frac{T}{T-\d_n}\right)^2\alpha^2 \left((\tau-\d_n)\frac{T}{T-\d_n}\right) \, d\tau= 
\d_n+\int_{0}^T\frac{T}{T-\d_n}\alpha (\tau)^2 \, d\tau\\
&=& \d_n+\|\alpha\|_2^2+\frac{\d_n}{T-\d_n}\|\alpha\|_2^2,
\end{eqnarray*}
which easily implies the bound of \eqref{eq:lemma1}-$(ii)$.

We now prove \eqref{eq:lemma1}-$(iii)$. From \eqref{eq:rescal2}, $g(y_n(T))=g(y(T))$. Hence,
\begin{equation}\label{eq:rescal4}
J_0(x_n;(y_n,\alpha_n))-J_0(x;(y,\alpha))= \sum_{i=1}^4I_i
\end{equation}
where
\begin{eqnarray*}
I_1&=&\int_0^{\d_n}\left[\sum_{i=1}^N \ell_i(y_n(\tau),\tau)\car_{y_n(\tau)\in J_i\setminus\{O\}} +\ell_O(\tau)\car_{y_n(\tau)=O}\right] d\tau,\\
I_2&=& \frac{\|\alpha_n\|_2^2-\|\alpha\|_2^2}{2},\\
I_3&=&\int_{\d_n}^T\left[\sum_{i=1}^N \ell_i(y_n(\tau),\tau)\car_{y_n(\tau)\in J_i\setminus\{O\}} +\ell_O(\tau)\car_{y_n(\tau)=O}\right] d\tau,\\
I_4&=&-\int_0^{T}\left[\sum_{i=1}^N \ell_i(y(\tau),\tau)\car_{y(\tau)\in J_i\setminus\{O\}} +\ell_O(\tau)\car_{y(\tau)=O}\right] d\tau.
\end{eqnarray*}
The boundedness of $\ell_i$ implies
\begin{equation*}
|I_1|\leq K\d_n,
\end{equation*}
for $K=\sum_1^N\|\ell_i\|_\infty+\|\ell_*\|_\infty$.
On the other hand, \eqref{eq:lemma1}-(ii) entails
\begin{equation*}
|I_2|\leq \frac{\d_n}{2}\left(1+\frac{\|\alpha\|_2^2}{T-\d_n}\right).
\end{equation*}
The definition of $\alpha_n$ and \eqref{eq:rescal} 
yield
\begin{multline*}
I_3=\int_{\d_n}^T\left[\sum_{i=1}^N \ell_i\left(y\left((\tau-\d_n)\frac{T}{T-\d_n}\right),\tau\right)\car_{y\left((\tau-\d_n)\frac{T}{T-\d_n}\right)\in J_i\setminus\{O\}}\right.\\
\left. +\ell_O(\tau)\car_{y\left((\tau-\d_n)\frac{T}{T-\d_n}\right)=O}\right] d\tau,
\end{multline*}
%We note that $\frac{T}{T-\d_n}\alpha\left((s-\d_n)\frac{T}{T-\d_n}\right)\in A_i$ if and only if $\alpha\left((s-\d_n)\frac{T}{T-\d_n}\right)\in A_i$.
which becomes after a change of variable,
\begin{multline*}
I_3=\int_{0}^T\left[\sum_{i=1}^N \ell_i\left(y(\theta),\frac{T-\d_n}{T}\theta+\d_n\right)\car_{y\left(\theta\right)\in J_i\setminus\{O\}}\right.\\
\left. +\ell_O\left(\frac{T-\d_n}{T}\theta+\d_n\right)\car_{y\left(\theta\right)=O}\right] \left(1-\frac{\d_n}{T}\right)d\theta.
\end{multline*}
Let $\cG'$ a bounded subset of $\cG$ such that $y(s)$ belongs to $\cG'$ for all $s\in[0,T]$ and let $\omega$ be a common modulus of continuity of the $\ell_i$ in~$J_i\cap \cG'$. The latter observation  and the definition of $I_4$ yield
\begin{equation*}
|I_3+I_4|\leq \int_0^T (N+1)\omega \left(\d_n\frac{T-\theta}{T}\right)\, d\theta+\d_n(N+1)K\leq (N+1)T \omega(\d_n).
\end{equation*}
Combining all the estimates  with \eqref{eq:rescal4} and taking the $\limsup$, we complete the proof of \eqref{eq:lemma1}-(iii).
\end{proof}

\subsection{Lipschitz continuity  of the value function}
We investigate the Lipschitz continuity of the value function $u$.  We will see below that special assumptions will be needed for it to hold up to the horizon  $T$.

\begin{proposition}\label{prop:loclip}
Under the same assumption as in Theorem~\ref{optcur_lip}, the value function is locally Lipschitz continuous in $\cG\times [0,T)$. 
\end{proposition}
\begin{remark}
Note that in contrast with Remark \ref{rmk:U_holder}, we do not suppose that the costs $\ell_i$ are Lipschitz continuous with respect to time.
\end{remark}
\begin{proof}
The proof borrows some ideas of \cite[Proposition 4.1]{CCC1} and is split into several steps. 
For brevity, we set
\begin{equation}
 \label{eq:cGdelta}   \cG_{\delta}=\{ x\in \cG:  d(x,O)\le \delta\}.
\end{equation}

\texttt{Step 1.} We first prove that $u(\cdot,t)$ is locally Lipschitz continuous in $J_i\setminus\{O\}$ locally uniformly with respect to $t\in[0,T)$. More precisely, having fixed $T_1\in(0,T)$ 
and $R>0$,
%a compact $\cG'\subset\cG$, 
we wish to prove that  for any $t\in[0,T_1]$,
$x_0= \bar x_0 e_i \in J_i \cap \cG_R \setminus\{O\}$ 
and $r>0$ sufficiently small, the function~$u(\cdot,t)$ is Lipschitz continuous on $(\bar x_0-r,\bar x_0+r)e_i$ with a Lipschitz constant which depends only on the parameters of the problem and on $T_1$ and $R$ (it is independent of $\bar x_0$, $r$, $t$ and $i$).

For that, 
%fix $T_1\in(0,T)$, $t\in[0,T_1]$, $\bar x_0\in(0,\infty)$ with $\bar x_0 e_i\in \cG'$, and 
fix  some $r$, $0<r< \bar x_0/4$. Observe that $(\bar x_0-4r,\bar x_0+4r)e_i\subset J_i\setminus\{O\}$. Consider $\bar x,\bar x_1\in (\bar x_0-r,\bar x_0+r)$, with $\bar x\ne\bar x_1$ and $|\bar x-\bar x_1|\leq 2(T-t)V$, where $V$ is the constant found in Theorem~\ref{optcur_lip} for the set $\cG_{5R/4} \times[0,T_1]$. Set $x=\bar x e_i$, $x_1=\bar x_1 e_i$ and $\tau=\frac{|\bar x-\bar x_1|}{ 2 V}$. For $(y,\alpha)\in \Gamma^{\rm{opt}}_t[x]$, let $(y_1, \alpha_1)$ be the trajectory starting at $x_1$ at time~$t$ and associated  to the control
\begin{equation*}
\alpha_1(s)=\alpha(s)+(x-x_1)/\tau \quad \textrm{for }s\in[t,t+\tau],\qquad
\alpha_1(s)=\alpha(s) \quad \textrm{for }s\in(t+\tau,T].
\end{equation*}
From Theorem~\ref{optcur_lip}, $y$ does not reach the origin~$O$ before time $t+\frac{3r}{ V}$.
On the other hand, $\tau\leq \frac{r}{V}$.
%$\tau\leq \frac{r}{2V}$. 
Hence, in the time interval $(t,t+\tau)$, $y$ stays in $J_i\setminus\{O\}$. 

It is clear that  $y(\cdot)= y_1(\cdot)$ in $(t+\tau,T]$. We claim that
\begin{itemize}
\item[(i)] $(y_1, \alpha_1)\in \Gamma_t[x_1]$.
%\item[(i)] $y(\cdot)= y_1(\cdot)$ in $(t+\tau,T]$
\item[(ii)] $d(y(\cdot),y_1(\cdot))\leq |\bar x-\bar x_1|$ in $(t,T]$
\end{itemize}
Let us prove $(i)$.  From the observation above, it is enough  to prove that $y_1(s)\in J_i$ for $s\in[t,t+\tau]$. We observe that
\begin{eqnarray*}
d(y_1(s),x_0)&\leq& d(y_1(s),x_1)+|\bar x_1-\bar x_0|\leq \int_t^s\left|\alpha(\theta)+\frac{x-x_1}{\tau}\right|d\theta +r\\
&\leq& \int_t^s|\alpha(\theta)|d\theta +\frac{s-t}{\tau}|\bar x-\bar x_1| +r\leq V(s-t) +|\bar x-\bar x_1| +r\\
&\leq& 4r,
\end{eqnarray*}
where the last inequality is due to our choice of $\tau$. 
The inequality found above yields that  $y_1(s)\in J_i$ for $s\in[t,t+\tau]$, then (i).

Let us now prove $(ii)$.
For $s\in (t+\tau,T]$, (ii) is obvious.  For $s\in (t,t+\tau]$, there holds
\begin{equation*}
d(y(s), y_1(s))=\left|(\bar x-\bar x_1)-\int_t^s\frac{\bar x-\bar x_1}{\tau}ds\right|=|(\bar x-\bar x_1)\frac{\tau-(s-t)}{\tau}|\leq |\bar x-\bar x_1|.
\end{equation*}
The claims $(i)$ and  $(ii)$ are proved.

By definition of $u$, and recalling that in the interval $(t,t+\tau)$ both $y$ and $y_1$ stay in $J_i\setminus\{O\}$,
\begin{equation}\label{eq:29dic_1}
u(x_1,t)-u(x,t)\leq \int_t^{t+\tau}\left(\frac{|\alpha_1(s)|^2}{2}-\frac{|\alpha(s)|^2}{2}+\ell_i(y_1(s),s)-\ell_i(y(s),s)\right)\, ds.
\end{equation}
The definition of $\alpha_1$ and Theorem~\ref{optcur_lip} imply that
\begin{eqnarray*}
\int_t^{t+\tau}\left(\frac{|\alpha_1(s)|^2}{2}-\frac{|\alpha(s)|^2}{2}\right)\, ds&\leq& \frac12\int_t^{t+\tau}\left(\frac{|\bar x-\bar x_1|^2}{\tau^2}+2\frac{|\bar x-\bar x_1||\alpha(s)|}{\tau}\right)\, ds\\
&\leq& \frac12 \frac{|\bar x-\bar x_1|^2}{\tau} +|\bar x-\bar x_1|V \\
&\leq&  2V|\bar x-\bar x_1|, %3V
\end{eqnarray*}
where the last inequality is due to the  choice of $\tau$.
On the other hand, assumption~\eqref{eq:9999} and point~$(ii)$ entail
\begin{equation*}
\int_t^{t+\tau}\left(\ell_i(y_1(s),s)-\ell_i(y(s),s)\right)\, ds\leq L_\ell|\bar x-\bar x_1|\tau =\frac{L_\ell |\bar x-\bar x_1|^2 }{2V} \le \frac{L_\ell r }{V} |\bar x-\bar x_1| 
\le \frac{L_\ell R }{4V}  |\bar x-\bar x_1| 
\end{equation*}
because $|\bar x-\bar x_1| \le 2r \le \bar x_0 /2 \le R/2 $.
%because {\color{red} https://www.overleaf.com/project/647f26ffe5e93a290d76a2bc$\tau=\frac{|\bar x-\bar x_1|}{4V}\leq \frac{r}{2V}\leq \frac{R}{8V}$.} 
The latter two inequalities and  \eqref{eq:29dic_1} yield
\begin{equation*}
u(x_1,t)-u(x,t)\leq \left(2V+\frac{L_\ell R}{4 V}\right)|\bar x-\bar x_1|.
\end{equation*}
Reversing the role of $x$ and $x_1$, we obtain the desired Lipschitz
continuity with constant $2V+ L_\ell R /4V$, and complete {\texttt Step 1}.

\texttt{Step 2.} We observe that the Lipschitz constant found in {\texttt Step 1} is independent of $\bar x_0$, provided that $\bar x_0\in \cG_R$. Hence, $u(\cdot,t)$ is Lipschitz continuous in $(\cG_R\cap J_i)\setminus\{O\}$ with the same Lipschitz constant as above.

\texttt{Step 3.} By the continuity of $u$ (see Proposition~\ref{prp:U_cont}), $u(\cdot, t)$ is Lipschitz continuous in $\cG_R$ with Lipschitz constant $2V+L_\ell R / 4V$. Note that this Lipschitz constant depends implicitly on $T_1$ through $V$.

\texttt{Step 4.} We now prove the Lipschitz continuity in time of $u$ for $t\in[0,T_1]$. 
Consider $x\in\cG_R$ and $t_1,t_2\in[0,T_1]$. Without loss of generality, we may assume that $t_1\leq t_2$.

Consider $(y,\alpha)\in \Gamma^{\rm{opt}}_{t_1}[x]$.
  Observe that  $y(t_2) \in \cG_{R+V T}$. Let $W\ge V$ be the constant found in Theorem ~\ref{optcur_lip} for the set $\cG _{5 (R+V T)/4} \times [0, T_1]$.  Obviously,
\begin{eqnarray*}
|u(x,t_2)-u(x,t_1)|\leq |u(x,t_2)-u(y(t_2),t_2)|+|u(y(t_2),t_2)-u(x,t_1)|.
\end{eqnarray*}
From \texttt{Step 3} and  Theorem~\ref{optcur_lip}, 
\begin{eqnarray*}
|u(x,t_2)-u(y(t_2),t_2)|&\leq&  \left (2W+\frac {L_\ell (R+VT) } {4W}\right) d(x,y(t_2))\leq (2W+ \frac {L_\ell (R+VT)} {4W})\int_{t_1}^{t_2}|\alpha(s)|\, ds\\
&\leq&  \left(2VW+ \frac {L_\ell (R+VT)}4 \right)|t_2-t_1|.
\end{eqnarray*}
On the other hand, the Dynamic Programming Principle (see Proposition~\ref{prp:prop_vf}) ensures that
\begin{equation*}
|u(y(t_2),t_2)-u(x,t_1)|\leq\int_{t_1}^{t_2}\left(\frac{|\alpha(s)|^2}{2}+|L(y(s),s)|\right)\, ds\leq \left(\frac{V^2}{2}+M_\ell\right)|t_2-t_1|.
\end{equation*}
From the latter three inequalities, we deduce that
\begin{equation*}
|u(x,t_2)-u(x,t_1)|\leq
\left(2VW+ \frac {V^2} 2  + \frac {L_\ell (R+VT)}4 +M_\ell \right)|t_2-t_1|.
\end{equation*}
Hence,  \texttt{Step 4} is done.\\
\texttt{Step 5.}  We achieve the proof by combining the results obtained in steps 3 and 4.
 %\texttt{Step 3} and \texttt{Step 4} .
% ensure: for all $x_1,x_2\in {\color{red} \cG_R}$, $0\leq t_1,t_2\leq T_1$, there holds
%\begin{equation*}
%|u(x_1,t_1)-u(x_2,t_2)|\leq {\color{red} \left(3V+\frac{ L_\ell R %}{8V}\right)|x_2-x_1|+\left(\frac{7V^2}{2}+ \frac {L_\ell R} 8 +M_\ell%\right)|t_2-t_1|}
%\end{equation*}
%which is equivalent to our statement.
\end{proof}

If furthermore the terminal cost $g$ is continuous on $\cG$, then the   Lipschitz continuity of  $u$ w.r.t. $(x,t)$   holds  locally in $x$ and globally in  $t\in [0,T]$:
\begin{corollary}\label{cor:ulipT}
Under the assumptions of Theorem~\ref{optcur_lipT}, the value function~$u$ is locally Lipschitz continuous in~$\cG\times[0,T]$.
\end{corollary}
\begin{proof}
Since $u$ is continuous on $\cG \times [0,T]$, it is enough to repeat the 
proof of Proposition \ref{prop:loclip} using Theorem \ref{optcur_lipT} instead of Theorem~\ref{optcur_lip}.
\end{proof}

The following proposition, which will not be used in the remaining part of the paper, addresses
the local Lipschitz continuity of the value function   with respect to $x$  up to the horizon $T$,  provided that the terminal cost $g$ is Lipschitz continuous on $\cG$ and the running costs $\ell_i$ are Lipschitz continuous w.r.t. $x$, but without assuming $C^2$ continuity of the costs in $J_i\setminus \{O\}$. Note that its proof does not rely on the optimality conditions stated in Lemmas  \ref{lemma:EL} and \ref{lemma:trasver}, in contrast with Corollary \ref{cor:ulipT}.  
\begin{proposition}\label{prop:loclipT}
If the terminal cost $g$ is Lipschitz continuous in~$\cG$ with Lipschitz constant~$L_g$ and the costs $\ell_i$  are bounded ($\|\ell_i\|_\infty\leq M_\ell$) and Lipschitz continuous in~$x$ with Lipschitz constant~$L_\ell$, then, the value function is locally Lipschitz continuous with respect to~$x$ in $\cG\times [0,T]$. 
\end{proposition}
\begin{proof}
For what follows, let us fix $v$  an arbitrary positive constant.

There is no loss of generality in assuming that  $x_1$ and $x_2$ belong to the same edge, say $J_i$, i.e. $x_1=\bar x_1 e_i$ and $x_2=\bar x_2 e_i$. From Remark~\ref{rmk:bound_contr}, there exists $C>0$ such that  for every $(y,\alpha)\in \Gamma^{\rm{opt}}_t[x]$,  $\|\alpha\|_2\leq C$ and $y$ is $1/2$-H\"older continuous with H\"older constant $C$. Let us  distinguish  several cases.

\texttt{Case 1: $x_1,x_2\in J_i\setminus\{O\}$ with $ \bar x_1, \bar x_2\geq C(T-t)^{1/2}$.} Consider $(y_2,\alpha_2)\in \Gamma^{\rm{opt}}_t[x_2]$. Since $d(y_2(s),x_2)\leq C(T-t)^{1/2}$ for every $s\in[t,T]$, the control~$\alpha_2$ is also admissible  for $(x_1,t)$ because $d(x_1,O)\geq C(T-t)^{1/2}$. Let $y_1$ be the path starting from~$x_1$ at time~$t$ and associated to the control ~$\alpha_2$. For $s\in[t,T]$, both $y_2(s)$ and $y_1(s)$ belong to $J_i$, and $d(y_2(s),y_1(s))=d(x_2,x_1)=|\bar x_2-\bar x_1|$. By definition of~$u$, there holds
\begin{eqnarray*}
u(x_1,t)-u(x_2,t)&\leq&\int_t^T\left|\ell_i(y_1(s),s)-\ell_i(y_2(s),s)\right|ds + \left|g(y_1(T))-g(y_2(T))\right|\\
&\leq& (L_\ell T+L_g)\; d(x_1,x_2).
\end{eqnarray*}
The proof is completed by reversing the roles of $x_1$ and  $x_2$.

\texttt{Case 2: $x_1,x_2\in J_i$ with $\bar x_1\leq 2C(T-t)^{1/2}$ and  $x_2=O$.} For $(y_2,\alpha_2)\in \Gamma^{\rm{opt}}_t[x_2]$, set
\begin{equation*}
\alpha_1(s)=-\max\{v,|\alpha_2(s)|\} e_i,\qquad \textrm{for }s\in[t,t_*],
\end{equation*}
where $v$ is the constant fixed above.
Let $y_1$ be the path defined on $[t,t_*]$ such that $y_1(t)=x_1$ and corresponding to the control $\alpha_1$. The time $t_*$ is defined by 
\begin{equation*}
t_*=\min\Bigl\{T,\: \: \: \min\Bigl\{s\in[t,T]\,:\,y_1(s)=y_2(s) \Bigr\}, \:\:\: \min\Bigl\{s\in[t,T]\,:\, y_1(s)=O\Bigr\}\Bigr\}.
\end{equation*}
Then 
\begin{equation}\label{eq:2feb_1}
t_*-t\leq \frac {\bar x_1} v=\frac {d(x_1,x_2)}v.
\end{equation}
because  $y_1(s)\not =O $ for $s\in [t,t_*)$ and $\alpha_1(s)\cdot e_i\leq -v$ for  $s\in [t,t_*]$.

The definition of $\alpha_2$ also implies that
%Moreover, in $[t,t_*]$, there hold: $|\alpha_1(s)|^2\leq 1+|\alpha_2(s)|^2$ and
\begin{equation}\label{eq:2feb_2}
d(y_1(s),y_2(s))\leq d(y_1(s),O)+d(y_2(s),O)\leq \bar x_1-\int_t^s|\alpha_2(\tau)|d\tau+\int_t^s|\alpha_2(\tau)|d\tau=d(x_1,x_2),
\end{equation}
for $s\in [t,t_*]$. Again from \eqref{eq:2feb_1}, 
\begin{equation}\label{eq:2feb_4}
\int_t^{t_*}\left[\frac{|\alpha_1(s)|^2}{2}-\frac{|\alpha_2(s)|^2}{2}+L(y_1(s),s)-L(y_2(s),s)\right]ds\leq \left[\frac {v^2} 2+2M_\ell\right]\frac {d(x_1,x_2)}{v}.
\end{equation}
The following arguments will differ  according to the value of $t_*$.
\\
\texttt{Subcase 2-a: $t_*=T$.} From \eqref{eq:2feb_2} and the Lipschitz continuity of $g$,
\begin{equation*}
g(y_1(T))-g(y_2(T))\leq L_g d(x_1,x_2).
\end{equation*}
This inequality and \eqref{eq:2feb_4} yield 
\begin{equation*}
u(x_1,t)-u(x_2,t)\leq \left(\frac {4M_\ell+v^2} {2v}+L_g\right) 
d(x_1,x_2) .
\end{equation*}
\texttt{Subcase 2-b: $t_*=\min\{s\in[t,T]\;:\;y_1(s)=y_2(s) \}<T$.} In this case, set $\alpha_1(s)=\alpha_2(s)$ for $s\in (t_*,T]$. Clearly, $y_1(s)=y_2(s)$ for $s\in (t_*,T]$. This and \eqref{eq:2feb_4} imply
\begin{equation*}
u(x_1,t)-u(x_2,t)\leq (2M_\ell+v^2/2)  \frac {| \bar x_2-\bar x_1|} v.
\end{equation*}

\texttt{Subcase 2-c: $t_*=\min\{s\in[t,T]:\, y_1(s)=O\}<\min\{T,\min\{s\in[t,T]:\, y_1(s)=y_2(s) \}\}$.} 
Then,  $y_2(t_*)$ belongs to some $J_j\setminus\{O\}$ with $j\ne i$ and $y_1(t_*)=O$. Indeed, should 
$y_2(t_*)$ belong to $J_i\setminus\{O\}$, then there would exist a time $\tau\in (t, t_*) $ such that 
$y_1(\tau)=y_2(\tau)$, in contradiction with the definition of $t_*$, and $y_1(t_*)=y_2(t_*)=O$ has been addressed in \texttt{Subcase 2-b}.

%{\color{red} I have a problem with this  : isn'it the opppsite? } {\color{green}CLA: I think that at time $t_*$ we have $y_1(t_*)=O$ and, in the time interval $[t,t_*]$ the path $y_1$ moved from $\bar x_1 e_i$ to $O$ ``spanning'' all the interval $[0,\bar x_1]e_i$. On the other hand, we have $y_2(t)=O$. If $y_2(t_*)\in J_i\setminus\{O\}$ then we must have $y_1(\theta)=y_2(\theta)$ for some $\theta \in(t,t_*)$; hence, we are in subcase 2-b. If $y_2(t_*)=O$, then we are in subcase 2-b. Hence, for $y_2(t_*)\in J_i$, we have a contradiction of the strict inequality in the definition of subcase 2-c.} 
Let us  define $(y_1,\alpha_1)$ by 
\begin{equation*}
\alpha_1(s) = |\alpha_2(s)|e_j \qquad \textrm{for }t\in (t_*,t_{**}] ,
\end{equation*}
where $t_{**}=\min\Bigl\{T,\min\Bigl\{s\in(t_*,T]: y_1(s)=y_2(s)\Bigr\}\Bigr\}$. Note that, in~$[t_*,t_{**}]$, both $y_1(\cdot)$ and $y_2(\cdot)$ belong to~$J_i$ with $y_2(\cdot)\ne O$. 
Here again, the arguments differ according to the cases 
 in the definition of~$t_{**}$.\\
\texttt{Subcase 2-c1: $t_{**}=T$.} For $s\in[t_*,t_{**}]$, there holds 
\begin{eqnarray}\notag
d(y_1(s),y_2(s))&=&y_2(t_*)\cdot e_j+\int_{t_*}^s \alpha_2(\tau)\cdot e_j d\tau-\int_{t_*}^s |\alpha_2(\tau)| d\tau\\ \label{eq:2feb_5}
&\leq& y_2(t_*)\cdot e_j= d(y_2(t_*),y_1(t_*))\leq d(x_1,x_2),
\end{eqnarray}
the last inequality stemming from \eqref{eq:2feb_2}. Taking into account estimate~\eqref{eq:2feb_4}, we get
\begin{eqnarray*}
u(x_1,t)-u(x_2,t)&\leq& (2M_\ell+\frac {v^2} 2)d(x_1,x_2)+\int_{t_*}^{t_{**}}\left[L(y_1(s),s)-L(y_2(s),s)\right]ds\\
&&\qquad  +g(y_1(T))-g(y_2(T))\\
&\leq &(L_g+2M_\ell+\frac {v^2} 2)d(x_1,x_2)+\int_{t_*}^{t_{**}}\left[\ell_j(y_1(s),s)-\ell_j(y_2(s),s)\right]ds
\end{eqnarray*}
where the last inequality is due to the Lipschitz continuity of~$g$ and~\eqref{eq:2feb_5}. Then the Lipschitz continuity of~$\ell_j$ and  \eqref{eq:2feb_5} again lead to
\begin{equation}\label{eq:2feb_6}
u(x_1,t)-u(x_2,t)\leq(TL_\ell+L_g+2M_\ell+\frac {v^2} 2)|\bar x_2-\bar x_1|.
\end{equation}
\texttt{Subcase 2-c2: $t_{**}<T$.} Hence, $y_1(t_{**})=y_2(t_{**})$. Set $(y_1,\alpha_1)=(y_2,\alpha_2)$   on $(t_{**},T]$.
The same calculations as in \texttt{Subcase 2-c1} yield  \eqref{eq:2feb_6}.\\
\texttt{Case 3: $x_1,x_2\in J_i$ with $0<\bar x_2<\bar x_1\leq 2C(T-t)^{1/2}$.}
Consider $(y_2,\alpha_2)\in \Gamma^{\rm{opt}}_t[x_2]$ and define the path $y_1$  starting at $x_1$ at time $t$
and corresponding to the control
\begin{equation*}
\alpha_1(s)=-|\alpha_2(s)|e_i \qquad \textrm{for }s\in[t,t_*],
\end{equation*}
where
\begin{equation*}
t_*=\min\Bigl\{T,\min\Bigl\{s\in[t,T]\;:\; y_2(s)=O\Bigr\},\min\Bigl\{s\in[t,T]\;:\; y_1(s)=y_2(s)\Bigr\}\Bigr\}.
\end{equation*}
Observe that, for  $s\in [t,t_*)$, both $y_1(s)$ and $y_2(s)$ belong to $J_i\setminus\{O\}$, and
\begin{equation}\label{eq:2feb_7}
d(y_1(s),y_2(s))\leq \bar x_1-\int_t^s|\alpha_2(\tau)|d\tau-\bar x_2-\int_t^s\alpha_2(\tau)\cdot e_id\tau\leq \bar x_1-\bar x_2=d(x_1,x_2).
\end{equation}
This implies
\begin{equation}\label{eq:2feb_8}
\int_t^{t_*}\left[\frac{|\alpha_1(s)|^2}{2}-\frac{|\alpha_2(s)|^2}{2}+L(y_1(s),s)-L(y_2(s),s)\right]ds\leq L_\ell Td(x_1,x_2).
\end{equation}
Let us argue differently according to the cases in the definition of $t_*$.\\
\texttt{Subcase 3-a: $t_*=T$.} Arguing as in \texttt{Subcase 2-a} and using \eqref{eq:2feb_7}-\eqref{eq:2feb_8} leads to the desired result.\\
\texttt{Subcase 3-b: $t_*=\min\Bigl\{s\in[t,T]\;:\; y_2(s)=O\Bigr\}<T$.} Combining the conclusions in \texttt{Case 2} and \eqref{eq:2feb_7}-\eqref{eq:2feb_8} leads to the desired result.\\
\texttt{Subcase 3-c: $t_*=\min\Bigl \{s\in[t,T]\;:\; y_1(s)=y_2(s)\Bigr\}<T$.} The conclusion follows by setting  $(y_1,\alpha_1)=(y_1,\alpha_1)$ on $(t_*,T]$. \\
\texttt{Case 4: $x_1,x_2\in J_i$ with $0<\bar x_1<\bar x_2\leq 2C(T-t)^{1/2}$.} Consider $(y_2,\alpha_2)\in \Gamma^{\rm{opt}}_t[x_2]$ and the trajectory $(y_1,\alpha_1)$ such that  $\alpha_1(s)=|\alpha_2(s)|e_i$ on $[t,t_*]$, where
\begin{equation*}
t_*=\min\Bigl\{T,\min\Bigl\{s\in[t,T]\;:\; y_1(s)=y_2(s)\Bigr\}\Bigr\}.
\end{equation*}
Note that, in $[t,t_*]$, $\alpha_1(s)\cdot e_i\geq 0$. Hence, $y_2$ cannot hit the vertex $O$ before crossing $y_1$. For $s\in[t,t_*]$ and
\begin{equation}\label{eq:2feb_9}
d(y_1(s),y_2(s))=x_2+\int_t^s\alpha_2(\tau)d\tau-x_1-\int_t^s|\alpha_2(\tau)|d\tau\leq x_2-x_1=d(x_1,x_2).
\end{equation}
This implies
\begin{equation}\label{eq:2feb_10}
\int_t^{t_*}\left[\frac{|\alpha_1(s)|^2}{2}-\frac{|\alpha_2(s)|^2}{2}+L(y_1(s),s)-L(y_2(s),s)\right]ds\leq L_\ell Td(x_1,x_2).
\end{equation}
The  arguments differ according to the cases in the definition of $t_*$.\\
\texttt{Subcase 4-a: $t_*=T$.} Arguing as in \texttt{Subcase 2-a} and using by \eqref{eq:2feb_9}-\eqref{eq:2feb_10} yields the desired result.\\
\texttt{Subcase 4-b: $t_*=\min\Bigl\{s\in[t,T]\;:\; y_1(s)=y_2(s)\Bigr\}$.} The result follows from the same arguments as in \texttt{Subcase 3-c} using ~\eqref{eq:2feb_9}-\eqref{eq:2feb_10}. The proof is complete.
\end{proof}
%\begin{corollary}\label{cor:ulipT}
%Under the assumption of Theorem~\ref{optcur_lipT}, the value function~$u$ is %locally Lipschitz continuous in~$\cG\times[0,T]$.
%\end{corollary}
%\begin{proof}
%The Lipschitz continuity of~$u$ with respect to~$x$ follows from %Proposition~\ref{prop:loclipT}. The remainder of the proof 
%consists of following  \texttt{Step~4} and  \texttt{Step~5} in the proof of %Proposition~\ref{prop:loclip} and
%invoking Theorem~\ref{optcur_lipT} instead of Theorem~\ref{optcur_lip}.
%\end{proof}

\subsection{Local semi-concavity of the value function away from the vertex}
Here, we wish to prove that the value function~$u$ is semi-concave with respect to~$x$ with a linear modulus of semi-concavity, locally in $J_i\setminus \{O\}$ and for $t$ bounded away from the horizon $T$. For the definition  of semi-concavity and the main related properties, we refer the reader to the monograph~\cite{CS}. 
\begin{proposition}\label{prop:semiconc_loc}
We keep the assumptions of Theorem~\ref{optcur_lip} and assume furthermore that for all $i$ and $s$, $\ell_i(\cdot, s)\in C^{1,1}(J_i)$ and that $\|\partial_{xx}^2 \ell_i\|_{L^\infty(J_i\times [0,T])}<\infty$. Consider $t\in[0,T)$ and $x,y\in J_i\setminus\{O\}$ with $0<r\leq |x|,|y|\leq R$. Under the same assumptions as in Theorem~\ref{optcur_lip}, there exists a constant $C$ (depending on $r$, $R$ and on $T-t$) such that
\begin{equation*}
\lambda u(x,t)+(1-\lambda)u(y,t)-u(\lambda x+(1-\lambda),t)\leq C\lambda(1-\lambda)|x-y|^2, \qquad \forall\lambda \in[0,1].
\end{equation*}
\end{proposition}
The main technical part of the proof of Proposition \ref{prop:semiconc_loc} makes use of the following lemma. Recall that by \eqref{EL_giugno}  $\alpha(t)$ is well defined.
\begin{lemma}\label{lemma:CC2_3.1}
 Consider $x\in J_i\setminus\{O\}$ for some $i=1,\dots,N$, $t\in[0,T)$ and $(y,\alpha)\in\Gamma^{\rm{opt}}_t[x]$. Set $x=\bar x e_i$. Under the same assumptions as in~Proposition \ref{prop:semiconc_loc},  there exists a constant $C$ (depending on  $|\bar x|$ and on $T-t$) such that
\begin{equation}\label{eq:SCconc1}
u(x+h,t)-u(x,t)+\alpha(t)\cdot h\leq C|h|^2
\end{equation}
for any $h=\bar h e_i$ with $|\bar h|$ sufficiently small.
\end{lemma}
\begin{proof}[Proof of Proposition~\ref{prop:semiconc_loc}]
Our arguments are reminiscent of \cite[Corollary 3.2]{CCC2}.  Lemma~\ref{lemma:CC2_3.1} implies that there exists a constant $C$ (depending on $r$, $R$ and on $T-t$) such that
\begin{equation*}
\frac12 u(x+h,t) + \frac12 u(x-h,t) - u(x,t)\leq C|h|^2, \qquad\forall h,\, |h|\leq |x|.
\end{equation*}
Since $u$ is continuous (see Proposition~\ref{prp:U_cont}),  the latter inequality is equivalent to \eqref{eq:SCconc1}, see \cite[Theorem 2.1.10]{CS}.
\end{proof}
\begin{proof}[Proof of Lemma \ref{lemma:CC2_3.1}]
The arguments are reminiscent of the proof of \cite[Lemma 3.1]{CCC2}.  Consider $t$, $x$, $(y,\alpha)$  as in the statement. 

Take $h=\bar h e_i$ with $|h|< \bar x/2$, and set
\begin{equation*}
t_*=\left(\frac{T-t}{2}\right)\wedge \left(\frac{\bar x}{2V}\right),
\end{equation*}
where $V$ is the constant associated to $\cG_{\bar x}$, see \eqref{eq:cGdelta}.
Consider the trajectory $(y_h,\alpha_h)$ starting at $x+h$ at time $t$ with  the control
\begin{equation}
\label{eq:SCC1}
\alpha_h(s)=\alpha(s)-h/t_*\quad\textrm{for }s\in[t,t+t_*],\quad \hbox{and} \quad \alpha_h(s)=\alpha(s) \quad\textrm{for }s\in[t+t_*,T].
\end{equation}
One easily checks that
\begin{equation}
\label{eq:SCC2}
y_h(s)=y(s)+h\frac{t_*-s+t}{t_*}\quad\textrm{for }s\in[t,t+t_*],\quad \hbox{and} \quad y_h(s)=y(s) \quad\textrm{for }s\in[t+t_*,T],
\end{equation}
and that $y_h(s)\in J_i\setminus\{O\}$ for all $s\in [t,T]$.
Therefore,
\begin{equation*}
u(x+h,t)-u(x,t)\leq \int_t^{t+t_*}\left(\frac{|\alpha_h(s)|^2-|\alpha(s)|^2}{2}+\ell_i(y_h(s),s)-\ell_i(y(s),s)\right)ds.
\end{equation*}
On the other hand, since $\alpha|_{[t,t+t_*]}\in W^{1,\infty}$, \begin{eqnarray*}
\alpha(t)\cdot h&=&-\int_t^{t+t_*}\frac{d}{ds}\left[\alpha(s)\cdot(y_h(s)-y(s))\right]\, ds\\
&=&-\int_t^{t+t_*}\left[\alpha'(s)\cdot(y_h(s)-y(s))+\alpha(s)\cdot(\alpha_h(s)-\alpha(s))\right]\, ds\\
&=&-\int_t^{t+t_*}\left[\partial_x\ell_i(y(s),s)e_i\cdot(y_h(s)-y(s))+\alpha(s)\cdot(\alpha_h(s)-\alpha(s))\right]\, ds,
\end{eqnarray*}
where the latter identity is due to Euler-Lagrange condition \eqref{EL_giugno}.

Combining the latter two observations leads to 
\begin{multline*}
u(x+h,t)-u(x,t)+\alpha(t)\cdot h\leq \int_t^{t+t_*}\frac{|\alpha_h(s)-\alpha(s)|^2}{2}\, ds+\int_t^{t+t_*}\ell_i(y_h(s),s)-\ell_i(y(s),s) ds \\ - \int_t^{t+t_*}\partial_x\ell_i(y(s),s)e_i\cdot(y_h(s)-y(s))ds.
\end{multline*}
In what follows, $C$ is a constant which may change from line to line  and depends only on $\bar x$ and $T-t$. The regularity of $\ell_i$ implies
\begin{eqnarray}\notag
u(x+h,t)-u(x,t)+\alpha(t)\cdot h&\leq& \int_t^{t+t_*}\left(\frac{|\alpha_h(s)-\alpha(s)|^2}{2}+\frac{\|\partial^2_{xx}\ell_i\|_\infty}{2}|y_h(s)-y(s)|^2\right)ds\\ \label{eq:semic}
&\leq&
C\|y_h-y\|_{W^{1,2}([t,t+t_*],\cG)}^2\\
\notag &\leq&
C |h|^2 ,
\end{eqnarray}
the last line being obtained thanks to \eqref{eq:SCC1} and  \eqref{eq:SCC2}. The desired inequality is proved.

\end{proof}

\subsection{Regularity of $u$ along optimal trajectories and optimal synthesis}
Here, we investigate some regularity properties of $u$ in the interiors of the edges. The following lemma is reminiscent of \cite[Lemma 4.9]{C}. 
\begin{lemma}\label{lemma:note4.9}
 Consider $t\in[0,T)$, $x\in J_i\setminus \{O\}$ for some $i=1,\dots,N$, $(y,\alpha)\in \Gamma^{\rm{opt}}_t[x]$ and set
\begin{equation*}
t_*=T\wedge \min \{\tau\in[t,T]\;:\; y(\tau)=O\}.
\end{equation*}
Under the same assumptions as in Proposition \ref{prop:semiconc_loc}, the following properties hold:
\begin{itemize}
\item[(i)] 
For any $s\in(t,t_*)$, $\alpha_{|(s,t_*)}$ is the unique optimal control for $u(y(s),s)$ up to time $t_*$. In other words for any $(y_1,\alpha_1)\in \Gamma^{\rm{opt}}_s[y(s)]$, $\alpha_1$ coincides with $\alpha$ in $(s,t_*)$
\item[(ii)] $\partial_x u(x,t)$ exists if and only if the set
\begin{equation*}
\mathcal{A}(x)=\left\{\alpha(t)\;:\; (y,\alpha)\in \Gamma^{\rm{opt}}_t[x] \right\}
\end{equation*}
is as singleton. Moreover, in this case, $\mathcal{A}=\Bigl\{-\partial_x u(x,t)e_i\Bigr\}$. 
\item[(iii)] For any $s\in(t,t_*)$, the function $u(\cdot,s)$ is differentiable at $y(s)$  with $\partial_x u(y(s),s)e_i=-\alpha(s)$. 
\end{itemize}
\end{lemma}
\begin{proof}
$(i)$. The arguments are similar to  the proof of \cite[Lemma 4.9-(1)]{C}, so we refer the reader to that paper for the details and focus only on the main new aspects.

For any $s\in(t,t_*)$, consider $(y_1,\alpha_1)\in \Gamma^{\rm{opt}}_s[y(s)]$ and set $t_{*,1}=T\wedge \min\{\tau\in[t,T]\;:\; y_1(\tau)=O\}$. For $0<h<(s-t)\wedge(t_*\wedge t_{*,1}-s)$, we consider the following control
\begin{equation*}
\alpha_h(\tau)=\left\{\begin{array}{ll}
\alpha(\tau)&\qquad\textrm{if }\tau \in[t,s-h]\\
\frac{y_1(s+h)-y(s-h)}{2h}&\qquad\textrm{if }\tau \in(s-h, s+h)\\
\alpha_1(\tau)&\qquad\textrm{if }\tau \in[s+h,T]
\end{array}\right.
\end{equation*}
and the corresponding trajectory~$(y_h,\alpha_h)$ which is admissible for $u(x,t)$, from the choice of $h$. Let $(y_0,\alpha_0)$ stand for the concatenation of $(y,\alpha)$ and $(y_1,\alpha_1)$ at time $s$. From Remark~\ref{rmk:concat_opt}, $(y_0,\alpha_0)\in  \Gamma^{\rm{opt}}_t[x]$. Comparing the costs associated  $(y_0,\alpha_0)$ and to $(y_h,\alpha_h)$
and letting $h$ tend to $0$ permits to prove that   $\alpha(s)=\alpha_1(s)$, see \cite{C}. Then, from Lemma~\ref{lemma:EL}, $y(\cdot)$ and $y_1(\cdot)$ satisfy the same second order differential equation with the same initial conditions: $y(s)=y_1(s)$ and $y'(s)=\alpha(s)=\alpha_1(s)=y_1'(s)$. Therefore, $y(\tau)=y_1(\tau)$ and $\alpha(\tau)=\alpha_1(\tau)$ for $\tau\in (s,t_*)$,  and  $t_*=t_{*,1}$.\\
$(ii)$. Assume that $\partial_x u(x,t)$ exists. We wish to prove that $\mathcal{A}(x)$ is a singleton.\\
Let $(y,\alpha)$ belong to $\Gamma^{\rm{opt}}_t[x]$. By the  local semi-concavity of $u$, see  Lemma \ref{lemma:CC2_3.1},
\begin{equation*}
u(x+h,t)-u(x,t)+ \alpha(t) h\leq C h^2\qquad \textrm{for } h {\textrm{ sufficiently small.}} %(al\ posto\ di \ |h|\leq |x|).}}}
\end{equation*}
Then, from \cite[Proposition3.3.4]{CS}, we infer: $-\bar \alpha(t)\in D^+u(x,t)$. Moreover, since $u(\cdot, t)$ is differentiable at $x$, $D^+u(x,t)$ is a singleton. Hence, $\mathcal{A}(x)$ is the singleton $\{-\partial_x u(x,t) e_i\}$.

\medskip

Conversely, assume that $\mathcal{A}(x)$ is a singleton. We wish to prove that $u$ is differentiable at $(x,t)$. To this end, we  claim that, if $p\in D^* u(x,t)$, then the unique solution to 
\begin{equation}\label{Clemma4.9_1}
\xi''(\tau)=\partial_x \ell_i (\xi(\tau),\tau)e_i,\qquad \xi(t)=x,\qquad \xi'(t)=-p e_i
\end{equation}
%({\it ho messo gli $e_i$ perch\'e altrimenti $\xi$ \`e 1-dimensionale mentre la $y$ \`e d-dimensionale})
is such that there exists $(y,\alpha)\in\Gamma^{\rm{opt}}_t[x]$ with $y(\tau)=\xi(\tau)$ for $0\leq \tau\leq t_{*,\xi}:=T\wedge \min\{\tau\in[t,T]\;:\; \xi(\tau)=O\}$.
\\
Before proving the claim, let us first see how to use this intermediate result to conclude: since $\mathcal{A}(x)$ is a singleton, if the claim is true, then  also $D^*u(x,t)$ is a singleton and it coincides with $\mathcal{A}(x)$. Then \cite[Proposition 3.3.4]{CS} yields that $u$ is differentiable at $(x,t)$ with $\partial u(x,t)e_i=-\alpha(t)$ for every $(y,\alpha)\in\Gamma^{\rm{opt}}_t[x]$ and the proof of (ii) is complete.
\\
There remains to prove the claim above:
 since $p\in D^* u(x,t)$, there exists a sequence $\{x_n\}_{n\in\N}$ with $x_n\to x$ and $\partial_x u(x_n,t)\to p$ as $n\to\infty$. Consider the unique solution to 
\begin{equation}\label{Clemma4.9_2}
\xi_n''(\tau)=\partial_x \ell_i (\xi_n(\tau),\tau)e_i,\qquad \xi_n(t)=x_n,\qquad \xi_n'(t)=-\partial_x u(x_n,t) e_i.
\end{equation}
Since $u$ is differentiable at $(x_n,t)$, we have 
already proved  that $\mathcal{A}(x_n)$ is the singleton $\{-\partial_x u(x_n,t)e_i\}$. On the other hand, from Lemma~\ref{lemma:EL}, any trajectory $(y_n,\alpha_n)\in\Gamma^{\rm{opt}}_t[x_n]$ satisfies~\eqref{Clemma4.9_2} on $[t,t_{*,n})$ where $t_{*,n}=T\wedge \min\{\tau\in[t,T]\;:\; y_n(\tau)=O\}$. Observe now that, from Theorem~\ref{optcur_lip}, there exists $t_{*,\textrm{min}}>t$ such that $t_{*,n}\geq t_{*,\textrm{min}}$ for any $n$. Hence, $y_n(\tau)=\xi_n(\tau)$ for $\tau\in[t,t_{*,\textrm{min}}]$. From the uniform Lipschitz continuity of optimal trajectories (see Theorem~\ref{optcur_lip}), we deduce that $y_n$ uniformly converges to $y$ as $n\to\infty$. Next, Proposition~\ref{prp:prop1} ensures that there exists a measurable function~$\alpha$ such that $(y,\alpha)\in\Gamma^{\rm{opt}}_t[x]$. Passing to the limit in~\eqref{Clemma4.9_2}, we infer that $y(\tau)=\xi(\tau)$ in $[t,t_{*,\textrm{min}}]$. The  claim is proved.

%\\
% Our assumption is that $\mathcal{A}(x)$ is a singleton. Therefore, the intermediate result that has just been proved  also $D^*u(x,t)$ is a singleton and it coincides with $\mathcal{A}(x)$. Invoking~\cite[Proposition 3.3.4]{CS}, we conclude that $u$ is differentiable at $(x,t)$ with $\partial u(x,t)e_i=-\alpha(t)$ for every $(y,\alpha)\in\Gamma^{\rm{opt}}_t[x]$.\\
$(iii)$. It is enough to combine the previous two statements (see also \cite[Remark 4.10]{C}). 
\end{proof}
\begin{corollary}\label{cor:10000}
 Consider two optimal trajectories $\gamma_i\in\Gamma^{\textrm{opt}}[x_i]$ such that $\gamma_1(t)=\gamma_2(t)\in J_k\setminus\{O\}$ for some  $t\in(0,T)$. Let $I_i$,  $i=1,2$, be the largest open interval containing $t$ such that $\gamma_i(s)\in J_k\setminus\{O\}$ for $s\in I_i$. Under the same assumptions as in Proposition~\ref{prop:semiconc_loc}, $I_1=I_2$.
\end{corollary}
\begin{proof}
There exists $\delta>0$ such that both $\gamma_1(s)$ and $\gamma_2(s)$ lie in $J_k\setminus\{O\}$ for $s\in(t-\delta,t+\delta)$. Let us prove first that $\gamma_1$ and $\gamma_2$ coincide in $ (t,t+\delta)$. For that, let $\gamma_3$ be the concatenation of $\gamma_{1|_{[0,t]}}$ and $\gamma_{2|_{[t,T]}}$. From Lemma~\ref{lemma:note4.9}-($i$), 
  $\gamma_1=\gamma_3$ in $(t-\delta, t+\delta)$. This implies that $\gamma_1=\gamma_2$ in $(t, t+\delta)$.\\
As a second step, from the latter result and Euler-Lagrange optimality condition, we deduce that $\gamma_1$ and $\gamma_2$ coincide also in $(t-\delta,t)$.\\
By a standard connexity argument, $I_1=I_2$ and $\gamma_1$ and $\gamma_2$ coincide in this interval.
\end{proof}
We now tackle the counterpart of \cite[Lemma 4.11]{C} on optimal synthesis in the time interval in which the trajectory remains in the interior of  a given edge. We first need the following definition:
\begin{definition}\label{def:traj_opt_tt1}
Consider $(x,t)\in\cG\times [0,T]$ and $t_1\in (t,T)$. We say that the trajectory $(y,\alpha)\in \Gamma_{t,t_1}[x]$ is optimal for $u(x,t)$ on the interval $(t,t_1)$ if there exists $(\tilde y,\tilde\alpha)\in\Gamma^{\rm{opt}}_t[x]$ with $(y,\alpha)=(\tilde y,\tilde \alpha)$ on $(t,t_1)$.
\end{definition}
\begin{lemma}\label{lemma:OS}
The assumptions are the same as in~Proposition \ref{prop:semiconc_loc}. Consider $t\in[0,T)$, $x\in J_i\setminus \{O\}$ for some $i=1,\dots,N$.

If $u(\cdot,t)$ is differentiable at $x$, then 
there is a unique $t_*\in (t,T]$ and a unique $y$ such that
\begin{equation}\label{eq:C_30}
y'(s)=-\partial_x u(y(s),s)\qquad \textrm{a.e. in }(t,t_*),\quad  \hbox{and} \quad y(t)=x,
\end{equation}
and
$t_*=T\wedge \min\{\tau\in[t,T]\;:\; y(\tau)=O\}$.

The trajectory $(y,y')$ is optimal for $u(x,t)$ on the interval $(t,t_*)$ in the sense of Definition~\ref{def:traj_opt_tt1}.

\end{lemma}
\begin{proof}

The first part of the statement is a consequence of Lemma~\ref{lemma:note4.9}-(ii) and -(iii).

  The arguments for proving the optimality of $y$ on $(t,t_*)$   are reminiscent  of~\cite[Lemma 4.11]{C}. Hence, we focus on the main new aspects and refer the reader to \cite{C} for the details. From Proposition~\ref{prop:loclip}, $u$ is Lipschitz continuous on each interval $[t,t_1]\subset[t,T)$. Hence, also $y$ is Lipschitz continuous on $[t,t_1]$.

The same arguments as  in the proof of~\cite[Lemma 4.11]{C} yield
\begin{equation*}
\frac{d}{ds} u(y(s),s)=-\frac12|y'(s)|^2-\ell_i(y(s),s),
\end{equation*}
for a.a. $s\in(t,t_*)$. Integrating this inequality on $(t,t_*)$ leads to
\begin{equation*}
u(x,t)=u(x,t)+\int_{t}^{t_*}\left(\frac{|y'(s)|^2}{2} +\ell_i(y(s),s)\right)\, ds.
\end{equation*}
From Remark \ref{rmk:quasi_dpp}, we infer that $(y,y')$ is optimal on $(t,t_*)$.
\end{proof}

\subsection{ The PDE satisfied by $u$ on  $\cG\times [0,T)$}

The aim of this paragraph is to prove that 
the value function is the unique viscosity solution 
(in a suitable sense that will defined) 
of  Hamilton-Jacobi equations in the network, with a 
suitable {\sl transmission} condition at the origin. 
%We follow the approach by Imbert and Monneau~\cite{IM}.

%\paragraph{The Hamilton-Jacobi equations on the junction.}  In what follows, the Hamiltonian has a different formulation according to the fact that the point coincides or not coincides with the junction~$O$. 
\subsubsection{Relaxed controls}
To start with, let us recall the definition of  the relaxed controls introduced  in \cite{ACCT}. They will be  used  to construct the Hamiltonians involved in the Hamilton-Jacobi equations on $\cG\times [0,T)$. For $x\in J_i$, $i=1,\dots,N$, set 
\begin{equation*}
FL_i(x,t)=\co\{(a,a^2/2)\;:\; a\in\R\}, \quad \hbox{and} \quad FL_i^{\downarrow}(x,t)=FL_i(x,t)\cap\{(\zeta,\xi)\in\R^2\;:\; \zeta\geq 0\}.
\end{equation*}
Here, the notation  {\sl co} stands is used for the convex hull. It can be easily checked that  \begin{equation*}
FL_i(x,t)=\{(\zeta,\xi)\in\R^2\;:\; \xi\geq \zeta^2/2\},\quad \hbox{and} \quad FL_i^{\downarrow}(x,t):=\{(\zeta,\xi)\in\R^2\;:\; \zeta\geq 0,\, \xi\geq \zeta^2/2\}.
\end{equation*}
For $x=O$, set
\begin{equation*}
FL(O,t)=\bigcup_{i=1}^N FL_i^{\downarrow}(O,t).
\end{equation*}
%\green{In questo modo gli insiemi $FL$ non dipendono da $t$ e praticamente neanche da $x$. Li scrivo comunque cos\`i per eventuali generalizzazioni.}\\

\subsubsection{Hamiltonians} For $x\in J_i$, $i=1,\dots,N$, $t\in[0,T]$, $\bar p\in\R$, $p=(p_1,\dots,p_N \in\R^N$, set
\begin{eqnarray*}
H_i(x,t,\bar p)&=&\sup_{(\zeta,\xi)\in FL_i(x,t)}\{-\bar p\zeta-\xi-\ell_i(x,t)\},\\
H_i^{\downarrow}(x,t,\bar p)&=&\sup_{(\zeta,\xi)\in FL_i^{\downarrow}(x,t)}\{-\bar p\zeta-\xi-\ell_i(x,t)\},\\
H_O(t,p)&=&\max\Bigl\{-\ell_*(t),\:\max_{i=1,\dots,N}\left\{-\ell_i(O,t)\right\},\:\max_{i=1,\dots,N}\left\{H_i^{\downarrow}(O,t,p_i)\right\}\Bigr\}\\
&=&\max\Bigl\{-\ell_O(t),\:\max_{i=1,\dots,N}\left\{H_i^{\downarrow}(O,t,p_i)\right\}\Bigr\}.
\end{eqnarray*}
Elementary calculus  yields 
\begin{eqnarray}\label{eq:1dic_easy0}
H_i(x,t,\bar p)&=&\sup_{a\in\R}\left\{-\bar p a-\frac{|a|^2}2-\ell_i(x,t)\right\}=\frac{|p|^2}2-\ell_i(x,t)\quad\forall x\in J_i,\\ \label{eq:1dic_easy}
H_i^{\downarrow}(O,t,\bar p)&=&\max_{\bar\alpha\geq 0}\{-\bar \alpha \bar p-\ell_i(O,t)-|\bar\alpha|^2/2\}=
\left\{\begin{array}{ll}
\frac{|\bar p|^2}{2}-\ell_i(O,t)&\textrm{if }\bar p\leq 0,\\
-\ell_i(O,t)&\textrm{if }\bar p> 0.
\end{array}\right.
\end{eqnarray}

\subsubsection{Hamilton-Jacobi equations on $\cG\times [0,T)$}
We are interested in the system of first-order PDEs on $\cG\times (0,T)$:
\begin{equation}\label{HJ}
\left\{\begin{array}{rcll}
-\partial_t u+ H_i(x,t,Du)&=&0,&\qquad\textrm{if } x\in J_i\setminus\{O\},\\
-\partial_t u+ H_O(t,Du)&=&0,&\qquad\textrm{if } x=O,
\end{array}\right.
\end{equation}
where $Du(x,t) $ is defined in~\eqref{eq:def_deriv} and is a $1$-dimensional 
(resp. $N$-dimensional)
object if $x\in J_i\setminus\{O\}$ (resp. $x=O$).

\subsubsection{Viscosity solution of  \eqref{HJ}}
\begin{definition}
A function $u\in C(\cG\times(0,T))$ is a viscosity subsolution (resp. supersolution) of \eqref{HJ} if for every function $\varphi\in C^1(\cG\times(0,T))$ touching~$u$ from above (resp. below) at $(x,t)\in\cG\times(0, T)$, there holds
\begin{equation*} 
\begin{array}{rcll}
-\partial_t \varphi(x,t)+ H_i(x,t,D\varphi) &\leq& 0\quad (\textrm{resp. }\geq 0)&\qquad\textrm{if } x\in J_i\setminus\{O\},\\
-\partial_t \varphi+ H_O(t,D\varphi)&\leq& 0\quad (\textrm{resp. }\geq 0)&\qquad\textrm{if } x=O.
\end{array}
\end{equation*}
A function $u\in C(\cG\times(0,T))$ is a viscosity solution of \eqref{HJ} if it is both a viscosity subsolution and a viscosity supersolution of \eqref{HJ}.
\end{definition}

\subsubsection{Main result} 
\begin{theorem}\label{thm:HJ}
Under assumptions [H0] and [H1], the value function $u$ defined in~\eqref{eq:4} is a viscosity solution of \eqref{HJ} in $\cG\times (0,T)$. Moreover, for all $x\in \cG$, $t\mapsto u(x,t)$ is continuous in $ [0,T]$ and $u(x,T)=g(x)$. %Moreover, if $g$ is uniformly continuous on~$\cG$, then~$u$ is the unique solution.
\end{theorem}
\begin{proof}
%The uniqueness part of the statement follows from~\cite[Theorem 1.5]{IM}.\\
%Let us now prove that~$u$ is a solution to~\eqref{HJ}; to this end,
We borrow some arguments from the proof of~\cite[Theorem 6.4]{IM}. Clearly, the standard theory on viscosity solutions can be applied in $\cG\setminus \{O\}$, so it suffices to focus on the origin $O$.

\texttt{Step 1: $u$ is a supersolution at $O$}. Let $\varphi \in C^1(\cG\times[0,T])$ be a function touching~$u$ from below at $(O,\bar t)$, for some $\bar t\in(0,T)$. Without loss of generality, since $u$ is bounded, we may assume that $u-\varphi$ achieves a global minimum at $(O,\bar t)$ with value $0$, i.e. $\varphi(x,t)\leq u(x,t)$ $\forall (x,t)\in\cG\times[0,T]$ and $\varphi(O,\bar t)= u(O,\bar t)$. Let $(y,\alpha)\in\Gamma_{\bar t}[O]$ be an optimal trajectory for $u(O,\bar t)$. The Dynamic Programming Principle in Proposition~\ref{prp:prop_u}-(i) and Remark~\ref{rmk:restr_OC} ensure
\begin{equation*}
u(O,\bar t)=u(y(s),s)+\int_{\bar t}^s\left[L(y(\tau),\tau)+\frac{|\alpha(\tau)|^2}{2}\right]\, d\tau,\qquad\forall s\in[\bar t, T],
\end{equation*}
which entails
\begin{equation*}
\varphi(y(s), s)-\varphi(O,\bar t)+\int_{\bar t}^s\left[L(y(\tau),\tau)+\frac{|\alpha(\tau)|^2}{2}\right]\, d\tau\leq 0,\qquad\forall s\in[\bar t, T].
\end{equation*}
With the same arguments as in \cite[Theorem 6.4 (proof)]{IM}, we deduce
\begin{equation}\label{eq:1dic_1}
\int_{\bar t}^s\left[\partial_t \varphi( y(\tau),\tau)+D \varphi(y(\tau),\tau )\cdot\alpha(\tau)+L(y(\tau),\tau)+\frac{|\alpha(\tau)|^2}{2}\right]\, d\tau\leq 0,\quad\forall s\in[\bar t, T],
\end{equation}
setting $D \varphi(\tau, y(\tau))\cdot\alpha(\tau)=0$ for a.a. $\tau\in \Bigl\{\tau\in[\bar t,T]\;:\; y(\tau)=O\Bigr\}=:{\mathcal T}_0$, which makes sense because from Stampacchia theorem, $\alpha(\tau)=0$ for a.a. $\tau\in{\mathcal T}_0$.

From the uniform bound of the optimal control in~$L^2$, see Remark~\ref{rmk:bound_contr}, there holds
\begin{equation*}
d(y(\tau),O)\leq \int_{\bar t}^\tau |\alpha(s)|\, ds\leq C(\tau-\bar t)^{1/2},\qquad \forall \tau\in[\bar t,T].
\end{equation*}
Hence, from the regularity of~$\varphi$, there exists a constant $K$ such that,  for $\psi=\varphi, \partial_t\varphi, D\varphi$,
\begin{equation}\label{eq:1dic_2}
|\psi(y(\tau),\tau)-\psi(O,\bar t)|\leq K (\tau-\bar t)^{1/2},\qquad \forall\tau\in[\bar t,T].
\end{equation}
%for $\psi=\varphi, \partial_t\varphi, D\varphi$.

It is convenient to introduce the following sets of times:
\begin{equation*}
{\mathcal T}^s_0=\{\tau\in(\bar t,s)\;:\;y(\tau)=O\}\quad \hbox{and} \quad
{\mathcal T}^s_i=\{\tau\in(\bar t,s)\;:\;y(\tau)\in J_i\setminus\{O\}\}\; \textrm{for }i=1,\dots,N.
\end{equation*}
Note that ${\mathcal T}^s_O$ is closed while if $i>0$, then ${\mathcal T}^s_i$ is open, and that $(\bar t,s)=\bigcup_{i=0}^N{\mathcal T}^s_i$. Hence \eqref{eq:1dic_1} becomes:
\begin{equation}\label{eq:1dic_3}
\sum_{i=0}^N\int_{{\mathcal T}^s_i}\xi(\tau)\, d\tau\leq 0, \qquad \forall s\in[\bar t,T],
\end{equation}
where
\begin{equation*}
\xi(\tau)=\partial_t \varphi( y(\tau),\tau)+D \varphi(y(\tau),\tau)\cdot\alpha(\tau)+L(y(\tau),\tau)+\frac{|\alpha(\tau)|^2}{2}.
\end{equation*}
In \eqref{eq:1dic_3}, let us  address separately the terms  corresponding to $i=0$ and  $i=1,\dots,N$.

\medskip

Consider $i=0$ first. From Stampacchia theorem,   $\alpha(\tau) =0$ and $L(y(\tau),\tau)=\ell_O(\tau)$  for a.a. $\tau\in {\mathcal T}^s_0$.
Hence, 
\begin{equation*}
\int_{{\mathcal T}^s_0}\xi(\tau)\, d\tau=
\int_{{\mathcal T}^s_0}\left[\partial_t \varphi(O,\tau)+ \ell_O(\tau)\right]\, d\tau\geq \int_{{\mathcal T}^s_0}\left[\partial_t \varphi(O,\bar t)+ \ell_O(\bar t)\right]\, d\tau- (s-\bar t)\o(s-\bar t),
\end{equation*}
where the inequality is due to~\eqref{eq:1dic_2} and to the continuity of~$\ell_O$, and where $\o$ is a modulus of continuity depending  on the constant $K$ in \eqref{eq:1dic_2} and on the modulus of continuity of $\ell_O$.
On the other hand, the definition of $H_O$ guarantees
\[
\ell_O(\bar t)\geq-H_O(\bar t, p)\qquad\forall p\in\R^N.
\]
The latter two observations imply that
\begin{equation}\label{eq:1dic_4}
\int_{{\mathcal T}^s_0}\xi(\tau)\, d\tau\geq |{\mathcal T}^s_0|\left(\partial_t \varphi(O,\bar t)-H_O(\bar t, D\varphi(\bar t,O))\right)- (s-\bar t)\o(s-\bar t).
\end{equation}

\medskip

Consider now $i\in\{1,\dots,N\}$. For a.a. $\tau\in{\mathcal T}^s_i$, the control $\alpha(\tau)$ has the form $\alpha(\tau)=\bar \alpha(\tau)e_i$ with $\bar \alpha(\tau)\in\R$. From \eqref{eq:1dic_2}, Remark~\ref{rmk:bound_contr} and the continuity of~$\ell_i$, there exists a modulus of continuity~$\o$ such that
\begin{eqnarray}\notag
\int_{{\mathcal T}^s_i}\xi(\tau)\, d\tau&=&
\int_{{\mathcal T}^s_i}\left[\partial_t \varphi(y(\tau),\tau)+D \varphi_{\mid J_i}(y(\tau),\tau)\bar \alpha(\tau)+\ell_i(y(\tau),\tau)+\frac{|\bar \alpha(\tau)|^2}{2}\right]\, d\tau\\\label{eq:1dic_5}
&\geq &\int_{{\mathcal T}^s_i}\left[\partial_t \varphi(O,\bar t)+D \varphi_{\mid J_i}(O,\bar t)\bar \alpha(\tau)+\ell_i(O,\bar t)+\frac{|\bar \alpha(\tau)|^2}{2}\right]\, d\tau-(s-\bar t)\o(s-\bar t).
\end{eqnarray}
Thanks to the convexity of the set $FL_i$, the same arguments as those in \cite[eq.(6.22)]{IM} (as a matter of fact, it is enough to use Jensen inequality in the present case), lead to the  existence of  $(\zeta_i,\xi_i)\in FL_i(O,\bar t)$ such that
\begin{equation*}
\begin{array}{rcl}
\ds \int_{{\mathcal T}^s_i}D \varphi_{\mid J_i}(O,\bar t)\bar\alpha(\tau)\, d\tau&=& \ds D \varphi_{\mid J_i}(O,\bar t)\int_{{\mathcal T}^s_i}\bar \alpha(\tau)\, d\tau=|{\mathcal T}^s_i|D \varphi_{\mid J_i}(O,\bar t)\zeta_i, \\
\ds \int_{{\mathcal T}^s_i}|\bar \alpha(\tau)|^2/2\, d\tau&=&\ds |{\mathcal T}^s_i| \xi_i.
\end{array}
\end{equation*}
Note that the path $y(\bar t)=O$ and that during the interval $(\bar t,s)$ may enter and exit several edges. However, if $y(s)\in J_i\setminus\{O\}$ for $s\in(t_1,t_2)$ and $y(t_1)=y(t_2)=O$, then, there holds
\begin{equation*}
\int_{t_1}^{t_2}\bar \alpha(\tau)\, d\tau=0,
\end{equation*}
and consequently
\begin{equation*}
\int_{{\mathcal T}^s_i}\bar \alpha(\tau)\, d\tau=\left\{\begin{array}{ll}
y(s)&\quad\textrm{if }y(s)\in J_i,\\
0&\quad\textrm{otherwise},
\end{array}\right.
\end{equation*}
which implies that $\zeta_i \geq 0$. Therefore,
\begin{eqnarray*}
\int_{{\mathcal T}^s_i}\left[D \varphi_{\mid J_i}(O,\bar t)\bar\alpha(\tau)+\ell_i(O,\bar t)+\frac{|\bar \alpha(\tau)|^2}{2}\right]\, d\tau &=& |{\mathcal T}^s_i|\left[D \varphi_{\mid J_i}(O,\bar t) \zeta_i+\ell_i(O,\bar t)+\xi_i\right]\\
&\geq& -|{\mathcal T}^s_i|H_i^{\downarrow}(O,\bar t, D \varphi_{\mid J_i}(O,\bar t)).
\end{eqnarray*}
The latter inequality and \eqref{eq:1dic_5} yield
\begin{eqnarray*}
\int_{{\mathcal T}^s_i}\xi(\tau)\, d\tau&\geq& |{\mathcal T}^s_i|\left[\partial_t \varphi(O,\bar t)-H_i^{\downarrow}(O,\bar t, D \varphi_{\mid J_i}(O,\bar t))\right]-(s-\bar t)\o(s-\bar t)\\
&\geq& |{\mathcal T}^s_i|\left[\partial_t \varphi(O,\bar t)-H_O(\bar t, D \varphi(O,\bar t))\right]-(s-\bar t)\o(s-\bar t).
\end{eqnarray*}

This, \eqref{eq:1dic_4} and \eqref{eq:1dic_3} then imply that
\begin{eqnarray*}
(N+1)(s-\bar t)\o(s-\bar t)&\geq&\left(\sum_{i=0}^N|{\mathcal T}^s_i|\right) |\left[\partial_t \varphi(O,\bar t)-H_O(\bar t, D \varphi(O,\bar t))\right]\\
&\geq&(s-\bar t) \left[\partial_t \varphi(O,\bar t)-H_O(\bar t, D \varphi(O,\bar t))\right],\end{eqnarray*}
the last line is obtained because $(\bar t,s)=\cup_{i=0}^N{\mathcal T}^s_i$ and  ${\mathcal T}^s_i\cap {\mathcal T}^s_j=\emptyset$ for $i\ne j$. Dividing t by $(s-\bar t)$ and letting $s$ tend to $\bar t^+$ yield
\[
-\partial_t \varphi(O,\bar t)+H_O(\bar t, D \varphi(O,\bar t))\geq 0,
\]
i.e. the  desired inequality.

\texttt{Step 2: $u$ is a subsolution at $O$}.
Let $\varphi \in C^1(\cG\times[0,T])$ be a function touching $u$ from above at $(O,\bar t)$, for some $\bar t\in(0,T)$.  As above, it may be assumed  that $\varphi(x,t)\geq u(x,t)$ $\forall (x,t)\in\cG\times[0,T]$ and $\varphi(O,\bar t)= u(O,\bar t)$. 
The Dynamic Programming Principle in Proposition \ref{prp:prop_vf} ensures that for any $s\in(\bar t,T)$ and any $(y, \alpha)\in \Gamma_{\bar t,s}[O]$:
\begin{equation*}
u(O,\bar t)\leq
u(y(s),s)+\int_{\bar t}^s \left( L(y(\tau),\tau) +\frac{|\alpha(\tau)|^2}{2} \right)d\tau.
\end{equation*}
This implies that, for any $s\in(\bar t,T)$ and any $(y, \alpha)\in \Gamma_{\bar t,s}[O]$, 
\begin{equation}\label{eq:1dic_6}
\varphi(y(s),s)-\varphi(O,\bar t)+\int_{\bar t}^s \left( L(y(\tau),\tau) +\frac{|\alpha(\tau)|^2}{2} \right)\, d\tau\geq 0.
\end{equation}
Note that \eqref{eq:1dic_6} can be written 
\begin{equation}\label{eq:1dic_7}
\varphi(y(s),s)-\varphi(O,\bar t)+\sum_{i=0}^N\int_{{\mathcal T}_i^s} \left(L(y(\tau),\tau) +\frac{|\alpha(\tau)|^2}{2} \right)d\tau\geq 0,
\end{equation}
where the sets ${\mathcal T}_i^s$ are defined as in Step 1 and depend upon the trajectory $(y,\alpha)$.
The arguments below will differ whether $(y, \alpha)\in \Gamma_{\bar t,s}[O]$ remains at $O$ or enters in a given edge $J_i$.

{\it Case $(a)$: the trajectory remains at $O$}. For any %$i\in\{0,\dots,N\}$ and
$s\in(\bar t,T]$, consider the trajectory $(y, \alpha)\in \Gamma_{\bar t,s}[O]$ with $\alpha(\cdot)=0$. Clearly, $y(\cdot)=O$ and $(\bar t,s)={\mathcal T}_0^s$. Then \eqref{eq:1dic_7} becomes
\begin{equation*}
\varphi(O,s)-\varphi(O,\bar t)+\int_{\bar t}^s \ell_O(\tau)d\tau\geq 0.
\end{equation*}
From the continuity of $\ell_O$ with respect to $t$,
\begin{equation*}
\varphi(O,s)-\varphi(O,\bar t)+(s-\bar t) \ell_O(\bar t)\geq -(s-\bar t)\o(s-\bar t),
\end{equation*}
for some modulus of continuity $\o$.
Dividing by $(s-\bar t)$,letting $s\to \bar t^+$  taking into account the regularity of~$\varphi$ yield
\begin{equation}\label{eq:1dic_a}
\partial_t\varphi(O,\bar t)+\ell_O(\bar t)\geq 0.
\end{equation}
%and, by the arbitrariness of~$i$, also
%\begin{equation}\label{eq:1dic_a}
%\partial_t\varphi(O,\bar t)-\max\left\{-\ell_*(\bar t),\max_{i=1,\dots,N}\left\{-\ell_i(O,\bar t)\right\}\right\}\geq 0.
%\end{equation}
{\it Case $(b)$: the trajectory enters in a given edge}. Fix $i\in\{1,\dots,N\}$. For any $n\in\N\setminus\{0\}$, fix $\bar a\in(0,n)$. For any $s\in(\bar t,T]$, consider the trajectory $(y, \alpha)\in \Gamma_{\bar t,s}[O]$ with $\alpha(\tau)=\bar a e_i$ for $\tau\in(\bar t,s)$. Clearly, $\alpha(\tau)\in A_i$ and $y(\tau)\in J_i\setminus\{O\}$ for $\tau\in(\bar t,s)$. Thus $(\bar t,s)={\mathcal T}_i^s$. Note that here the unboudedness of~$J_i$ is not essential. Indeed, if $J_i$ had a finite length $l_i$, then it would be  enough to choose $s\leq \bar t+l_i/\bar a$.
By the same arguments as in Step~$1$ (see \eqref{eq:1dic_1}), inequality~\eqref{eq:1dic_7} can be written 
\begin{equation*}
\int_{\bar t}^s\left[\partial_t \varphi( y(\tau),\tau)+D \varphi_{\mid J_i}(y(\tau),\tau )\bar a+\ell_i(y(\tau),\tau)+\frac{\bar a^2}{2}\right]\, d\tau\geq 0, \quad \quad\forall s\in[\bar t, T].
\end{equation*}
As in Step~$1$, taking into account Remark~\ref{rmk:bound_contr}, estimate~\eqref{eq:1dic_2} and the uniform continuity of~$\ell_i$ in any neighbourhood of~$O$, we get
\begin{equation*}
\int_{\bar t}^s\left[\partial_t \varphi(O,\bar t)+D \varphi_{\mid J_i}(O,\bar t)\bar a+\ell_i(O,\bar t)+\frac{\bar a^2}{2}\right]\, d\tau\geq -(s-\bar t)\o(s-\bar t),\quad \quad\forall s\in[\bar t, T]
\end{equation*}
for a suitable modulus of continuity~$\o$. Dividing the previous inequality by $(s-\bar t)$ and letting~$s\to\bar t^+$ yield
\begin{equation*}
\partial_t \varphi(O,\bar t)+D \varphi_{\mid J_i}(O,\bar t)\bar a+\ell_i(O,\bar t)+\frac{\bar a^2}{2}\geq 0.
\end{equation*}
Since $\bar a\in(0,n)$ is arbitrary,
\begin{equation*}
\partial_t \varphi(O,\bar t)-\sup_{\bar a\in(0,n)}\left\{-D \varphi_{\mid J_i}(O,\bar t)\bar a-\ell_i(O,\bar t)-\frac{\bar a^2}{2}\right\}\geq 0,
\end{equation*}
and, since  $n$ is arbitrary,
\begin{equation*}
\partial_t \varphi(O,\bar t)-\sup_{\bar a\geq 0}\left\{-D \varphi_{\mid J_i}(O,\bar t)\bar a-\ell_i(O,\bar t)-\frac{\bar a^2}{2}\right\}\geq 0.
\end{equation*}
Then \eqref{eq:1dic_easy} yields
\begin{equation*}
\partial_t \varphi(O,\bar t)-H_i^{\downarrow}(O,\bar t,D \varphi_{\mid J_i}(O,\bar t))\geq 0.
\end{equation*}
Since  $i$ is arbitrary,  and from inequality~\eqref{eq:1dic_a}, we deduce 
\begin{equation*}
\partial_t \varphi(O,\bar t)-\max\left\{\max\left\{-\ell_*(\bar t),\max_{i=1,\dots,N}\left\{-\ell_i(O,\bar t)\right\}\right\},\max_{i=1,\dots,N}\left\{H_i^{\downarrow}(x,t,D \varphi_{\mid J_i}(O,\bar t))\right\}\right\}
\geq 0,
\end{equation*}
i.e. the desired inequality.
\end{proof}
\begin{remark}
If  $g\in C(\cG)$, then there is a unique viscosity solution $u$ of  \eqref{HJ}
such that $u\in C_b(\cG\times [0,T])$ and $u(\cdot,T)=g$, see e.g. \cite{IM}.
We have not found any relevant literature on uniqueness when  $g$ is not  continuous and plan to address this topic in a future work.
\end{remark}

\section{Relaxed Mean Field Games equilibria}\label{sect:MFGequil}
%
%Setting and notations
%
We are now ready to tackle Mean Field Games. Relying on the results contained in Section~\ref{sec:OC}, we prove that there exists a relaxed MFG equilibrium and study the related mild solutions.

\subsection{Setting and notations}\label{subsec:setting}

\paragraph{Probability sets and evaluation map.} Let $\cP(\cG)$ denote the set of Borel probability measures on $\cG$ endowed with the narrow topology. Similarly, $\cP(\Gamma)$ stands for the set of Borel probability measures on $\Gamma$.
For $t\in [0,T]$, the evaluation map $e_t:\Gamma\to \cG$ is defined by $e_t(y_x, \alpha)=y_x(t)$.
For any $\mu\in \cP(\Gamma)$ and  $t\in [0,T]$, the Borel probability measure $m^\mu(t)$ on $ \cG$ is defined by $m^\mu(t)=e_t\sharp \mu$.

\paragraph{Costs.} The running cost and the terminal cost depend on the distribution of the population.
We consider the costs $L_i\in C (\cP(\cG);C_b(\cG\times [0,T]))$, for $i=1,\dots,N$, and $L_*\in C(\cP(\cG);C([0,T])$.
Similarly, let $G_i:\cP(\cG)\to C_b(\cG)$, $i=1,\dots,N$, and $G_*:\cP(\cG)\to \R$ be continuous functions.
%With a small abuse of notations, we shall denote $F_0[m]$ and $G_0[m]$ also the corresponding functions on~$\cG\times [0,T]$ which are constant; we need these notations only for considering all the $F_i$'s together without splitting them according to $i=0$ or $i\ne 0$, and similarly for the $G_i$'s.
The images of $m\in \cP(\cG)$ by $L_i$, respectively by $G_i$ are denoted by $L_i[m](\cdot,\cdot)$, respectively $G_i[m](\cdot)$, and we introduce similar notations for $L_*$ and $G_*$.

Let the real number  $K$ be defined as follows:
\begin{multline}\label{eq:37}\tag{$H^{\textrm{MFG}}_1$}
K=\max\left(\sup_{m\in \cP(\cG)} \|L_*[m]\|_{L^\infty} ,\max_{i=1,\dots,N}\sup_{m\in \cP(\cG)} \|L_i[m]\|_{L^\infty} , \sup_{m\in \cP(\cG)} \|G_*[m]\|_{L^\infty} ,\right.\\ \left. \max_{i=1,\dots,N}\sup_{m\in \cP(\cG)} \|G_i[m]\|_{L^\infty}\right)\in\R^+.  
\end{multline}
For brevity,  we  write
\begin{equation}
\label{eq:37b}
\begin{array}[c]{rcl}
    L[m](x,t)&=&\ds \sum_{i=1}^N L_i[m](x,t)\car_{x\in J_i\setminus\{O\}} + L_O[m](t)\car_{x=O},\\  
     G[m](x)&=&\ds \sum_{i=1}^N G_i[m](x)\car_{x\in J_i\setminus\{O\}} +\min\left\{G_*[m],\min_{i=1,\dots,N}G_i[m](O)\right\}\car_{x=O},
\end{array}
\end{equation}
%\begin{eqnarray}
%L[m](x,t)&=&\sum_{i=1}^N L_i[m](x,t)\car_{x\in J_i\setminus\{O\}} + L_O[m](t)\car_{x=O}\notag ,\\ %\label{eq:37b}
%G[m](x)&=&\sum_{i=1}^N G_i[m](x)\car_{x\in J_i\setminus\{O\}} +\min\{G_*[m],\min_{i=1,%\dots,N}G_i[m](O)\}\car_{x=O}
%\end{eqnarray}
for $x\in \cG$ and $t\in[0,T]$,
where
\begin{equation*}
L_O[m](\tau)=\min\left\{L_*[m](\tau),\min_{i=1,\dots,N}L_i[m](O,\tau)\right\}.
\end{equation*}
\paragraph{Admissible paths.}
Let us introduce the sets of admissible {\it paths}
\begin{equation}\label{eq:gamma_tilde_x}
\tilde \Gamma_C[x]=\left\{y\in Y_{x,0}\;:\; d(y(s),O)\leq C,\  \forall s\in[0,T],\  \|\dot y\|_2\leq C\right\},\quad
\tilde \Gamma_C=\bigcup_{x\in\cG}\tilde \Gamma_C[x],
\end{equation}
and endow $\tilde \Gamma_C$  with the topology of uniform convergence. Note that a {\it path} is the sole $y\in Y_{y(0),0}$ while a {\it trajectory} is formed by the couple $(y,\alpha)\in\Gamma$.
\begin{lemma}\label{lemma:Gcomp}
For every positive constant~$C$, the set~$\tilde \Gamma_C$ is compact.
\end{lemma}
\begin{proof}
Fix $C>0$ and consider a sequence $\{y_n\}_{n\in\N}$, with $y_n\in \tilde \Gamma_C$. Possibly for a subsequence (still denoted by $y_n$), 
$\{\dot y_n\}_n$ converges in the weak topology of~$L^2([0,T],\R^d)$ to some~$\alpha\in L^2([0,T],\R^d)$, with $\|\alpha\|_2\leq C$. Then, $\{y_n\}_n$ converges  uniformly to some  $y\in C([0,T],\cG)$. Clearly, $\alpha=\dot y$. The same arguments as in the proof of Proposition~\ref{prp:ex_OT} yield that the path~$y$ is admissible, i.e. $y\in Y_{y(0),0}$, and consequently that $y$ belongs to~$\tilde \Gamma_C$.
\end{proof}
%
%Lipschitz admissible paths
%
\paragraph{Lipschitz admissible paths.}
Given two positive constants $V$ and $C$,
let us  introduce the sets of Lipschitz admissible {\it paths} 
\begin{equation}\label{gamma_lip}
\Gamma^{\textrm{Lip}}_{C,V}[x]=\left\{y\in \tilde \Gamma_C[x]\;:\; \|y'\|_\infty\leq V\right\},\quad
\Gamma^{\textrm{Lip}}_{C,V}=\bigcup_{x\in\cG}\Gamma^{\textrm{Lip}}_{C,V}[x],
\end{equation}
and endow 
$\Gamma^{\textrm{Lip}}_{C,V}$ with the topology of uniform convergence. The same arguments as in Lemma~\ref{lemma:Gcomp} yield that $\Gamma^{\textrm{Lip}}_{C,V}$ is compact.

\paragraph{The set $\cP(\tilde \Gamma_C)$ and the associated costs.} Let $\cP(\tilde \Gamma_C)$ denote the set of probability measures on~$\tilde \Gamma_C$ endowed with the narrow topology. 
For $t\in [0,T]$, the evaluation map $e_t:\tilde \Gamma_C\to \cG$ is defined by $e_t(y)=y(t)$.
For any $\mu\in \cP(\tilde \Gamma_C)$ and  $t\in [0,T]$, define the Borel probability measure $m^\mu(t)$ on $ \cG$ by $m^\mu(t)=e_t\sharp \mu$. Clearly, $\textrm{supp}(m^\mu(t))\subset\{x\in\cG\;:\;d(x,O)\leq C\}$.
It is possible to prove that, if $\mu\in \cP(\tilde\Gamma_C)$, then the map $t\mapsto m^\mu(t)$ belongs to $C^{1/2}([0,T],\cP(\cG))$, see Lemma~\ref{lemma:3.1bis} below. Hence, for all $(y,\alpha)\in \Gamma$, the functions~$t\mapsto F_i[m^\mu(t)](y(t))$ are continuous and bounded by the constant $K$ introduced in (\ref{eq:37}).
\\
With $\mu\in \cP(\tilde\Gamma_C)$ and $(y,\alpha)\in\Gamma[x]$, we associate the cost
\begin{equation}\label{costMFG}
J^\mu(x;(y,\alpha))
=\int_0^T \left(L[m^\mu(\tau)] (y(\tau),\tau)+\frac{|\alpha(\tau)|^2}{2} \right)d\tau+  G[ m^\mu(T)](y(T)).
\end{equation}
\begin{remark}\label{rmk:alpha_y}
We recall that for each $y\in\tilde\Gamma_C[x]$ there exists $\alpha\in L^2([0,T],\R^d)$ such that $(y,\alpha)\in\Gamma[x]$, from Theorem~\ref{sec:optim-contr-probl} and Remark~\ref{rmk:2.2} . Such a control~$\alpha$ is unique for a.e. $t\in\{t\in[0,T]:y(t)\ne O\}$, which is not the case in $\{t\in[0,T]:y(t)= O\}$. However, the associated cost is independent of the choice of this control, namely: for any $y\in\tilde\Gamma_C[x]$, there holds
\begin{equation*}
J^\mu(x;(y,\alpha_1))=J^\mu(x;(y,\alpha_2))\qquad \forall (y,\alpha_1),(y,\alpha_2)\in \Gamma[x]. 
\end{equation*}
%\red{With a small abuse of notation, for each $y\in\tilde\Gamma_C[x]$, we denote $y'$ any of such controls.}
\end{remark}
For every  $y\in\tilde\Gamma_C[x]$, we define $\alpha_y$ the control such that $(y,\alpha_y)\in \Gamma[x]$ and $\alpha_y(t)=0\in A_0$ for a.e. $t\in\{t\in[0,T]:y(t)= O\}$. Note that this control is uniquely defined up to a set of null measure.
\paragraph{Optimal trajectories.} Fix $\mu\in \cP(\tilde\Gamma_C)$; for any $x\in \cG$, let us set
\begin{equation}\label{eq:42}
\Gamma^{\mu,{\rm opt}}[x]=\left\{(y,\alpha)\in \Gamma[x]\;:\; J^\mu(x;(y,\alpha))=\min_{( \widetilde y,\widetilde \alpha)\in \Gamma[x]}  J^\mu(x; (\widetilde y,\widetilde \alpha)) \right\}
\end{equation}
where $J^\mu $ is defined in (\ref{costMFG}).
Proposition~\ref{prp:ex_OT} entails that for each $\mu\in \cP(\tilde\Gamma_C)$ and $x\in \cG$, the set $\Gamma^{\mu,{\rm opt}}[x]$ of optimal trajectories starting from~$x$ is not empty.\\
We set $\Gamma^{\mu,{\rm opt}}=\cup_{x\in\cG} \Gamma^{\mu,{\rm opt}}[x]$.
\begin{remark}\label{rmk:OT_unif}
From assumption~\eqref{eq:37}, there exists a positive constant~$\tilde C$ such that, for every~$\mu\in \cP(\tilde\Gamma_C)$, $x\in\cG$ and $(y,\alpha)\in \Gamma^{\mu,{\rm opt}}[x]$, there holds $\|\alpha\|_2\leq \tilde C$.
In particular, if $m_0\in\cP(\cG)$ has compact support,  then for every~$\mu\in \cP(\tilde\Gamma_C)$, $x\in \textrm{supp}(m_0)$ and $(y,\alpha)\in \Gamma^{\mu,{\rm opt}}[x]$, there holds $y\in \tilde\Gamma_{\tilde C}[x]$ (possibly after taking a larger value of the constant~$\tilde C$).
\end{remark}

%\paragraph{Best control.}
%Let us recall from Theorem~\ref{sec:optim-contr-probl} and Remark~\ref{rmk:2.2} that, for any admissible path $y\in Y_{x,0}$, there exists (at least) a control $\alpha$ such that $(y,\alpha)$ belongs to~$\Gamma[x]$ and that in general this control is not unique even up to a set of null measure. Nevertheless, among these controls, we can select the one (unique up to a set of null measure) which is the ``best'' for minimizing our cost~$J^\mu$. For any admissible path $y\in Y_{x,0}$, we introduce the control $\alpha_y$ as
%\begin{equation}\label{eq:best_alpha}
%\alpha_y(\tau)=\dot y(\tau)\quad\textrm{on }\{s\in[0,T]\mid y(s)\ne O\},\qquad
%\alpha_y(\tau)=(i_\tau,0) \quad \textrm{on }\{s\in[0,T]\mid y(s)= O\}
%\end{equation}
%where $i_\tau$ is the minimum among the indices $j_\tau\in\{0,\dots,N\}$ such that $F_{j_\tau}[m^\mu(\tau)](O,\tau)=\min_{i\in\{0,\dots,N\}}F_{i}[m^\mu(\tau)](O,\tau)$. Note that the couple~$(y,\alpha_y)$ belongs to~$\Gamma[x]$ because Stampacchia theorem ensures $\dot y=0$ a.e. in $\{s\in[0,T]\mid y(s)= O\}$ so we can choose $\alpha_y$ in any set~$A_i$.\\
%The map $\iota:\tilde\Gamma_C\to\Gamma$, defined as $\imath(y)=(y,\alpha_y)$ is injective; hence, we shall identify~$\tilde\Gamma_C$ with the subset $\{(y,\alpha_y)\;:\; y\in\tilde\Gamma_C\}$ of~$\Gamma$.
%\begin{remark}\label{rmk:best_alpha}
%the trajectory $(y,\alpha_y)$ minimizes the cost $J^\mu$ among all the trajectories $(y,\tilde \alpha)\in\Gamma[x]$.
%\end{remark}
\paragraph{The set $\cP_{m_0}(\tilde \Gamma_C)$.}
We assume
\begin{equation}\label{eq:371}\tag{$H^{\textrm{MFG}}_2$}
m_0\in \cP(\cG)\qquad \textrm{has compact support.}
\end{equation}
Let $\cP_{m_0}(\tilde \Gamma_C)$ denote the set of measures $\mu\in\cP(\tilde \Gamma_C)$ such that $e_0\sharp \mu=m_0$.
In general, $\cP_{m_0}(\tilde \Gamma_C)$ may be empty.
However, in the present framework, this is not the case:
\begin{lemma}\label{lemma:Pm0_notempty}
Under assumptions~\eqref{eq:37} and~\eqref{eq:371}, for $C$ sufficiently large, $\cP_{m_0}(\tilde \Gamma_C)$ is not empty.
\end{lemma}
\begin{proof}
The proof consists of adapting some arguments in \cite[Remark 3.2]{CC}. For $C\geq \tilde C$ (where $\tilde C$ is the constant introduced in Remark~\ref{rmk:OT_unif}),  consider the map: $j:\textrm{supp}(m_0)\to \tilde \Gamma_C$, $j(x)(t)=x$ for any $t\in[0,T]$. Set $\tilde m_0=m_{0\mid \textrm{supp}(m_0)}$, the restriction of $m_0$ to its support. Observe that $e_0\#(j\#\tilde m_0)= m_0$, hence $(j\#\tilde m_0)\in \cP_{m_0}(\tilde \Gamma_C)$.
\end{proof}
\paragraph{The set $\cPmzeroGammaLipCV$.}
We assume \eqref{eq:371}.
%and  define 
%{\color{red} \begin{equation}\label{PLip}
%\cPmzeroGammaLipCV
%:=\left\{\mu\in \cP_{m_0}(\Gamma^{\textrm{Lip}}_{C,V})%\;:\;e_t\#\mu:[0,T]\to \cP(\cG) \textrm{ is Lipschitz}\right\}.
%\end{equation}
%why is the first index Lip useful?  With a  lemma similar to lemma 3.8, the lip continuity of $e_t\sharp \mu$ seems automatic in $\cP_{m_0} (\Gamma_{C,V}^{Lip})$}
Let $\cPmzeroGammaLipCV$ denote the set of measures $\mu\in\cP( \Gamma^{\mathrm{Lip}}_{C,V})$ such that $e_0\sharp \mu=m_0$.
Adapting the arguments in the proof of Lemma~\ref{lemma:Pm0_notempty}, we obtain that, for $C$ and $V$ sufficiently large, $\cPmzeroGammaLipCV$ is not empty.

Let us give an example, particularly simple because the agents do not interact, in which the distribution of states may develop a singularity.

%%%%%%%%%%%%%
\begin{example}\label{exa:dirac}
In a junction with two edges, %$e_1=(-1,0,\dots,0)$, $e_2=(1,0,\dots,0)$ and
consider the costs: $L_1[m]\equiv -1$, $L_2[m]\equiv 1$, $L_*[m]=-1$ and $G_i[m]\equiv 0$ ($i=*,1$) and $G_2[m]= m(J_2)$ for every $m\in\cP(\cG)$. Assume that the initial distribution of states is uniform on $[0,1/2]e_1\cup [0,1/2]e_2$. 
%The distribution of agents develops a Dirac delta in the vertex immediately after time $t=0$.\\
Fix $\bar x\in(0,1/2]$. Let $(y,\alpha)$ be an optimal trajectory starting at~$\bar x e_2$ at time $t=0$.\\
We claim that, for $T$ sufficiently large, $(y,\alpha)$ reaches~$O$ at time $t_{\bar x}=\bar x/2$ and stops there.
Indeed, either $y(\cdot)=\bar x e_2$ in $[0,T]$ (and the corresponding cost is equal to $T$) or there exists $s_1\in(0,T]$ such that $s_1=\min\{s\in[0,T]:y(s)=O\}$ because the other possibilities are less convenient. In the latter case, $y(\cdot)=O$ in $[s_1,T]$  is the optimal choice among all the trajectories $(\tilde y, \tilde \alpha )$ such that $\tilde \alpha(s)=\alpha(s)$ if $s\in[0,s_1]$.

Then, from the Euler-Lagrange condition in Lemma~\ref{lemma:EL}, there holds $\alpha(\cdot)=-\bar \alpha e_2$ in $(0,s_1)$ for a constant $\bar \alpha>0$. Hence, $s_1=\bar x/\bar\alpha$. The resulting cost is 
\[
\frac{\bar x\bar\alpha}{2}+\frac{2\bar x}{\bar\alpha}-T,
\]    
whose minimum w.r.t. $\bar\alpha\in(0,+\infty)$ is attained when $\bar\alpha=2$. With this choice of $\bar \alpha$, the cost is $2\bar x-T$ which is the minimal one, provided that $T$ is sufficiently large. Our claim is completely proved.

Therefore, the distribution of agents develops a singularity at  the vertex $O$ immediately after time $s=0$: for $s\in (0,T]$, the  singularity is $c(s)\delta_O$ (here, $\delta_O$ is the Dirac delta at $O$) with $c(s)=2s$ for $s\in(0,1/4]$ and $c(s)=1/2$ for $s\in[1/4,T]$.

Analogously, for $L_1=L_2=1$ and $L_*=-1$, a Dirac delta immediately appears at $O$ and after the time $1/4$, the whole population is concentrated at $O$.
\end{example}
In the next example, again without interactions, the distribution  of states develops a singularity that, after a while, starts travelling inside the edges.

\begin{example}\label{exa:moving_Dirac}
Consider a network with two vertices $V_1$ and~$V_2$ and three edges $J_1$, $J_2$ and $J_3$ such that $J_1\cap J_2=V_1$, $J_2\cap J_3=V_2$, $J_1\cap J_3=\emptyset$. For simplicity, assume that $V_1$ coincides with the origin $O$. The edges $J_1$ and $J_3$ are unbounded while the edge $J_2$ has length equal to $1$, say $J_2=[0,1]e_2$ for some unit vector $e_2$ ( i.e. $V_1=0e_2$ and $V_2=e_2$). The running cost $L$ and the terminal cost $G$ are defined on the three edges as follows: for any measure $m$ on the network
\begin{equation*}
L[m](x,t)=\left\{\begin{array}{ll}
0,&\qquad\textrm{if }x\in J_1\cup\{V_1\},\\
k_L,&\qquad\textrm{if }x\in (J_2\cup J_3\cup\{V_2\})\setminus\{V_1\},
\end{array}\right.
\end{equation*}
\begin{equation*}
G[m](x,t)=\left\{\begin{array}{ll}
0,&\qquad\textrm{if }x\in (J_1\cup J_2\cup\{V_1\})\setminus\{V_2\},\\
-k_G,&\qquad\textrm{if }x\in J_3\cup\{V_2\},\end{array}\right.
\end{equation*}
for some positive constants $k_L$ and $k_G$ which will be chosen later on. Note that these costs fulfill the assumptions~\eqref{eq:460} and \eqref{eq:461}. The time horizon~$T$ will be chosen suitably large later on.

Assume for a moment that, for $T>(2k_L)^{-1/2}$ and $k_G> 2\sqrt{2k_L}$  %{\color{red} I find $k_G>2\sqrt{2k_L}$ in order to get case A> Subcase B2, but I am not sure}\green{CLA: I agree. The assumption $k_G>\sqrt{2k_L}$ is enough for \eqref{eq:26feb_1} but is not enough for subcase $B2$. So, in this example, we must choose $k_G>2\sqrt{2k_L}$.  }
, for any $t\in[0,T-(2k_L)^{-1/2}]$, any $(y,\alpha)\in \Gamma^{{\rm opt}}_t[V_1]$ is such that 
\begin{equation}\label{eq:26feb_1}
y(s)=\left\{\begin{array}{ll}
V_1,&\qquad \textrm{for }s\in [t,T-(2k_L)^{-1/2}],\\
(2k_L)^{1/2}\left(-T+(2k_L)^{-1/2}+s\right)e_2,&\qquad \textrm{for }s\in (T-(2k_L)^{-1/2},T],
\end{array}\right.
\end{equation}
i.e. the  trajectory remains at $V_1$ up to time $T-(2k_L)^{-1/2}$ and enters afterwards  in $J_2$ with constant velocity, so to reach $V_2$ at time $T$.

Under the latter assumption, let us prove that,  for $T$ sufficiently large, if $(y,\alpha)\in \Gamma^{{\rm opt}}_0[\bar x e_2]$, with $\bar x\in[0,1]$, then 
\begin{equation}\label{eq:26feb_2}
y(s)=\left\{\begin{array}{ll}
\left(-(2k_L)^{1/2} s+\bar x\right)e_2,&\qquad \textrm{for }s\in [0,\bar x(2k_L)^{-1/2}),\\
V_1,&\qquad \textrm{for }s\in [\bar x(2k_L)^{-1/2},T-(2k_L)^{-1/2}],\\
(2k_L)^{1/2}\left(-T+(2k_L)^{-1/2}+s\right)e_2,&\qquad \textrm{for }s\in (T-(2k_L)^{-1/2},T],
\end{array}\right.
\end{equation}
i.e., the trajectory moves towards $V_1$ with velocity $(2k_L)^{1/2}$, reaches $V_1$ at time $(2k_L)^{-1/2}\bar x $ and remains there until time $T-(2k_L)^{-1/2}$, then moves towards $V_2$ with velocity $(2k_L)^{1/2}$ and reaches $V_2$ at the horizon $T$. Clearly, if $m_0$ is supported in $J_2$, then \eqref{eq:26feb_2} entails that all the agents first reach $V_1$, (so a singularity appears in the distribution), then all together start to move toward $V_2$  at time $T-(2k_L)^{-1/2}$.

Let us prove \eqref{eq:26feb_2}.
Since $\bar x e_2\in J_2\setminus\{V_1,V_2\}$,  Euler-Lagrange condition in Lemma~\ref{lemma:EL} implies that  the control $\alpha$ is constant on an interval $[0,\tau)$, for some $\tau\in(0,T]$. 

Let us list all the possible strategies and  compare the corresponding costs.
\\
{\it Strategy $A$:} $y(s)=\bar x e_2$ and $\alpha(s)=0$ for all $s\in [0,T]$. The cost is $J_A=k_L T$.\\
{\it Strategy $B$:} $\alpha$ is constant on $[0,\tau)$, where $\tau=\min\{T, \min\{s>0:\, y(s)\in \{V_1,V_2\}\}\}$.
Note that, if $y$ remains in $J_2\setminus\{V_1,V_2\}$ in the whole interval $[0,T]$, then the cost $J_A$ is not larger.
Thus  we may assume $\tau=\min\{s>0:\, y(s)\in \{V_1,V_2\}\}< T$. We distinguish two subcases whether $y(\tau)=V_1$ or $y(\tau)=V_2$.\\
{\it Strategy $B1$:} $y(\tau)=V_2$. Euler-Lagrange conditions yields $\alpha(s)=(1-\bar x)/\tau \; e_2$ and $y(s)= 
\tau^{-1} [\bar x \tau +(1-\bar x)s]e_2$ on $[0,\tau)$. It is then clear that $y(s)=V_2$ for $s\in [\tau,T]$, because the other possibilities lead to higher costs. The corresponding cost is $
(1-\bar x)^2/(2\tau) \, +k_LT-k_G$.  Since the latter quantity is strictly decreasing w.r.t. $\tau$, its minimum in $\tau$ is 
achieved by $\tau=T$.  Hence the optimal cost with Strategies of type   $B1$ is $J_{B1}=(1-\bar x)^2/(2T) \, +k_LT-k_G$.\\
{\it Strategy $B2$:} $y(\tau)=V_1$.  Euler-Lagrange condition yields  $\alpha(s)=-(\bar x/\tau) e_2$ and $y(s)=\bar x(1-s/\tau)e_2$ on $[0,\tau)$. Then \eqref{eq:26feb_1} implies that 
\begin{equation*}
y(s)=\left\{\begin{array}{ll}
V_1&\qquad \textrm{for }s\in [\tau,T-(2k_L)^{-1/2}],\\
(2k_L)^{1/2}\left(-T+(2k_L)^{-1/2}+s\right)e_2&\qquad \textrm{for }s\in (T-(2k_L)^{-1/2},T].
\end{array}\right.
\end{equation*}
The cost corresponding to this trajectory is  
\[
\frac12\frac{\bar x^2}{\tau}+k_L\tau+\frac 12\sqrt{2k_L}+k_LT-k_L\left(T-\frac{1}{\sqrt{2k_L}}\right)-k_G,
\]
and its minimum w.r.t. $\tau\in[0,T]$ is achieved by $\tau=\bar x / \sqrt{2k_L}$. Hence the optimal cost in Strategy  $B2$  is $J_{B2}=(1+\bar x)   \sqrt{2k_L}-k_G$.\\
{\it Conclusion.} Comparing the costs $J_A$, $J_{B1}$ and $J_{B2}$, we obtain that $J_{B1}>J_{B2}$ and $J_A>J_{B2}$ for $k_L>1$, $k_G$ satisfying the assumptions before \eqref{eq:26feb_1} and $T$ sufficiently large. Hence, the optimal trajectory is that of Strategy $B2$.

There remains to prove \eqref{eq:26feb_1}. To this end, let us distinguish several possible strategies.\\
{\it Strategy $\tilde A$:} $y(s)=V_1$ in $[t,T]$. The cost is $J_{\tilde A}=0$.\\
{\it Strategy $\tilde B$:} Immediately or after a while, the trajectory $y$ enters in $J_1\setminus \{V_1\}$ and remains in $J_1$. Since the cost associated to the kinetic energy is higher than with Strategy $\tilde A$, Strategy $\tilde B$ is strictly suboptimal.
\\
{\it Strategy $\tilde C$:} Immediately or after a while, the trajectory $y$ enters in $J_2\setminus \{V_1\}$ and is such that  $y(T)\in (J_1\cup J_2)\setminus J_3$, in particular, $y(T)\ne V_2$. Since the cost associated to the kinetic energy is higher than with Strategy $\tilde A$, Strategy $\tilde C$ is strictly suboptimal.
\\
{\it Strategy $\tilde D$:} Immediately or after a while, the trajectory $y$ enters in $J_2$  and is such that $y(T)= V_2$. Hence,
\begin{itemize}
\item $y(s)=V_1$ on $[t,s_1]$ for some $s_1\in[t,T]$
\item for $s_2=\min\{s\in(s_1,T]:\, y(s)=V_2\}$, there holds: $y(s_2)=V_2$ and $y(s)\in J_2\setminus\{V_1\}$ for $s\in(s_1,s_2)$. Then, from Euler-Lagrange condition, $\alpha(s)=(s_2-s_1)^{-1}e_2$ for $s\in(s_1,s_2)$
\item $y(s)=V_2$ for $s\in[s_2,T]$ because all the other possibilities result in a higher cost.
\end{itemize}
The resulting cost is
\[
\frac12\frac{1}{s_2-s_1}+k_L T-k_Ls_1-k_G.
\]
Let us minimize the latter cost w.r.t. $s_1\in[t,T)$ and $s_2\in(s_1,T]$. Since it is strictly decreasing w.r.t. $s_2$, 
let us choose $s_2=T$ so there remains to minimize $\frac12\frac{1}{T-s_1}+k_L T-k_Ls_1-k_G$ with respect to $s_1\in[t,T)$.
The minimum is reached at $s_1=T-(2k_L)^{-1/2}$, and 
takes the value $J_{\tilde D}=\sqrt{2k_L}-k_G$ which is less than $J_{\tilde A}$ from the assumption on $k_G$.\\
{\it Strategy $\tilde E$:} Immediately or after a while, the trajectory $y$ enters  in $J_2$ and is such that $y(T)\in J_3\setminus \{V_2\}$.
Comparing the resulting cost with that of Strategy $\tilde D$, one can check that  Strategy $\tilde E$ is strictly suboptimal.
\end{example}
\subsection{Relaxed MFG equilibrium}
Fix $\mu\in \cP_{m_0}(\tilde\Gamma_C)$; for any $x\in \cG$, let us set
\begin{equation}\label{eq:42new}
\GammamuoptC[x]=\left\{y\in \tilde\Gamma_C[x]\;:\; J^\mu(x;(y,\alpha_y))=\min_{( \widetilde y,\widetilde \alpha)\in \Gamma[x]}  J^\mu(x; (\widetilde y,\widetilde \alpha)) \right\}
\end{equation}
where $J^\mu$ is defined in (\ref{costMFG}) and $\alpha_y$ is a control such that $(y,\alpha_y)\in\Gamma[x]$ (see Remark~\ref{rmk:alpha_y}).
\begin{definition}
  \label{sec:setting-notation-3}
The complete probability measure $\mu \in \cP_{m_0}(\tilde\Gamma_C)$ is a relaxed mean field game equilibrium associated with the initial distribution $m_0$ if 
\begin{equation}\label{eq:43}
  {\rm{supp}}(\mu)\subset \mathop \bigcup_{
x\in {\rm{supp}}(m_0)}\GammamuoptC[x].
\end{equation}
\end{definition}
The following two theorems address the existence of MFG equilibria under different hypothesis.
\begin{theorem}\label{sec:setting-notation-6}
Assume~$(H_0)$,~\eqref{eq:37} and in~\eqref{eq:371}; consider $C\geq \tilde C$ (where $\tilde C$ is the constant introduced in Remark~\ref{rmk:OT_unif}). Then, there exists a relaxed mean field equilibrium $\mu\in \cP_{m_0}(\tilde\Gamma_C)$. %Moreover, \red{propriet\`a di regolarit\`a della $\mu$ da controllare.}% $t\mapsto e_t\sharp \mu \in C^{1/2} ([0,T];\cP(K_C))$,  ($K_C$ is defined in (\ref{eq:33}) and $\cP(K_C)$ is endowed with the  Kantorovitch-Rubinstein distance).
\end{theorem}
The proof of Theorem \ref{sec:setting-notation-6} is postponed to subsection~\ref{subsect:proofmain}.
\begin{theorem}\label{thm:MFGlip}
Keeping the assumptions of  Theorem~\ref{sec:setting-notation-6}, we also assume that, for some positive constant~$K$, 
\begin{equation}\label{eq:431}\tag{$H^{\textrm{MFG}}_3$}
\left\{\begin{array}{c}
G_i[m]\in C^2(J_i)\qquad \textrm{and}\qquad L_i[m](\cdot,t)\in C^2(J_i)\qquad \forall m\in\cP(\cG),\, t\in[0,T]\\
\sup_{m\in\cP(\cG)}\max_{i=1,\dots,N}\|\partial G_i[m]\|_\infty\leq K\\ 
\sup_{m\in\cP(\cG)}\max_{i=1,\dots,N}\sup_{t\in[0,T]}\|\partial L_i[m](\cdot,t)\|_\infty\leq K,
\end{array}\right.
\end{equation}
Then, there exists a  relaxed MFG equilibrium $\mu\in \cPmzeroGammaLipCV$,
 where the constants  $V$ and $C$  appear respectively in Theorem~\ref{optcur_lip} with $t=0$ and $x\in \textrm{supp}(m_0)$ and in Theorem~\ref{sec:setting-notation-6}. 
\end{theorem}
The proof of Theorem \ref{thm:MFGlip} is postponed to subsection~\ref{subsect:proofmain2}.
%
%Preliminary results
%
\subsection{Preliminary results}

\begin{lemma}\label{sec:setting-notation-4}
Let a sequence of probability measures $\{\mu_n\}_{n\in \N}$,  $\mu_n\in\cP(\tilde\Gamma_C)$, be narrowly convergent to $\mu\in \cP(\tilde\Gamma_C)$ as $n\to\infty$. For all $t\in [0,T]$, the sequence $\{m^{\mu_n}(t)\}_{n\in \N}$ is narrowly convergent to $m^{\mu}(t)$.
\end{lemma}
\begin{proof}
Adapting the arguments of \cite[Lemma 3.1]{AMMT2} leads to
\begin{equation*}
\ds \int_{\cG} f(x) dm^{\mu_n}(t)(x)= \ds \int_{\tilde\Gamma_C} f(y(t)) d\mu_n(y)
\to\ds  \int_{\tilde\Gamma_C} f(y(t)) d\mu(y)= \ds\int_{\cG} f(x) dm^{\mu}(t)(x),
\end{equation*}
for all $f\in C^0_b(\cG;\R)$.
\end{proof}
\begin{lemma}\label{lemma:3.1bis}
Assume \eqref{eq:371}. There holds
\[
\sup\limits_{\mu\in \cP_{m_0}(\tilde\Gamma_C)}\textrm{Wass}_1(m^\mu(t),m^\mu(s))\leq C|t-s|^\half \qquad\forall t,s\in[0,T].
\]
Similarly,
\[
\sup\limits_{\mu\in  \cPmzeroGammaLipCV}\textrm{Wass}_1(m^\mu(t),m^\mu(s))\leq V|t-s| \qquad\forall t,s\in[0,T].
\]
\end{lemma}
\begin{proof}
Consider any $\mu\in \cP_{m_0}(\tilde\Gamma_C)$. For any $t,s\in[0,T]$, there holds
\begin{multline*}
\sup\limits_\phi\int_\cG\phi(x)[dm^\mu(t)-dm^\mu(s)](x)=
\sup\limits_\phi\int_{\tilde\Gamma_C}\left[\phi(y(t))-\phi(y(s))\right]d\mu(y)\\ 
\leq \int_{\tilde\Gamma_C}|y(t)-y(s)|d\mu(y)\leq |t-s|^\half\|\alpha\|_2
\end{multline*}
where the supremum is performed over all the continuous $1$-Lipschitz function. Owing to the definition of~$\tilde\Gamma_C$ in~\eqref{eq:gamma_tilde_x} and to the arbitrariness of $\mu\in \cP_{m_0}(\tilde\Gamma_C)$, the latter relation entails the first statement.
The second statement is obtained in a singular way.
\end{proof}

It is useful to recall the disintegration theorem:
\begin{theorem}\label{sec:setting-notation-1}
Let $X$ and $Y$ be Radon metric spaces, $\pi: X\to Y$ be a Borel map,  $\mu$ be a probability measure on $X$. Set $\nu=\pi\sharp\mu$. There exists a $\nu$-almost everywhere uniquely defined Borel measurable family of probability measures $(\mu_y)_{y\in Y}$ on $X$ such that
\begin{equation*}
  \label{eq:38}
\mu_y(X\setminus \pi^{-1}(y))=0,\quad \hbox{ for } \nu\hbox{-almost all } y\in Y,
\end{equation*}
and for every Borel function $f: X\to [0,+\infty]$,
\begin{equation*}
  \label{eq:39}
\int_X f(x) d\mu(x)=\int_Y  \left( \int_{X}f(x) d\mu_y(x) \right) d\nu(y)=\int_Y  \left( \int_{\pi^{-1}(y)}f(x) d\mu_y(x) \right) d\nu(y).
\end{equation*}
Recall that $(\mu_y)_{y\in Y}$ is a Borel family  of probability measures if for any Borel subset $B$ of $X$, $Y\ni y\mapsto \mu_y(B)$ is a Borel 
function from $Y$ to $[0,1]$.
\end{theorem}

%
%Proof of Theorem \ref{sec:setting-notation-6}
%
\subsection{A closed graph property}

 Choosing $C\geq \tilde C$, where $\tilde C$ is the constant introduced in Remark~\ref{rmk:OT_unif}, we  first establish a closed graph property for the map~$\GammamuoptC[x]$.
\begin{proposition}\label{prop2}
Consider $\mu\in\cP_{m_0}(\tilde\Gamma_C)$ and $x\in\textrm{supp}(m_0)$. Consider also a sequence of probability measures $\{\mu_n\}_{n\in\N}$, with $\mu_n\in\cP_{m_0}(\tilde\Gamma_C)$, narrowly convergent to~$\mu$ as $n\to\infty$ and a sequence of points $\{x_n\}_{n\in\N}$, with $x_n\in\cG$ and $x_n\to x$ as $n\to\infty$. 
Let~$\{y_n\}_{n\in\N}$ be a sequence of paths such that $y_n\in \tilde\Gamma_C^{\mu_n,{\rm opt}}[x_n]$ and $y_n$ uniformly converge to some path~$y$ as $n\to\infty$. 
Then, $y$ belongs to $\GammamuoptC[x]$, namely any trajectory~$(y,\alpha_y)$ is an optimal trajectory for~$J^\mu$. In other words, the multivalued map $(x,\mu)\rightrightarrows \GammamuoptC[x]$ enjoys the closed graph property.
\end{proposition}
\begin{proof}[Proof of Proposition~\ref{prop2}]
There are similar arguments as in the proof of Proposition~\ref{prp:prop1}, so we will mailnly focus on the new aspects. We wish to prove that
\begin{equation*}
(i)\quad y\in \tilde\Gamma_C[x],\qquad (ii)\quad (y,\alpha_y)\textrm{ is optimal for }J^\mu.
\end{equation*}
From the definition of~$\tilde\Gamma_C$, the controls $\alpha_{y_n}$ are uniformly bounded in~$L^2$.
%By our assumption~\eqref{eq:37}, we infer that the costs $J^{\mu_n}(x_n;(y_n,\alpha_n))$ are uniformly bounded and consequently that the $\|\alpha_n\|_{2}$ are uniformly bounded.
The same arguments as in the proof of Proposition~\ref{prp:prop1} show that, possibly up to a subsequence (still denoted by $\alpha_{y_n}$),  $\{\alpha_{y_n}\}_n$ converges in the weak topology of~$L^2((0,T),\R^d)$ to some control~$\alpha_y$, with $\|\alpha_y\|_2\leq C$, that $(y,\alpha_y)\in \Gamma[x]$ and $y\in\tilde\Gamma_C[x]$. The proof of point~$(i)$ is done.\\
Concerning ~$(ii)$, it suffices to prove that
\begin{equation*}
J^\mu(x;(y,\alpha_y))\leq J^\mu(x;(\hat y,\hat \alpha))\qquad \forall (\hat y,\hat \alpha)\in\Gamma[x].
\end{equation*}
Fix any $(\hat y,\hat \alpha)\in\Gamma[x]$. Lemma~\ref{lemma1} ensures that there exists a sequence $\{\hat y_n,\hat \alpha_n)\}_{n\in\N}$ such that $(\hat y_n,\hat \alpha_n)\in\Gamma[x_n]$, $\hat y_n(T)=\hat y(T)$ and
\begin{equation*}
\hat y_n \to \hat y \textrm{ uniformly in $[0,T]$ as $n\to\infty$},\qquad\qquad \|\hat \alpha_n\|_{2}\leq \|\hat \alpha\|_{2} +o_n(1),
%\limsup\limits_{n\to\infty}J_0(x_n;(\hat y_n,\hat \alpha_n))\leq J_0(x;(\hat y,\hat \alpha))
\end{equation*}
where $o_n(1)$ is a sequence such that $\lim_n o_n(1)=0$.
Since $y_n\in \tilde\Gamma_C^{\mu_n,{\rm opt}}[x_n]$, 
\begin{equation}\label{eq:34}
J^{\mu_n}(x_n;(y_n,\alpha_{y_n}))\leq J^{\mu_n}(x_n;(\hat y_n,\hat \alpha_n)).
\end{equation}
Let us now study separately the two sides of \eqref{eq:34}. For the right hand side, the construction and the properties of~$(\hat y_n,\hat \alpha_n)$ entail
\begin{eqnarray*}
J^{\mu_n}(x_n;(\hat y_n,\hat \alpha_n))
%&=& \int_0^T \left(F[m^{\mu_n}(\tau)] (\hat y_n(\tau),\hat \alpha_n(\tau),\tau)+\frac{|\hat \alpha_n(\tau)|^2}{2} \right)d\tau+  G[m^{\mu_n}(T)](\hat y(T))\\
&\leq& J^{\mu}(x;(\hat y,\hat \alpha))+\sum_{i=1}^4\bar I_i
\end{eqnarray*}
where, for $\d_n=d(x,x_n)$,
\begin{equation*}
\begin{array}{ll}
\ds \bar I_1=\int_0^{\d_n}L[m^{\mu_n}(\tau)] (\hat y_n(\tau),\tau)d\tau, \quad  &\ds \bar I_2= \frac{\|\alpha_n\|_2^2-\|\alpha\|_2^2}{2}\leq  o_n(1),\\
\ds \bar I_3=\int_{\d_n}^T L[m^{\mu_n}(\tau)] (\hat y_n(\tau),\tau)d\tau,&
\ds \bar I_4=-\int_0^{T} L[m^{\mu}(\tau)](\hat y(\tau),\tau) d\tau.
\end{array}
\end{equation*}
The boundedness of $L$ implies that $\lim_{n\to\infty}\bar I_1=0$.
Then, arguing as in the proof of Lemma~\ref{lemma1}, 
\begin{eqnarray*}
\bar I_3&=& \int_{0}^T \left[\sum_{i=1}^N L_i\left[m^{\mu_n}\left(\frac{T-\d_n}{T}\xi+\d_n\right)\right]\left(\hat y(\xi),\frac{T-\d_n}{T}\xi+\d_n\right)\car_{\hat y(\xi)\in J_i\setminus\{O\}}\right.\\&& \left.+L_O\left[m^{\mu_n}\left(\frac{T-\d_n}{T}\xi+\d_n\right)\right]\left(\frac{T-\d_n}{T}\xi+\d_n\right)\car_{\hat y(\xi)=O}\right]\left(1-\frac{\d_n}{T}\right)\, d\xi\\
&=&\int_{0}^T L\left[m^{\mu_n}\left(\frac{T-\d_n}{T}\xi+\d_n\right)\right]\left(\hat y(\xi),\frac{T-\d_n}{T}\xi+\d_n\right)\left(1-\frac{\d_n}{T}\right)\, d\xi,
\end{eqnarray*}
and consequently,
\begin{equation*}
\bar I_3+\bar I_4=\bar I_5+\bar I_6+\bar I_7+\bar I_8
\end{equation*}
where
\begin{equation*}
\begin{array}{rcl}
\bar I_5&=&\ds-\frac{\d_n}{T} \int_{0}^T L\left[m^{\mu_n}\left(\frac{T-\d_n}{T}\xi+\d_n\right)\right]\left(\hat y(\xi),\frac{T-\d_n}{T}\xi+\d_n\right)\, d\xi,\\
\bar I_6&=&\ds\int_{0}^T\left( L\left[m^{\mu_n}\left(\frac{T-\d_n}{T}\xi+\d_n\right)\right]\left(\hat y(\xi),\frac{T-\d_n}{T}\xi+\d_n\right)-  L\left[m^{\mu_n}\left(\xi\right)\right]\left(\hat y(\xi),\frac{T-\d_n}{T}\xi+\d_n\right) \right) \, d\xi,\\
\bar I_7&=&\ds\int_{0}^T\left( L\left[m^{\mu_n}\left(\xi\right)\right]\left(\hat y(\xi),\frac{T-\d_n}{T}\xi+\d_n\right)- L\left[m^{\mu}(\xi)\right]\left(\hat y(\xi),\frac{T-\d_n}{T}\xi+\d_n\right)\right)\, d\xi,\\
%\red{\xout{\bar I_6}}&\red{\xout{=}}&\red{\xout{\ds\int_{0}^T\left( L\left[m^{\mu_n}\left(\frac{T-\d_n}{T}\xi+\d_n\right)\right]\left(\hat y(\xi),\frac{T-\d_n}{T}\xi+\d_n\right)\right.}}\\
%&&\red{\xout{\ds\qquad\left.- L\left[m^{\mu}\left(\frac{T-\d_n}{T}\xi+\d_n\right)\right]\left(\hat y(\xi),\frac{T-\d_n}{T}\xi+\d_n\right)\right)\, d\xi,}}\\
%\red{\xout{\bar I_7}}&\red{\xout{=}}&\red{\xout{\ds\int_{0}^T\left( L\left[m^{\mu}\left(\frac{T-\d_n}{T}\xi+\d_n\right)\right]\left(\hat y(\xi),\frac{T-\d_n}{T}\xi+\d_n\right)- L\left[m^{\mu}(\xi)\right]\left(\hat y(\xi),\frac{T-\d_n}{T}\xi+\d_n\right)\right)\, d\xi,}}\\
\bar I_8&=&\ds \int_{0}^T\left( L\left[m^{\mu}(\xi)\right]\left(\hat y(\xi),\frac{T-\d_n}{T}\xi+\d_n\right)%\right.\\&&\qquad\left.
- L\left[m^{\mu}(\xi)\right]\left(\hat y(\xi),\hat \alpha(\xi),\xi\right)\right)\, d\xi.
\end{array}
\end{equation*}
The boundedness of $L$ entails: $|\bar I_5|=o_n(1)$.
From Lemma~\ref{lemma:3.1bis}, the assumptions on the costs $L_i$ and Lebesgue dominated convergence theorem, $|\bar I_6|=o_n(1)$. From Lemma~\ref{sec:setting-notation-4}, again the assumptions on the costs $L_i$ and Lebesgue dominated convergence theorem,  $|\bar I_7|=o_n(1)$. Finally, since $\hat y$ is  bounded  and $L_i[m]$ are continuous,  $|\bar I_8|=o_n(1)$.

To summarize, there holds
\begin{equation}\label{eq:sum1}
\limsup_{n}J^{\mu_n}(x_n;(\hat y_n,\hat \alpha_n))\leq J^{\mu}(x;(\hat y,\hat \alpha)).
\end{equation}
The left hand side of~\eqref{eq:34} is addressed with arguments from the proof of Proposition~\ref{prp:ex_OT}. By definition of cost~\eqref{costMFG}, 
\begin{eqnarray}\label{eq:34bis}
J^{\mu_n}(x_n;(y_n,\alpha_{y_n}))
&=& \int_0^T \frac{|\alpha_{y_n}(\tau)|^2}{2}d\tau+\sum_{i=1}^5 \hat I_i,
\end{eqnarray}
where
\begin{eqnarray*}
\hat I_1&=& \int_0^T\sum_{i=1}^N L_i[m^{\mu_n}(\tau)](y_n(\tau),\tau)\car_{y_n(\tau)\in J_i\setminus\{O\}} \car_{y(\tau)\in J_i\setminus\{O\}} d\tau,\\
\hat I_2&=& \int_0^T\sum_{i=1}^N L_i[m^{\mu_n}(\tau)](y_n(\tau),\tau)\car_{y_n(\tau)\in J_i\setminus\{O\}} \car_{y(\tau)\in \cG\setminus J_i} d\tau,\\
\hat I_3&=& \int_0^T\sum_{i=1}^N L_i[m^{\mu_n}(\tau)](y_n(\tau),\tau)\car_{y_n(\tau)\in J_i\setminus\{O\}} \car_{y(\tau)=O} d\tau,\\
\hat I_4&=&\int_0^T L_O[m^{\mu_n}(\tau)](\tau)\car_{y_n(\tau)=O} d\tau,\\
\hat I_5&=&G[m^{\mu_n}(\tau)](y_n(T)).
\end{eqnarray*}
The convergence in the weak topology of $L^2([t,T];\R^d)$ entails
\begin{equation*}
\int_0^T\frac{|\alpha(\tau)|^2}{2}d\tau\leq \liminf_{n\to\infty} \int_0^T\frac{|\alpha_n(\tau)|^2}{2}d\tau.
\end{equation*}
Recall from Lemma~\ref{sec:setting-notation-4} that, for each $t\in[0,T]$, the map $\cP(\tilde\Gamma_C)\ni\mu\mapsto m^\mu(t)\in\cP(\cG)$ is continuous. Hence, by our assumption, for every $i\in\{1,\dots,N\}$, $L_i[m^{\mu_n}(t)](\cdot,\cdot)$ and $G_i[m^{\mu_n}(T)](\cdot)$ converge uniformly respectively to $L_i[m^{\mu}(t)](\cdot,\cdot)$ and to $G_i[m^{\mu}(T)](\cdot)$ as $n\to\infty$.
Therefore, the dominated convergence theorem yields
\[
\hat I_1\to \int_0^T\sum_{i=1}^N L_i[m^{\mu}(\tau)](y(\tau),\tau)\car_{y(\tau)\in J_i\setminus\{O\}} d\tau\qquad \textrm{and}\qquad\hat I_2\to 0, \quad\textrm{as }n\to\infty.
\]
The same arguments as in the proof of Proposition~\ref{prp:ex_OT} and the definition of~$G[m]$ in~\eqref{eq:37b} imply
\begin{equation*}
\liminf_{n\to\infty}\hat I_5\geq  G[m^\mu(T)](y(T)).
\end{equation*}
Furthermore, 
\begin{eqnarray*}
\hat I_3+\hat I_4&=&\int_0^T\left[\sum_{i=1}^N L_i[m^{\mu_n}(t)](y_n(\tau),\tau)\car_{y_n(\tau)\in J_i\setminus\{O\}} +L_O[m^{\mu_n}(t)](\tau)\car_{y_n(\tau)=O}\right]\car_{y(\tau)=O} d\tau\\
&&+\int_0^T L_O[m^{\mu_n}(t)](\tau)\car_{y_n(\tau)=O}\car_{y(\tau)\ne O} d\tau.
\end{eqnarray*}
Again the dominated convergence theorem ensures
\[
\int_0^T L_O[m^{\mu_n}(t)](\tau)\car_{y_n(\tau)=O}\car_{y(\tau)\ne O} d\tau\to 0 \qquad\textrm{as }n\to\infty.
\]
Then, from Fatou's Lemma and the boundedness of $L_i$, 
\begin{multline*}
\liminf_{n\to\infty}\int_0^T\left[\sum_{i=1}^N L_i[m^{\mu_n}(\tau)](y_n(\tau),\tau)\car_{y_n(\tau)\in J_i\setminus\{O\}} \right.\\
\qquad\left. + L_O[m^{\mu_n}(\tau)](\tau)\car_{y_n(\tau)=O}\right]\car_{y(\tau)=O} d\tau\geq \int_0^T  L_O[m^{\mu}(\tau)](\tau)\car_{y(\tau)=O} d\tau.
\end{multline*}
Combining all the observations above with \eqref{eq:34bis} yields
\begin{eqnarray}%\label{eq:34ter}
\liminf_{n\to\infty} J^{\mu_n}(x_n;(y_n,\alpha_{y_n}))&\geq& \int_0^T\left[\frac{|\alpha(\tau)|^2}{2}+\sum_{i=1}^N L_i[m^{\mu}(\tau)](y(\tau),\tau)\car_{y(\tau)\in J_i\setminus\{O\}}\right.\notag \\&&\qquad \left.+ L_O[m^{\mu}(\tau)](\tau)\car_{y(\tau)=O}\right] d\tau
+ G[m^\mu(T)](y(T)) \notag\\ \label{eq:34quat}
&=&J^\mu(x;(y,\alpha_y)).
\end{eqnarray}
%Taking advantage of the definition of~$\alpha_y$, relation~\eqref{eq:34ter} can be rewritten as
%\begin{equation}
%\liminf_{n\to\infty} J^{\mu_n}(x_n;(y_n,\alpha_n))\geq J^\mu(x;(y,\alpha_y)).
%\end{equation}
In conclusion, \eqref{eq:34}, \eqref{eq:sum1} and~\eqref{eq:34quat} entail
\[
J^\mu(x;(y,\alpha_y))\leq J^\mu(x;(\hat y,\hat \alpha)).
\]
Since $(\hat y,\hat \alpha)\in\Gamma[x]$ is
arbitrary, we get
$J^\mu(x;(y,\alpha_y))=\min_{(\hat y,\hat \alpha)\in \Gamma[x]}J^\mu(x;(\hat y,\hat \alpha))$
which is equivalent to~$(ii)$.
\end{proof}

\subsection{Proof of Theorem \ref{sec:setting-notation-6}}\label{subsect:proofmain}
Let us first recall some notations. For every $\mu\in\cP_{m_0}(\tilde\Gamma_C)$, let $J^\mu$ be the associated cost as in~\eqref{costMFG}; for any $x\in\cG$, let $\GammamuoptC[x]$ be the set of optimal paths starting from~$x$ for the cost~$J^\mu$ as in~\eqref{eq:42}. Proposition~\ref{prp:ex_OT} ensures: $\GammamuoptC[x]\ne\emptyset$ for every $x\in\cG$. It is worth recalling that the set~$\tilde\Gamma_C$ is compact, from Lemma~\ref{lemma:Gcomp}. By Prokhorov theorem \cite[Theorem 5.1.3]{AGS},  $\cP(\tilde\Gamma_C)$ is also compact.

The multivalued map $E:\cP_{m_0}(\tilde\Gamma_C)\rightrightarrows \cP_{m_0}(\tilde\Gamma_C)$  is defined as follows:
\begin{equation}\label{eq:mappa}
E(\mu)=\{\hat \mu\in\cP_{m_0}(\tilde\Gamma_C)\;:\; \textrm{supp}\,\hat \mu_x\subset \GammamuoptC[x]\quad m_0-\textrm{a.e. } x\in\cG\},
\end{equation}
where $\{\hat \mu_x\}_{x\in\cG}$ is the family of Borel probability measures on $\cP_{m_0}(\tilde\Gamma_C)$ obtained applying the disintegration Theorem~\ref{sec:setting-notation-1} to $\mu$, $X$, $Y$ and $\pi$ being replaced respectively by $\hat \mu$, $\cP_{m_0}(\tilde\Gamma_C)$, $\cG$ and $e_0$ (so, clearly, $\nu$ coincides with $m_0$).
The proof of the theorem amounts to proving that the map~$E$ admits a fixed point.
Let us assume for the moment the following properties
\begin{itemize}
\item[(i)] for every $\mu\in\cP_{m_0}(\tilde\Gamma_C)$, the set $E(\mu)$ is not empty and convex
\item[(ii)] the map~$E$ enjoys the closed graph property.
\end{itemize}
Then, Kakutani fixed point theorem ensures that $E$ admits a fixed point $\mu$. Without any loss of generality, we can complete the measure~$\mu$ and obtain a relaxed MFG equilibrium. It remains to prove the above mentioned two properties.\\
(i).  Recall that $\GammamuoptC[x]\ne\emptyset$ for every $x\in\cG$ and that the map $\GammamuoptC[\cdot]$ has the closed graph property, from Proposition~\ref{prp:ex_OT} and Proposition~\ref{prp:prop1},
Therefore, the result~\cite[Theorem 8.1.4]{AF} guarantees that the map $\GammamuoptC[\cdot]$ has a Borel measurable selection  denoted $x\mapsto y^\mu_x$ for every $x\in\cG$.
We introduce a measure $\hat \mu$ on $\tilde\Gamma_C$ as follows:
\begin{equation*}
\hat \mu(B)=\int_{\cG}\delta_{y^\mu_x}(B)\, m_0(dx)\qquad \forall  \textrm{ Borel }B\subset\tilde\Gamma_C,
\end{equation*}
where $\delta_{y^\mu_x}(\cdot)$ is the Dirac delta-function centered in $y^\mu_x$. Note that $\hat \mu_x=\delta_{y^\mu_x}$ for $m_0$-a.e. $x\in\cG$. Hence, $\hat \mu$ belongs to~$E(\mu)$.\\ %; hence, $E(\mu)$ is not empty.\\
Let us now prove that $E(\mu)$ is convex. Fix $\mu^1,\mu^2\in E(\mu)$ and $\lambda\in[0,1]$. By easy calculation, one obtains $\lambda \mu^1+(1-\lambda)\mu^2\in \cP_{m_0}(\tilde\Gamma_C)$. On the other hand, for $i=1,2$, since $\mu^i\in E(\mu)$, by the disintegration theorem~\ref{sec:setting-notation-1}, there exists a Borel measurable family $\{\mu^i_{x}\}_{x\in\cG}$ of probability measures (which is $m_0$-a.e. uniquely defined and ``disintegrate'' $\mu^i$ with respect to $m_0$) and a set ${\mathcal A}_i\subset \cG$ such that $m_0({\mathcal A}_i)=0$ and $\textrm{supp}\, \mu^i_x\subset \GammamuoptC[x]$ for every $x\in\cG\setminus {\mathcal A}_i$. 
Therefore, the measure~$\lambda \mu^1+(1-\lambda)\mu^2$ can be disintegrated as follows: for each Borel function~$f$ on~$\tilde\Gamma_C$, there holds
\[
\int_{\tilde\Gamma_C} f(\gamma)(\lambda \mu^1+(1-\lambda)\mu^2)\, (d\gamma)=
\int_{\cG}\left(\int_{\tilde\Gamma_C} f(\gamma)(\lambda \mu^1_x+(1-\lambda)\mu^2_x)\, (d\gamma)\right)m_0(dx)
\]
with $m_0({\mathcal A}_1\cup {\mathcal A}_2)=0$ and
\[\textrm{supp}\, (\lambda \mu^1_x+(1-\lambda)\mu^2_x)\subset \GammamuoptC[x]\qquad\forall x\in\cG\setminus({\mathcal A}_1\cup {\mathcal A}_2).
\]
Hence, $\lambda \mu^1+(1-\lambda)\mu^2$ belongs to~$E(\mu)$, so $E(\mu)$ is convex.\\
(ii). Consider a sequence $\{\mu_n\}_{n\in\N}$ of probability measures $\mu_n\in \cP_{m_0}(\tilde\Gamma_C)$ which narrowly converges to some $\mu\in \cP_{m_0}(\tilde\Gamma_C)$ as $n\to\infty$. Consider also a sequence $\{\hat \mu_n\}_{n\in\N}$, with $\hat \mu_n\in E(\mu_n)$ for any $n\in\N$, which narrowly converges to some $\hat\mu\in \cP_{m_0}(\tilde\Gamma_C)$ as $n\to\infty$. Our aim is to prove that $\hat \mu$ belongs to~$E(\mu)$.\\
By the disintegration theorem, there exists a $m_0$-a.e. uniquely defined Borel measurable family of measures $\{\hat \mu_x\}_{x\in\cG}$ on~$\tilde\Gamma_C$ and ${\mathcal A}\subset\cG$ such that: $m_0({\mathcal A})=0$, $\hat\mu_x(\tilde \Gamma_C\setminus e^{-1}_0(\{x\}))=0$ for every $x\in \cG\setminus {\mathcal A}$ and
\[
\int_{\tilde\Gamma_C}f(y)\hat \mu(dy)=\int_{\cG}\left(\int_{\tilde\Gamma_C[x]}f(y)\hat \mu_x(dy)\right)m_0(dx).
\]
Consider $x\in\cG\setminus {\mathcal A}$ and $\hat y\in \textrm{supp}\, \hat \mu_x$.
Kuratowski theorem (\cite[Proposition 5.1.8]{AGS}) ensures that there exists a sequence $\{y_n\}_{n\in\N}$, with $y_n\in\textrm{supp}\, \hat \mu_n$, which converges to~$\hat y$ in the topology of~$\tilde\Gamma_C$. Let $x_n=e_0(y_n)$. Since $\hat \mu_n\in E(\mu_n)$, there holds: $y_n\in \tilde\Gamma_C^{\mu_n,\rm{opt}}[x_n]$. By Proposition~\ref{prop2}, we infer $\hat y\in \tilde\Gamma_C^{\mu,\rm{opt}}[x]$. By the arbitrariness of $\hat y\in\textrm{supp}\, \hat \mu_x$, we obtain $\textrm{supp}\, \hat \mu_x\subset \tilde\Gamma_C^{\mu,\rm{opt}}[x]$ and consequently, by the arbitrariness of $x\in \cG\setminus {\mathcal A}$, that $\hat \mu$ belongs to~$E(\mu)$. 
\hfill\qed

%
%section: mild solutions
%
\subsection{Proof of Theorem~\ref{thm:MFGlip} }\label{subsect:proofmain2}
This paragraph contains the proof of Theorem~\ref{thm:MFGlip}. We proceed adapting the proof of Theorem \ref{sec:setting-notation-6} and using some ideas from \cite[Theorem 4.1]{CCC1}. We consider the multivalued map~$E$, defined in~\eqref{eq:mappa} which has the closed graph property (see point~$(ii)$ in the proof of Theorem~\ref{sec:setting-notation-6}). We then introduce the multivalued map~$E_0$ as the restriction of~$E$ to the set $\cPmzeroGammaLipCV$ where $C$ and $V$ are chosen as in the statement of the theorem. The proof consists of checking that $E_0$ fulfills the hypotheses of Kakutani fixed point theorem. To this end, we need to check that
\begin{itemize}
\item[(i)] $E_0(\mu)\subset \cPmzeroGammaLipCV$, $\forall\mu\in \cPmzeroGammaLipCV$
\item[(ii)] $\cPmzeroGammaLipCV$ is compact
\item[(iii)] $E_0(\mu)$ is a not empty convex set, $\forall\mu\in \cPmzeroGammaLipCV$
\item[(iv)] $E_0$ has the closed graph property.
\end{itemize}
Let us successively address the four properties.\\
$(i)$. Consider $\mu\in \cPmzeroGammaLipCV$. As in the proof of Theorem \ref{sec:setting-notation-6}, we see that $\tilde\Gamma_C^{\mu,\rm{opt}}[x]\ne\emptyset$ for any $x\in\textrm{supp}(m_0)$. On the other hand, from Theorem~\ref{optcur_lip}, $\tilde\Gamma_C^{\mu,\rm{opt}}[x]\subset \Gamma^{\textrm{Lip}}_{C,V}[x]$. Hence, $\textrm{supp}(\hat\mu)\subset \Gamma^{\textrm{Lip}}_{C,V}$, for any $\hat\mu\in E_0(\mu)$. Invoking \cite[Lemma 4.1]{CCC1}, we get: $\hat\mu\in \cPmzeroGammaLipCV$ and the proof of $(i)$ is achieved.\\
$(ii)$. From Lemma~\ref{lemma:Gcomp}, the set $\tilde \Gamma_C$ is compact. Then, from Ascoli-Arzel\`a Theorem,  $\Gamma^{\textrm{Lip}}_{C,V}$ is compact. From Prokhorov theorem, $\cP^{\textrm{Lip}}(\Gamma^{\textrm{Lip}}_{C,V})$ is compact so, in particular, $\cPmzeroGammaLipCV$ is compact.\\
$(iii)$ and $(iv)$. These properties have already been obtained in the proof of Theorem \ref{sec:setting-notation-6}. We refer the reader to that proof for the details.\\
In conclusion, by Kakutani theorem, there exists a fixed point of the map~$E_0$, namely a  relaxed MFG equilibrium in $\cPmzeroGammaLipCV$.  \hfill\qed

\subsection{Mild solutions}
%Assume the hypotheses of Section~\ref{sect:MFGequil}. Then

Let $\mu\in\cP_{m_0}(\Gamma)$ be a relaxed MFG equilibrium whose existence is guaranteed by 
Theorem~\ref{sec:setting-notation-6}. We consider the value function naturally associated to $\mu$:
\begin{equation}\label{eq:4.1}
u(x,t)= \inf_{( y_x, \alpha)\in \Gamma_t[x]}  J^\mu_t(x;( y_x, \alpha) ),
\end{equation}
where $J^\mu_t$ is the cost defined in~\eqref{costMFG}.
\begin{definition}\label{def:mildsol}
Let $\mu\in\cP_{m_0}(\Gamma)$ be a relaxed MFG equilibrium. The pair $(u,m)$ is the associated mild solution if $u$ is the value function defined in~\eqref{eq:4.1} and $m\in C([0,T],\cP(\cG))$ is defined by $m(t)=e_t\#\mu$.
\end{definition}
\begin{remark}\label{rmk:reg_m}
Lemma~\ref{lemma:3.1bis} ensures that $m\in C^{1/2}([0,T],\cP(\cG))$.
%{\color{red} If furthermore $\mu\in \cP^{\rm{Lip}}_{m_0}(\Gamma_{C,V}^{\rm{Lip}})$, then $m\in \rm{Lip} ([0,T],\cP(\cG))$.}
\end{remark}

For simplicity of notations, we set
\begin{equation}\label{eq:cost_ell_MFG}
\ell_i(x,t)=L_i[m^{\mu}(t)](x,t)\qquad\textrm{and} \qquad g_i(x)=G_i[m^{\mu}(T)](x) \qquad\forall (x,t)\in \cG\times[0,T]
\end{equation}
and we shall also use the abridged notation $L$ as in~\eqref{eq:462}. 
The costs $\ell_i$ are those  payed by the agents in the MFG. By Lemma~\ref{lemma:3.1bis}, the functions~$\ell_i$ fulfill assumption~$(H_1)$. 

The purpose of this section is to derive several properties of the value function from the results of Section~\ref{sec:OC}.
As a preliminary step, invoking Proposition~\ref{prp:prop_vf}, Proposition~\ref{prp:U_cont}, Remark~\ref{rmk:U_holder} and Lemma~\ref{lemma:3.1bis}, we obtain the following proposition:
%
%PER VERSIONE FINALE: togliere la Proposition~\ref{prp:prop_u} e metterne la parte (b) nel Theorem~\ref{thm:mildMFG}
%
\begin{proposition}\label{prp:prop_u}
Under the  assumptions of Theorem~\ref{sec:setting-notation-6}, the value function~$u$ defined in~\eqref{eq:4.1} has the following properties:
\begin{itemize}
\item[(i)] (Dynamic programming principle) 
\[
u(x,t)= \inf_{( y_x, \alpha)\in \Gamma_{t,\bar t}[x]}\left\{
u(y_x(\bar t),\bar t)+\int_t^{\bar t} \left( L(y_x(\tau),\tau) +\frac{|\alpha(\tau)|^2}{2} \right)d\tau
\right\}
\]
where
\[
\Gamma_{t,\bar t}[x]=\left\{
\begin{array}[c]{ll}
(y_x, \alpha)  \ds \in L^2([t,\bar t],M): \quad   & y_x\in W^{1,2} ([t,\bar t];\cG), \\
&\ds  y_x(s)=x+\int_t^s \alpha(\tau) d\tau \quad  \hbox{ in }[t,\bar t]
\end{array}\right\};
\]
\item[(ii)] the value function is continuous in~$\cG\times[0,T)$.
\end{itemize}
\end{proposition}
Applying Theorem~\ref{thm:HJ}, it can now be proved that $u$ solves the HJ problem associated with the costs~$\ell_i$.

\begin{theorem}\label{thm:mildMFG}
Under the same assumptions as in Theorem~\ref{sec:setting-notation-6}, the value function~$u$ defined in~\eqref{eq:4.1} is a solution to \eqref{HJ} with the costs~$\ell_i$ defined in~\eqref{eq:cost_ell_MFG}. Moreover, for all $x\in \cG$, $t\mapsto u(x,t)$ is continuous in $ [0,T]$ and $u(x,T)=g(x)$ with the costs $g_i$  defined in~\eqref{eq:cost_ell_MFG}.% Moreover, if $g$ is uniformly continuous on~$\cG$, then~$u$ is the unique solution.
\end{theorem}
\begin{corollary}\label{cor:MFGulip}
Under the same assumptions as in Theorem~\ref{thm:MFGlip},  %,Besides the hypotheses of Theorem~\ref{thm:mildMFG}, assume also~\eqref{eq:431}. % that $G_i[m]\in C^2(J_i)$ and, for $i=1,\dots, N$, $L_i[m](\cdot,t)\in C^2(J_i)$ and that there exists $M>0$ such that
%\begin{multline}\label{eq:9999mfg}
%\| G[m]\|_{L^\infty(\cG )}+ \max_{i=1,\dots, N} \| \partial_x G_i[m]\|_{L^\infty( J_i )}+ \| L[m]\|_{L^\infty(\mathcal G \times [t,T] )} \\+ \max_{i=1,\dots, N} \| \partial_x L_i[m]\|_{L^\infty( J_i\times [t,T] )}\leq M
%\end{multline}
%for any $m\in\cP(\cG)$.
%Let $\mu\in\cP_{m_0}(\Gamma)$ be a relaxed MFG equilibrium and respectively $(u,m)$ the associated mild solution. Then,
there holds
\begin{itemize}
\item[(a)] $u$ is locally Lipschitz continuous in~$\cG\times[0,T)$
\item[(b)] if, moreover, $G[m]$ is Lipschitz continuous for every $m\in\cP(\cG)$, then ~$u$ is locally Lipschitz continuous in~$\cG\times[0,T]$.
\end{itemize}
\end{corollary}
\begin{proof}
Assumption~\eqref{eq:431} entails that the costs $\ell_i$ and $g_i$ associated to~$\mu$ in~\eqref{eq:cost_ell_MFG} fulfill the assumption of Theorem~\ref{optcur_lip}. Hence, for proving points~$(a)$ and~$(b)$, it is enough to apply respectively Proposition~\ref{prop:loclip} and Corollary~\ref{cor:ulipT}.
\end{proof}
\begin{remark}[Uniqueness of the mild solution] 
We say that $F:\cG\times\cP(\cG)\to \R$ is monotone if, for any $m_1,m_2\in \cP(\cG)$, there holds $\int_{\cG} (F(x,m_1)-F(x,m_2))(m_1-m_2)(dx)\ge 0$. The strict monotonicity holds if furthermore $\int_{\cG} (F(x,m_1)-F(x,m_2))(m_1-m_2)(dx)=0$ if and only if $F(\cdot,m_1)\equiv F(\cdot,m_2)$.
 If for all $t$, $(x,m)\mapsto L[m](x,t)$ and $(x,m)\mapsto G[m](x)$ are strictly monotone, then it can be proved with the same arguments as in \cite[Theorem 4.1 and Remark 4.1]{CC}, that if $(u_1,m_1)$ and $(u_2,m_2)$ are mild solutions respectively associated to two  relaxed equilibria $\mu_1$ and $\mu_2$,  then $u_1=u_2$. Under a more restrictive monotonicity assumption on $L$, it can also be proved that $m_1=m_2$.
\\
It is worth noticing that the uniqueness of the mild solution does not imply the uniqueness of the relaxed MFG equilibrium as  shown in the following example.
\end{remark}

\begin{example}\label{exa:no!MFGeq}
Let us exhibit two probabilities $\mu_1,\mu_2\in\cP(\Gamma)$ such that
\begin{equation}\label{eq:no!MFGeq}
\mu_1\ne\mu_2,\qquad\textrm{and} \qquad e_t\#\mu_1=e_t\#\mu_2,\qquad\forall t\in[0,T].
\end{equation}
For $0<t_1\leq t_2<T$, consider four paths $\gamma_i\in \Gamma$ ($i=1,\dots,4$) such that 
\begin{equation*}
\begin{array}{ll}
\gamma_1=\gamma_2\qquad\textrm{and}\qquad \gamma_3=\gamma_4 &\qquad \textrm{on }[0,t_1]\\
\gamma_1=\gamma_2=\gamma_3=\gamma_4 &\qquad \textrm{on }[t_1,t_2]\\
\gamma_1=\gamma_3\qquad\textrm{and}\qquad \gamma_2=\gamma_4 &\qquad \textrm{on }[t_2,T],
\end{array}
\end{equation*}
and such that $\gamma_1$ does not coincide with $\gamma_3$ on $[0,t_1]$ and with $\gamma_2$ on $[t_2,T]$.
Then \eqref{eq:no!MFGeq} holds for the probabilities on $\Gamma$ defined by
\begin{equation*}
\mu_1=\frac14\sum_{i=1}^4 \delta_{\gamma_i},\quad \hbox{ and} \quad \mu_2=\frac12 \delta_{\gamma_1} + \frac12 \delta_{\gamma_4}.
\end{equation*}
\end{example}
Let us provide examples of strictly monotone operators.
%{\color{blue}\begin{example}\label{exa:no!MFGeq1}
%Set 
%\begin{equation*}
%L[m](x) =\sum_{i=1}^N a_i m(J_i\setminus\{O\}) \car_{x\in J_i\setminus%\{O\}}  
%\end{equation*}  
%for a collection $(a_i)_{1\le i \le N}$ of positive weights.
%Clearly $L[m]$ is a bounded and  lower semi-continuous function on %$\cG$, whose restriction to $J_i\setminus\{O\}$ is constant.
%$L$ is strictly monotone. Indeed,
%\begin{equation*}
%    \int_\cG (L[m_1](x)-L[m_2](x)) (m_1-m_2)(dx)=
%    \sum_{i=1}^N a_i (m_1(J_i\setminus\{O\}) -m_2(J_i\setminus\{O\})) ^2 \ge 0, \end{equation*} and $ \int_\cG (L[m_1](x)-L[m_2](x)) (m_1-m_2)(dx) =0$ if and only if $L[m_1]=L[m_2]$.
%\\
%Note that it is easy to find two different probability measures $m_1$ and $m_2$ on $\cG$ such that $L[m_1]=L[m_2]$.
%\end{example}
%}

\begin{example}\label{exa:no!MFGeq2}
%We provide an example of monotone operators on the network adapting some ideas of~\cite[Example 4.1]{CC}. To this end, we first need some notations. 
Fix a function $\kappa\in C^\infty_0(\R)$ with $0\leq \kappa \leq 1$, $\kappa'(0)=\kappa''(0)=0$ and let $K:\cG\times\cG\to \R$ be  defined by $K(x,y)= \kappa(d(x,y))$ where $d$ is the distance in~$\cG$.
The function $K$ has the following properties:
\begin{equation*}\label{eq:26giu_1}
K(x,y)=K(y,x), \quad \forall x,y\in\cG.
%\quad\textrm{and}\quad 
%|K(x,y)-K(\bar x,y)|\leq \|\kappa'\|_\infty d(x,\bar x) \qquad\forall x,\bar x,y\in\cG.
\end{equation*}
For any $m\in\cP(\cG)$, define $K*m:\cG\to\R$  by 
\begin{equation*}
K*m(x)=\int_{\cG} K(x,y) m(dy).
%\sum_{i=1}^N\int_{J_i\setminus\{O\}}\phi(x,y)\car_{y\in J_i\setminus\{O\}} m(dy)+\phi(x,O)m(\{O\}).
\end{equation*}
Let the running cost $L:\cP(\cG)\times\cG\to \R$ be defined by
\begin{equation*}
L[m](x)=\int_\cG \ell \left(y,K*m(y)\right) K(x,y)dy
\end{equation*}
where $\ell:\cG\times \R\to\R$ is a smooth function such that $w\mapsto \ell(y,w)$ is strictly increasing for every $y\in\cG$. It is standard that 
\begin{equation}\label{eq:26_giu_2}
\int_{\cG} (L[m_1](x)-L[m_2](x)) (m_1-m_2)(dx)\ge 0,
%\sum_{i=1}^N\int_{J_i\setminus\{O\}}\left(L[m_1](x)-L[m_2](x)\right) (m_1-m_2)(dx)\\+\left(L[m_1](O)-L[m_2](O)\right) (m_1-m_2)(\{O\})\geq 0
\end{equation}
and the equality in \eqref{eq:26_giu_2} holds true if and only if $L[m_1](x)=L[m_2](x)$ for every $x\in\cG$.

Note that  $x\mapsto L[m](x)$ is continuous on $\cG$, and $C^2$ on $J_i\setminus\{0\}$, $i=1,\dots, N$.

\end{example} 

\begin{example}\label{exa:no!MFGeq3}
With $L$ defined in Example \ref{exa:no!MFGeq2},
consider 
\begin{equation*}
\widetilde{L}[m](x)=
L[m](x) + \sum_{i=1}^N a_i  \car_{x\in J_i\setminus\{O\}}  
\end{equation*}  
for a collection $(a_i)_{1\le i \le N}$ of positive weights.
Clearly $\widetilde L$ is strictly monotone, Lipschitz continuous w.r.t. $m$ (for the Wasserstein distance ${\hbox{Wass}}_1$) and  $\widetilde L[m]$ is discontinuous in $x$ at $O$.
\end{example} 

\subsection{Regularity of~$u$ in the interior of the edges}
In what follows, we collect several properties of the value function in a mild solution, starting with easy consequences of the results contained in Section~\ref{sec:OC}. Then we aim at obtaining more accurate information at the points $(x,t)$ such that $x$ lies in the support of $m(t)$, i.e. the points that are actually hit by optimal trajectory.
\begin{lemma}\label{lemma:reg_u_mfg}
We make the same assumptions as in Theorem~\ref{thm:MFGlip}. Let $\mu\in\cP_{m_0}(\Gamma)$ be a relaxed MFG equilibrium and $(u,m)$ be  the related mild solution. 
\begin{itemize}
\item[(a)] The function $u$ is locally semi-concave in $(J_i\setminus\{O\})\times(0,T)$
\item[(b)] Lemma~\ref{lemma:note4.9} holds replacing $\Gamma_t^{\textrm{opt}}[x]$ with $\Gamma_t^{\textrm{opt},\mu}[x]$, in particular, the characterization of the optimal control with the $\partial_x u$ away from the vertex
\item[(c)] Corollary \ref{cor:10000}  holds replacing $\Gamma^{\textrm{opt}}[x]$ with $\Gamma^{\textrm{opt},\mu}[x]$ 
\item[(d)] Lemma~\ref{lemma:OS} holds.
\end{itemize}
\end{lemma}
\begin{proof}
Assumption~\eqref{eq:431} entails that the costs $\ell_i$ and $g_i$ associated to~$\mu$ in~\eqref{eq:cost_ell_MFG} fulfill the assumption of Theorem~\ref{optcur_lip}. We end the proof by applying Proposition~\ref{prop:semiconc_loc}, Lemma~\ref{lemma:note4.9},  
Corollary \ref{cor:10000} 
and Lemma~\ref{lemma:OS}.
\end{proof}
Next, let us prove that $u$ is a bilateral subsolution (see \cite[Definition III.2.27]{BCD}) of the Hamilton-Jacobi equation and is differentiable at least at the points~$(x,t)$ such that $x$ belongs to the support of $m(t)$ and does not coincide with~$O$. To this end, some new notations are useful. Set
\begin{eqnarray*}
\textrm{supp}(m(t))&=&\left\{x\in\cG\;:\; \forall \omega, \textrm{ open neighborhood of }x,\quad m(t)(\omega)>0\right\}\\
Q_m&=&\left\{(x,t)\in\cG\times(0,T)\;:\; x\in(\cG\setminus\{O\})\cap \textrm{supp}(m(t))\right\}\\
\partial Q_m&=&\left\{(O,t)\;:\; t\in(0,T],\quad O\in \textrm{supp}(m(t))\right\}
\end{eqnarray*}
and introduce the subdifferential of $u$ at $(x,t)\in Q_m$ as
\begin{equation*}
D^+u(x,t)=\left\{(\pi,q)\in \R^2\;:\; \limsup_{y\to x,\theta\to t}\frac{u(y,\theta)-u(x,t)-\pi(\theta-t)-q(\bar y-\bar x)}{|\bar y-\bar x|+|\theta-t|}\leq 0\right\},
\end{equation*}
where $x=\bar x e_j$, $y=\bar ye_j$ (note that   $j$ is uniquely defined, from the definition of~$Q_m$).
\begin{remark}\label{rmk:5_13apr}
Similar arguments as those in the proof of \cite[Theorem 4.5]{CCC2} yield that for any $t\in(0,T)$, for $\mu$-a.e. $\gamma\in\G$, the point $(\gamma(t),t)$ belongs to $Q_m\cup \partial Q_m$.
\end{remark}
\begin{proposition}\label{prp:CCC2_4.1+4.2}
The assumptions are those of Theorem \ref{thm:MFGlip} and Corollary \ref{cor:MFGulip}. Let $\mu\in\cP_{m_0}(\G)$ be a relaxed MFG equilibrium and $(u,m)$ be the associated mild solution. Then, for any $(x,t)\in Q_m$, 
\begin{itemize}
\item[$(a)$] there holds
\begin{equation*}
-p_1+H(x,t,p_2)=0\qquad\forall(p_1,p_2)\in D^+u(x,t)
\end{equation*}
\item[$(b)$] $u$ is differentiable at $(x,t)$.
\end{itemize}
\end{proposition}
\begin{proof}
\noindent $(a)$. The arguments are reminiscent of those used in ~\cite[Theorem 4.1]{CCC2}. Fix $(x,t)\in Q_m$ and consider $(p_1,p_2)\in D^+u(x,t)$. Without any loss of generality, let us assume that $x\in J_i$ with $x=\bar x e_i$. From Theorem~\ref{thm:mildMFG}, $u$ is a viscosity solution to problem~\eqref{HJ} with the cost $\ell_i$ defined in~\eqref{eq:cost_ell_MFG} and $g=G[m^\nu(T)]$. Hereafter, for simplicity, we refer to \eqref{HJ} as {\sl the HJ-problem}. Since $u$ is a viscosity subsolution to the HJ-problem, 
\begin{equation*}
-p_1+H(x,t,p_2)\leq 0.
\end{equation*}
Let us now prove the reverse inequality. Since $x\in\textrm{supp}(m(t))$, there exists a trajectory $(\gamma,\gamma')\in \Gamma^{{\rm opt}}[\gamma(0)]$ with $\gamma(t)=x$. Let $r$ be small enough such that $\gamma(t-s)\in J_i\setminus\{O\}$ for every $s\in[0,r]$; we write $\gamma(t-s)=\bar \gamma(t-s)e_i$. From the  definition of the subdifferential, 
\begin{equation*}
u(\gamma(t-s),t-s)-u(x,t)\leq -p_1s-p_2(x-\bar \gamma(t-s))+o(r)=-p_1s-p_2\int_{t-s}^t\gamma'(\tau)d\tau+o(r).
\end{equation*}
On the other hand, from Remark~\ref{rmk:restr_OC}, $(\gamma_{|[t-s,T]},\gamma'_{|[t-s,T]})$ and $(\gamma_{|[t,T]},\gamma'_{|[t,T]})$ belong respectively to $\Gamma^{{\rm opt}}_{t-s}[\gamma(t-s)]$ and to $\Gamma^{{\rm opt}}_{t}[\gamma(t)]$. Hence, 
\begin{equation*}
u(\gamma(t-s),t-s)-u(x,t)= \int_{t-s}^t\left(\frac{|\gamma'(\tau)|^2}2+\ell[m(\tau)](\gamma(\tau))\right)d\tau.
\end{equation*}
The latter two observations yield
\begin{equation*}
\int_{t-s}^t\left(\frac{|\gamma'(\tau)|^2}2+\ell_i[m(\tau)](\gamma(\tau))\right)d\tau\leq -p_1s-p_2\int_{t-s}^t\gamma'(\tau)d\tau+o(r).
\end{equation*}
Next, the regularity of~$m$ (see Theorem~\ref{thm:MFGlip}) and of~$\gamma$ in~$(t-s,t)$ (see the Euler-Lagrange relation in Lemma~\ref{lemma:EL}) entail
\begin{equation*}
\ell_i[m(\tau)](\gamma(\tau))=\ell_i[m(t)](x)+O(s),\quad\textrm{and}\qquad
\gamma'(\tau)= \gamma'(t)+o(1)
\end{equation*}
for any $\tau\in(t-s,t)$. This implies
\begin{equation*}
\frac{1}{s}\int_{t-s}^t\left(\frac{|\gamma'(t)|^2}2+\ell_i[m(t)](x)\right)d\tau+o(1)\leq -p_1-p_2\gamma'(t)+o(1).
\end{equation*}
Letting $s\to 0^+$, we infer
\begin{equation*}
0\leq -p_1-p_2\gamma'(t)-\frac{|\gamma'(t)|^2}2-\ell_i[m(t)](x)\leq -p_1+H(x,t,p_2)
\end{equation*}
where the last inequality comes from the  definition of~$H$, see \eqref{eq:1dic_easy0}.\\
\noindent $(b)$. Point $(b)$ is obtained with the arguments in the proof of \cite[Proposition 4.2]{CCC2} replacing \cite[Theorem 4.1]{CCC2} and 
\cite[Corollary 4.1]{CCC2} respectively with  point~$(a)$ and Lemma~\ref{lemma:reg_u_mfg}-$(a)$.
\end{proof}

\subsection{Properties of $m$}
Consider a mild solution $(u,m)$ associated to a given relaxed MFG equilibrium~$\mu\in \cPmzeroGammaLipCV$. Here, we wish to investigate the behaviour of the  point masses of $m$ if they exist.\\
Let us recall from Example~\ref{exa:dirac} and Example~\ref{exa:moving_Dirac} that $m$ may develop a singularity of the form of a point mass at the origin and that the latter singularity may be transported into the edges. Below, we prove that each singular point conserves its mass when it travels in the interior of an edge. This implies that point masses cannot
appear/vanish in the interior of a given edge. In particular, the creation of a point mass can occur only at the vertex.
%in particular, the point mass can vary in time only if is located at the vertex.
%Moreover, 
%it can disappear only at the vertex.
Finally, we  provide an example with two vertices in which  $m$ is a Dirac mass at the first vertex until some time $t_1$, a Dirac mass at the second vertex after $t_2>t_1$, and in which there is no mass points  between the two vertices at all $t$, $t_1<t<t_2$.

\begin{theorem}\label{thm:sing_in_edge}
Under the assumptions of Theorem~\ref{thm:MFGlip}, let $\mu\in \cPmzeroGammaLipCV$ be a relaxed MFG equilibrium and $(u,m)$ be the corresponding mild solution. Consider 
%$x\in\textrm{supp}(m(t))\setminus\{O\}$ 
$x\in\textrm{supp}(m(t))\cap J_j\setminus\{O\}$ for some $j=1,\dots,N$ 
and  $t\in(0,T)$. The following holds:
\begin{itemize}
\item[$(a)$] there exists $x_0\in\textrm{supp}(m_0)$ and $\gamma\in \Gamma^{\textrm{Lip}}_{C,V}$ with $(\gamma,\alpha_\gamma)\in \Gamma^{\mu,\textrm{opt}}[x_0]$ and $\gamma(t)=x$ (recall that the control $\alpha_\gamma$ was introduced in Remark~\ref{rmk:alpha_y})
\item[$(b)$] for $i=1,2$, consider $x_{0,i}$ and $\gamma_i$ satisfying point~$(a)$ and denote $t_{*,i}=\sup\{s\in[0,t]\:;\: \gamma_i(s)=O\}$ and $t_{*,i}=0$ if the latter set is empty and, similarly, $t^{*,i}=\inf\{s\in[t,T]\:;\: \gamma_i(s)=O\}$ and $t^{*,i}=T$ if the latter set is empty. Then, there holds
\begin{equation*}
t_{*,1}=t_{*,2}=:t_* \qquad t^{*,1}=t^{*,2}=:t^* \qquad\textrm{and} \qquad \gamma_1=\gamma_2\quad\textrm{on }(t_{*},t^*)
\end{equation*}
\item[$(c)$] for every $s\in(t_*,t^*)$ and every $\gamma$ as in point~$(a)$, there holds: $m(s)(\{\gamma(s)\})=m(t)(\{x\})$.
\end{itemize}
\end{theorem}
\begin{proof}
$(a)$.
For every positive $r$,  $m(t)(B(x,r))>0$. Hence, 
\begin{eqnarray*}%\label{eq:2_27apr}
0&<&m(t)(B(x,r))=\int_{\cG}\car_{\{\xi\in\cG \cap B(x,r)\}}m(t)(d\xi)\\
&=&\int_{\{y\in\Gamma^{\textrm{Lip}}_{C,V}\;:\;(y,\alpha_y)\in\Gamma^{\mu,\textrm{opt}},\, y(0)\in\textrm{supp}(m_0)\}}\car_{\{y\;:\;y(t)\in B(x,r)\}}\mu(dy)
\end{eqnarray*}
where the last equality comes from the definition of $m$.
Consequently the set
\begin{equation*}
E=\left\{y\in\Gamma^{\textrm{Lip}}_{C,V}\;:\;(y,\alpha_y)\in\Gamma^{\mu,\textrm{opt}},\, y(0)\in\textrm{supp}(m_0),\,y(t)\in B(x,r)\right\}
\end{equation*}
is not empty for every  $r>0$.
We infer that there exist a sequence $\{x_n\}_n$, with $x_n\in B(x,1/n)$ and a sequence $\gamma_n\in \Gamma^{\textrm{Lip}}_{C,V}$ with $(\gamma_n,\alpha_{\gamma_n})\in\Gamma^{\mu,\textrm{opt}}$, $\gamma_n(0)\in\textrm{supp}(m_0)$ and $\gamma_n(t)=x_n$.
By standard arguments, we see that, as $n\to \infty$, $\gamma_n$ uniformly converge to some path $\gamma\in \Gamma^{\textrm{Lip}}_{C,V}$ with $\gamma(0)\in\textrm{supp}(m_0)$ and $\gamma(t)=x$. From the  stability of optimal trajectories, $(\gamma,\alpha_\gamma)$ belongs to $\Gamma^{\mu,\textrm{opt}}$. Point $(a)$ is proved.\\
%By our assumptions and by definition of $m$, we have
%\begin{equation}\label{eq:1_27apr}
%0<m(t)(\{x\})=\int_\Gamma \car_{\{\gamma:\gamma(t)=x\}}\mu(d\gamma)=\int_{\{\gamma\in\Gamma^{\textrm{Lip}}_{C,V}\;:\;(\gamma,\alpha_\gamma)\in\Gamma^{\mu,\textrm{opt}},\, \gamma(0)\in\textrm{supp}(m_0)\}}\car_{\{\gamma:\gamma(t)=x\}}\mu(d\gamma);
%\end{equation}
%in particular, we deduce that the set $\{\gamma\in\Gamma^{\textrm{Lip}}_{C,V}\;:\;(\gamma,\alpha_\gamma)\in\Gamma^{\textrm{opt},\mu},\, \gamma(0)\in\textrm{supp}(m_0),\, \gamma(t)=x\}$ is not empty. Point~$(a)$ is proved.\\
$(b)$. It is a direct consequence of Lemma~\ref{lemma:reg_u_mfg}-$(c)$.\\
$(c)$. Consider $s\in(t_*,t^*)$. % because for $s=t$ there is nothing to prove.\\
%By point~$(a)$, there exists $\gamma$ with the desired properties. Let $t_*$ and $t^*$ be the times defined in the statement associated to~$\gamma$. 
The definition of~$m$ entails
\begin{eqnarray}\label{eq:2_27apr}
m(t)(\{x\})%&=&\int_{\cG}\car_{\{\xi\in\cG\;:\;\xi=\gamma(t)\}}m(t)(d\xi)\\
&=&\int_{\{y\in\Gamma^{\textrm{Lip}}_{C,V}\;:\;(y,\alpha_y)\in\Gamma^{\mu,\textrm{opt}},\, y(0)\in\textrm{supp}(m_0)\}}\car_{\{y\;:\;y(t)=\gamma(t)\}}\mu(dy)
\end{eqnarray}
where $\alpha_y$ is the control defined in Remark~\ref{rmk:alpha_y}.
From point~$(b)$, there holds
\begin{multline*}
\{y\in\Gamma^{\textrm{Lip}}_{C,V}\;:\;(y,\alpha_y)\in\Gamma^{\mu,\textrm{opt}},\, y(0)\in\textrm{supp}(m_0),\, y(t)=\gamma(t)\}\\=\{y\in\Gamma^{\textrm{Lip}}_{C,V}\;:\;(y,\alpha_y)\in\Gamma^{\mu,\textrm{opt}},\, y(0)\in\textrm{supp}(m_0),\, y(s)=\gamma(s)\}
\end{multline*}
for every $s\in(t_*,t^*)$. Combining the latter identity  and \eqref{eq:2_27apr} yields
\begin{equation*}
m(t)(\{x\})=\int_{\{y\in\Gamma^{\textrm{Lip}}_{C,V}\;:\;(y,\alpha_y)\in\Gamma^{\mu,\textrm{opt}},\, y(0)\in\textrm{supp}(m_0)\}}\car_{\{y\;:\;y(s)=\gamma(s)\}}\mu(dy)=m(s)(\{\gamma(s)\})
\end{equation*}
for every $s\in(t_*,t^*)$ which is our statement.
\end{proof}
The following result is  direct consequence of Theorem~\ref{thm:sing_in_edge}-$(c)$.
\begin{proposition}\label{cor:sin_in_edge}
Under the hypotheses of Theorem~\ref{thm:MFGlip}, let $\mu\in \cPmzeroGammaLipCV$ be a relaxed  MFG equilibrium and $(u,m)$ be the corresponding mild solution. For every $\gamma$ as in Theorem~\ref{thm:sing_in_edge}-$(a)$ such that there exists $t\in(0,T)$ with $\gamma(t)\in J_i\setminus\{O\}$, there holds
\begin{equation*}
m(s)(\{\gamma(s)\})=m(s')(\{\gamma(s')\}) \qquad \forall s,s'\in(t_*,t^*)
\end{equation*}
where
\begin{equation*}
t_*=\sup\{s\in[0,t]\;:\; \gamma(s)=O\},\qquad t^*=\inf\{s\in[t,T]\;:\; \gamma(s)=O\}
\end{equation*}
($t_*=0$ and respectively $t^*=T$ when the corresponding set is empty).  This implies that if $m$ has a point mass, then the latter is  conserved as long as it stays in the interior of a given edge.
%Consider $x\in J_i\setminus\{O\}$ for some $i=1,\dots,N$ and $t\in(0,T)$. 
%\begin{itemize}
%\item [(a)] If  $m(t)(\{x\})>0$, then, for every $\gamma$ as in Theorem~\ref{thm:sing_in_edge}-$(a)$, there holds:
%\begin{equation*}
%m(s)(\{\gamma(s)\})=m(t)(\{x\}) \qquad \forall s\in(t_*,t^*)
%\end{equation*}
%where
%\begin{equation*}
%t_*:=\sup\{s\in[0,t]\;:\; \gamma(s)=O\},\qquad t^*:=\inf\{s\in[t,T]\;:\; \gamma(s)=O\}
%\end{equation*}
%($t_*=0$ and respectively $t^*=T$ when the corresponding set is empty).  This means that if $m$ has a point mass, then the latter is  conserved as long as it stays in the interior of a given edge.
%\item [(b)] if  $m(t)(\{x\})=0$ and there exists $\gamma$ as in Theorem~\ref{thm:sing_in_edge}-$(a)$, then there holds
%\begin{equation*}
%m(s)(\{\gamma(s)\})=0%m(t)(\{x\}) 
%\qquad \forall s\in(t_*,t^*)
%\end{equation*}
%where $t_*$ and $t^*$ are defined as in point~$(a)$.
%\item[(c)] if  $m(t)(\{x\})=0$ and point~$(a)$ in Theorem~\ref{thm:sing_in_edge} does not hold, %$\nexists$ $\gamma$ as in Theorem~\ref{thm:sing_in_edge}-$(a)$,
%then there exists $\delta>0$ such that
%\begin{equation*}
%m(s)(\{x\})=0 \qquad \forall s\in(t-\delta,t+\delta).
%\end{equation*}
%\end{itemize}
\end{proposition}
%\begin{proof}
%$(a)$. .\\
%, there exists $\gamma$ with the desired properties. Let $t_*$ and $t^*$ be the times defined in the statement associated to~$\gamma$. The definition of~$m$ entails
%\begin{eqnarray*}%\label{eq:2_27apr}
%m(t)(\{x\})&=&\int_{\cG}\car_{\{\xi\in\cG\;:\;\xi=\gamma(t)\}}m(t)(d\xi)\\
%&=&\int_{\{y\in\Gamma^{\textrm{Lip}}_{C,V}\;:\;(y,\alpha_y)\in\Gamma^{\mu,\textrm{opt}},\, y(0)\in\textrm{supp}(m_0)\}}\car_{\{y\;:\;y(t)=\gamma(t)\}}\mu(dy)
%\end{eqnarray*}
%where $\alpha_y$ is the control defined in Remark~\ref{rmk:alpha_y}.
%By Theorem~\ref{thm:sing_in_edge}-$(b)$, there holds
%\begin{multline*}
%\{y\in\Gamma^{\textrm{Lip}}_{C,V}\;:\;(y,y')\in\Gamma^{\mu,\textrm{opt}},\, y(0)\in\textrm{supp}(m_0),\, y(t)=\gamma(t)\}\\=\{y\in\Gamma^{\textrm{Lip}}_{C,V}\;:\;(y,\alpha_y)\in\Gamma^{\mu,\textrm{opt}},\, y(0)\in\textrm{supp}(m_0),\, y(s)=\gamma(s)\}
%\end{multline*}
%for every $s\in(t_*,t^*)$. Replacing the last relation in the previous one, we get
%\begin{equation*}
%m(t)(\{x\})=\int_{\{y\in\Gamma^{\textrm{Lip}}_{C,V}\;:\;(y,\alpha_y)\in\Gamma^{\mu,\textrm{opt}},\, y(0)\in\textrm{supp}(m_0)\}}\car_{\{y\;:\;y(s)=\gamma(s)\}}\mu(dy)=m(s)(\{\gamma(s)\})
%\end{equation*}
%for every $s\in(t_*,t^*)$ which is our statement.\\
%$(b)$. The same arguments as those in the proof of Theorem~\ref{thm:sing_in_edge}-$(c)$ can be used even if $x\notin \textrm{supp}(m(t))$.\\
We now focus on the case when no optimal trajectory hits $x$ at time $t$.
%state a result on the transportation of the (null) mass of a point~$(x,t)$ where no optimal trajectory passes. 
\begin{proposition}\label{prp:0mass}
Let $\mu\in \cPmzeroGammaLipCV$ be a relaxed  MFG equilibrium and $(u,m)$ be the corresponding mild solution. If $m(t)(\{x\})=0$ and point~$(a)$ in Theorem~\ref{thm:sing_in_edge} does not hold, then there exists $\delta>0$ such that
\begin{equation*}
m(s)(\{x\})=0 \qquad \forall s\in(t-\delta,t+\delta).
\end{equation*}
\end{proposition}
\begin{proof}
From Theorem~\ref{thm:sing_in_edge}-$(a)$, $x\notin \textrm{supp}(m(t))$. Then, there exists a positive number $r$ such that $B(x,r)\cap \textrm{supp}(m(t))=\emptyset$ and, consequently, 
the set
\begin{equation*}
E=\left\{y\in\Gamma^{\textrm{Lip}}_{C,V}\;:\;(y,\alpha_y)\in\Gamma^{\mu,\textrm{opt}},\, y(0)\in\textrm{supp}(m_0),\,y(t)\in B(x,r)\right\}
\end{equation*}
is negligible for the measure~$\mu$.
Taking into account the uniform Lipschitz continuity of the optimal trajectories, we obtain that there exists a sufficiently small $\delta>0$ such that
\begin{equation*}
E(s):=\left\{y\in\Gamma^{\textrm{Lip}}_{C,V}\;:\;(y,\alpha_y)\in\Gamma^{\mu,\textrm{opt}},\, y(0)\in\textrm{supp}(m_0),\,y(s)\in B(x,r/2)\right\}\subset E
\end{equation*}
for every $s\in(t-\delta,t+\delta)$. Since $\mu$ is complete, $E(s)$ is also negligible for~$\mu$ and 
\begin{equation*}
m(s)(B(x,r/2))=0\qquad\forall s\in(t-\delta,t+\delta),
\end{equation*}
which achieves the proof.
%We now consider the case where $x\in \textrm{supp}(m(t))$. Our purpose is to prove that this case cannot occur namely that there exists a path $\gamma$ as in Theorem~\ref{thm:sing_in_edge}-$(a)$. Indeed, for every positive $r$, we have $m(t)(B(x,r))>0$. Hence, by the same arguments as before, there holds
%\begin{equation*}
%\left\{y\in\Gamma^{\textrm{Lip}}_{C,V}\;:\;(y,\alpha_y)\in\Gamma^{\mu,\textrm{opt}},\, y(0)\in%\textrm{supp}(m_0),\,y(t)\in B(x,r)\right\}\ne \emptyset \qquad\forall r>0.
%\end{equation*}
%We infer that there exists a sequence $\{x_n\}_n$, with $x_n\in B(x,1/n)$ and
%\begin{equation*}
%E_n:=\left\{y\in\Gamma^{\textrm{Lip}}_{C,V}\;:\;(y,\alpha_y)\in\Gamma^{\mu,\textrm{opt}},\, y(0)\in\textrm{supp}(m_0),\,y(t)=x_n\right\}\ne \emptyset.
%\end{equation*}
%We consider a sequence $\{y_n\}_n$ with $y_n\in E_n$ for each $n$. By standard theory we have that, as $n\to \infty$, $y_n$ uniformly converge to some path $y\in \Gamma^{\textrm{Lip}}_{C,V}$ with $y(0)\in\textrm{supp}(m_0)$ and $y(t)=x$. By the closed graph property for optimal trajectories, we obtain that $(y,\alpha_y)$ belongs to $\Gamma^{\mu,\textrm{opt}}$. Hence, we get the desired contradiction.
\end{proof}

We now provide an example with two vertices 
in which 
\begin{itemize}
    \item there is a Dirac mass at the first vertex which disappears 
    \item a Dirac mass arises at the second vertex
     \item  no Dirac mass travels in the edge between the two vertices.
\end{itemize}

\begin{example}\label{exa:sing_disapp} 
With the same network as in Example \ref{exa:moving_Dirac}, consider the costs which do not depend on $m$ (no interaction between the agents):
\begin{equation*}
L[m](x)=\left\{\begin{array}{ll}
K_L&\quad\textrm{if }x\in\cG\setminus J_2,\\ 1&\quad\textrm{if }x\in J_2\setminus\{V_1,V_2\},\\ 0&\quad\textrm{if }x\in \{V_1,V_2\},
\end{array}\right.\quad\hbox{and }\quad 
G[m](x)=\left\{\begin{array}{ll}
K_G&\quad\textrm{if }x\ne V_2,\\ 0 &\quad\textrm{if }x=V_2,
\end{array}\right.
\end{equation*}
where $K_L>1$ and $K_G$ are positive constants that will chosen later (note that they fulfill assumptions~\eqref{eq:460} and \eqref{eq:461}). The time horizon $T$ will be chosen later. Take $m_0=\delta_{V_1}$. It is obvious that every optimal trajectory $(y,\alpha)$ with $y(0)=V_1$ must remain at $V_1$ until a time $\tau_1\in(0,T]$, then  move inside $J_2$ if $\tau_1<T$ so to reach $V_2$ at some time $\tau_2\in(\tau_1,T]$ and finally remain at $V_2$ until  $T$. The constants  $T$, $K_L$ and $K_G$ will be chosen sufficiently large so that $\tau_1<T$.

From the Euler-Lagrange condition \eqref{EL_giugno}, there exists $c\in\R$ such that $\alpha(s)=c$ for $s\in (\tau_1,\tau_2)$ with $\tau_2=\tau_1+1/c$.
Hence, 
\begin{equation*}
y(s)=\left\{\begin{array}{ll}
V_1&\quad\textrm{for }s\in[0,\tau_1],\\
c (s-\tau_1)e_2&\quad\textrm{for }s\in(\tau_1,\tau_1+1/c],\\
V_2&\quad\textrm{for }s\in(\tau_2, T].
\end{array}\right.
\end{equation*}
Let us minimize the cost $J_0(V_1;(y,\alpha))$ with respect to $\tau_1$ and $c$. There holds
\begin{equation*}
J_0(V_1;(y,\alpha))=\int_{\tau_1}^{\tau_1+1/c}\left(\frac{c^2}{2}+1\right)ds=\frac{c}{2}+\frac{1}{c}.
\end{equation*}
Hence, the minimum of $J_0(V_1;(y,\alpha))$ is achieved by $c=\sqrt{2}$ independently of $\tau_1$.

Let us  take $T$ larger than $1+1/\sqrt{2}$ and  introduce the family  $\{y_\tau\}_{\tau\in[0,1]}$ \begin{equation*}
y_\tau(s)=\left\{\begin{array}{ll}
V_1&\quad\textrm{for }s\in[0,\tau],\\
\sqrt{2}(s-\tau) e_2&\quad\textrm{for }s\in(\tau,\tau+1/\sqrt{2}],\\
V_2&\quad\textrm{for }s\in(\tau+1/\sqrt{2}, T],
\end{array}\right.
\end{equation*}
which are all optimal from the above calculations. There exists  a positive constant $C$ sufficiently large such  each $y_\tau$ belongs  to $\tilde \Gamma_C$. Define the  measure $\mu$ on $\tilde \Gamma_C$ (defined in \eqref{eq:gamma_tilde_x})
as follows: for all Borel set $A\subset  \tilde \Gamma_C$,
\begin{equation*}
\mu(A)={\mathcal L} \left(\left\{\tau\in[0,1]\;:\; y_\tau\in A\right\}\right),
\end{equation*}
where ${\mathcal L}$ is the Lebesgue measure. The measure $\mu$ fulfills
\begin{itemize}
\item[$\bullet$] $\textrm{supp}(\mu)\subset  \Gamma_C^{\mu,\textrm{opt}}[V_1]$ 
\item[$\bullet$] $e_0\#\mu (\{V_1\})=\mu(\{\gamma\in \tilde \Gamma_C: \gamma(0)=V_1\})={\mathcal L} (\{ \tau\in[0,1]: y_\tau(0)=V_1\})=1=m_0(\{V_1\})$;
\end{itemize}
therefore, the measure $\mu$ is a relaxed MFG equilibrium. Let $(u,m)$ be the corresponding mild solution. We claim that for all 
$x\in\cG\setminus\{V_1,V_2\}$ and all $t\in [0,T]$
\begin{equation*}
m(t)(\{x\})=0.
\end{equation*}
Indeed, 
\begin{equation*}
m(t)(\{x\})=\mu(\{\gamma\in \tilde \Gamma_C\;:\; \gamma(t)=x\})={\mathcal L} \left(\left\{\tau\in[0,1]\;:\; y_\tau(t)=x\right\}\right)=0,
\end{equation*}
the last equality is true since the set $\left\{\tau\in[0,1]\;:\; y_\tau(t)=x\right\}$ contains at most one value.
\end{example}

\subsection{The continuity equation}
\label{cont_eq}

Consider a mild solution $(u,m)$ associated to some relaxed MFG equilibrium~$\mu\in \cPmzeroGammaLipCV$. Here, we make the same hypotheses as in Theorem~\ref{thm:MFGlip}, and
we study the evolution of the distribution $m(t)$.  We obtain that $m$ satisfies (in a suitable weak sense) a continuity equation in which the drift is given as the optimal feedback from the Hamilton-Jacobi equation. 
\begin{theorem}\label{thm:cont_eq}
Under the hypotheses of Theorem~\ref{thm:MFGlip}, let $\mu\in \cPmzeroGammaLipCV$ be a relaxed MFG equilibrium and $(u,m)$ be a related mild solution.  Then, for every $\phi\in C^\infty(\cG\times[0,T])$ such that $\textrm{supp}(\phi(\cdot,t))$ is contained in a compact subset of~$\cG$ independent of~$t$, there holds
\begin{multline}\label{eq:continuity}
m(t)(\{O\})\partial_t\phi(O,t)+\sum_{i=1}^N\int_{\cG}\car_{x\in J_i\setminus\{O\}}\left[\partial_t\phi(x,t)-Du(x,t)D\phi(x,t)\right]m(t)(dx)\\=\frac{d}{dt}\left[\int_{\cG}\phi(x,t)m(t)(dx)\right].
\end{multline}
For any $i=1,\dots, N$, 
\begin{equation}
\label{eq:defqi}  
 q_i= \frac d {dt} \left(\int_\G \car_{\gamma(\cdot)\in J_i\setminus\{O\}} \mu(d\gamma)\right)=  \frac d {dt} \Bigl( m(\cdot)\left(J_i\setminus\{O\}\right) \Bigr),
\end{equation}
is well defined in  $ {\mathcal D}'(0,T)$,
 and there holds 
\begin{equation}\label{eq:claim2_EC}
\frac d {dt}\left[m(\cdot)(\{O\})\right]+\sum_{i=1}^N q_i=0
\end{equation}
in the sense of ${\mathcal D}'(0,T)$.
\end{theorem}
\begin{remark}
In fact, in the proof of Theorem \ref{thm:cont_eq}, we will also 
obtain that, for all $i=1,\dots,N$,
\begin{multline}\label{eq:continuity22}
\frac d {dt} \left(\int_\G \car_{\gamma(\cdot)\in J_i\setminus\{O\}} \phi(\gamma(\cdot),\cdot)
\mu(d\gamma)\right)= \\ \phi(O,\cdot) q_i+ \int_{\cG}\car_{x\in J_i\setminus\{O\}}\left[\partial_t\phi(x,\cdot)-Du(x,\cdot)D\phi(x,\cdot)\right]m(\cdot)(dx)
\end{multline}
in the sense of ${\mathcal D}'(0,T)$.    
\end{remark}

\begin{remark}
Equation \eqref{eq:continuity} implies in particular that
for all $\phi\in C_0^\infty (\cG\times (0,T) )$,
\begin{multline*}
\int_0^T m(t)(\{O\})\partial_t\phi(O,t) dt \\
+\sum_{i=1}^N   \int_0^T \int_{\cG}\car_{x\in J_i\setminus\{O\}}\left[\partial_t\phi(x,t)-Du(x,t)D\phi(x,t)\right]m(t)(dx) dt =0.
\end{multline*}
\end{remark}

\medskip

Before giving the proof of Theorem  \ref{thm:cont_eq}, let us state a few useful lemmas.
\begin{lemma}\label{lemma:1_13apr}
For any $x\in\cG$, $t\mapsto m(t)(\{x\})$ is a measurable bounded function on~$[0,T]$. In particular, it admits a derivative in ${\mathcal D}'(0,T)$.
Moreover, for all $x\in \cG$, the set $\{t\in[0,T]\;:\; m(t)(\{x\})>0\}$ is measurable.
\end{lemma}
\begin{proof}
We consider only $x=O$ because the other cases are similar or simpler. Let us introduce a continuous and piecewise linear function $\phi_\epsilon$ on $\cG$ such that $\phi_\epsilon(O)=1$ and $\phi_\epsilon(x)=0$ for any $x\in\cG$ with $d(x,O)\geq \epsilon$. Clearly, $\{\phi_\epsilon\}_\epsilon$ is a monotone sequence of Lipschitz continuous functions with 
\[
\lim_{\epsilon\to 0}\phi_\epsilon(x)=\left\{\begin{array}{ll}
1&\qquad\textrm{if }x=O\\0&\qquad \textrm{if }x\ne O.
\end{array}\right.
\]
Monotone convergence theorem ensures: $m(t)(\{O\})=\lim_{\epsilon\to 0}\int_\cG\phi_\epsilon(x)m(t)(dx)$ for each $t\in[0,T]$.  On the other hand, by the definition of $\cPmzeroGammaLipCV$, for each $\epsilon>0$, the map $t\mapsto\int_\cG\phi_\epsilon(x)m(t)(dx)$ is Lipschitz continuous.
Hence, $t\mapsto m(t)(\{O\})$ is a measurable function because it is the pointwise limit of a sequence of (Lipschitz) continuous functions. In particular, $\{t\in[0,T]\;:\; m(t)(\{O\})>0\}$ is measurable. 
\end{proof}
\begin{lemma}\label{lemma:p5_13apr}
Consider $\phi$ as in Theorem~\ref{thm:cont_eq}. Then, for any $i\in\{1,\dots,N\}$, the function
\[
\mathcal{F}_{\phi,i}(t)= \int_{\Gamma}\phi(\gamma(t),t)\car_{\gamma(t)\in J_i\setminus\{O\}}\mu(d\gamma)
\]
is a bounded measurable function on~$(0,T)$. In particular, $\mathcal{F}_{\phi,i}$ admits a derivative in ${\mathcal D}'(0,T)$.
\end{lemma}
\begin{proof}
Fix $\phi$ and $i$ as in the statement. Consider a family of functions~$\{\psi_\epsilon\}$ such that: 
$\psi_\epsilon(x)\in [0,1]$, $\psi_\epsilon\in C^\infty(\cG)$, $\textrm{supp}(\psi_\epsilon)\subset J_i\setminus\{O\}$ and $\psi_\epsilon(x)=1$ for $d(x,O)\geq \epsilon$. Clearly, $\lim_{\epsilon\to 0}\psi_\epsilon(x)=\car_{J_i\setminus\{O\}}(x)$ for any $x\in \cG$.\\
The functions
\[
t\mapsto \int_{\Gamma}\phi(\gamma(t),t)\psi_\epsilon(\gamma(t))\mu(d\gamma)
\]
are (Lipschitz) continuous. On the other hand, from the dominated convergence theorem, 
\[
\lim_{\epsilon\to 0}\int_{\Gamma}\phi(\gamma(t),t)\psi_\epsilon(\gamma(t))\mu(d\gamma)= \mathcal{F}_{\phi,i}(t).
\]
Being the pointwise limit of bounded continuous functions, the function~$\mathcal{F}_{\phi,i}$ is measurable and bounded.
\end{proof}
\begin{proof}[Proof of Theorem~\ref{thm:cont_eq}]
Fix $\phi$ as in the statement. By the regularity of~$\phi$ and of~$m$ with respect to~$t$, the function
\[
\zeta(t)=\int_\cG\phi(x,t)m(t)(dx)
\]
is Lipschitz continuous on~$(0,T)$; in particular, $\partial_t\zeta\in L^\infty(0,T)$. There holds
\begin{eqnarray*}
\zeta(t)&=&\phi(O,t)m(t)(\{O\})+\sum_{i=1}^N \int_\cG\phi(x,t)\car_{x\in J_i\setminus\{O\}} m(t)(dx)\\
&=&\phi(O,t)m(t)(\{O\})+\sum_{i=1}^N \int_\Gamma\phi(\gamma(t),t)\car_{\gamma(t)\in J_i\setminus\{O\}}\mu(d\gamma).
\end{eqnarray*}
Note that Lemma~\ref{lemma:1_13apr} and Lemma~\ref{lemma:p5_13apr} ensure that each contribution in the right hand side of the latter identity  has a derivative in ${\mathcal D}'(0,T)$. We may therefore calculate the distributional derivative of $\zeta$. From now on, the notation $\langle \cdot,\cdot\rangle $ stands for the duality between ${\mathcal D}'(0,T)$ and $C^\infty_0(0,T)$.
We claim that, for distributions $q_i\in {\mathcal D}'(0,T)$, $i=1,\dots,N$, that will be characterized later, 
%(see definition~\eqref{eq:def_qi} below), 
there holds
\begin{multline}\label{eq:claim_CE}
\frac {d\zeta}{d t}=\frac {d}{d t}[\phi(O,\cdot)m(\cdot)(\{O\})] +\sum_{i=1}^N\phi(O,\cdot)q_i\\+\sum_{i=1}^N \int_\cG\left[\partial_t \phi(x,\cdot)-D\phi(x,\cdot) Du(x,\cdot)\right]\car_{x\in J_i\setminus\{O\}}m(\cdot)(dx)
\end{multline}
in the sense of ${\mathcal D}'(0,T)$.
Indeed, for every test function $\chi\in C^\infty_0(0,T)$, 
\begin{equation}\label{caim_EC_1}
    \left\langle\frac {d\zeta}{d t},\chi\right\rangle - \left\langle\frac d {dt} [\phi(O,\cdot)m(\cdot)(\{O\})],\chi\right\rangle =\sum_{i=1}^N I_i
\end{equation}
with 
\begin{equation}\label{caim_EC_1_1}
I_i=-\int_0^T\chi'(t)\left[\int_\Gamma\phi(\gamma(t),t)\car_{\gamma(t)\in J_i\setminus\{O\}}\mu(d\gamma)\right]dt.
\end{equation}
%Let $I_i$ stand for the~$i^{\textrm{th}}$ contribution in the right hand side:
Consider a function $\psi \in C^\infty(\cG)$, $\textrm{supp}(\psi)\subset J_i\setminus\{O\}$, $\psi(x)=1$ for all $x\in J_i$ such that $d(x,O)\geq 1$ and $\psi|_{J_i}$ is increasing with respect to $d(x,O)$. Setting $\psi_\epsilon(x)= \psi(\frac x \epsilon)$, we observe that $\psi_\epsilon$ converges pointwise to $\car_{J_i\setminus \{O\}}$ has $\epsilon\to 0$. Hence,
\begin{eqnarray*}
I_i&=&-\lim_{\epsilon\to 0}\int_0^T\chi'(t)\left[\int_\Gamma\phi(\gamma(t),t)\psi_\epsilon(\gamma(t))\mu(d\gamma)\right]dt.
\end{eqnarray*}
From Remark~\ref{rmk:5_13apr} and the definition of $\psi_\epsilon$,  for all $t\in (0,T]$, for $\mu$- a.a. $\gamma$,
\begin{eqnarray*}
    \begin{split}
       \ds \car_{\gamma(t)\in J_i\setminus\{0\}} &= \car_{\gamma(t)\in J_i\setminus\{0\}  } \car_{(\gamma(t),t)\in Q_m}, \\ 
       \ds \psi_\epsilon(\gamma(t)) &= \psi_\epsilon(\gamma(t)) \car_{(\gamma(t),t)\in Q_m}.
    \end{split}
\end{eqnarray*}
%\begin{eqnarray*}
%I_i&=&-\lim_{\epsilon\to 0}\int_0^T\chi'(t)\left[\int_\Gamma\phi(\gamma(t),t)\psi_\epsilon(\gamma(t))\car_{(\gamma(t),t)\in Q_m}\mu(d\gamma)\right]dt.
%\end{eqnarray*}
Therefore, Proposition~\ref{prp:CCC2_4.1+4.2}-$(b)$ and Lemma~\ref{lemma:reg_u_mfg}-$(b)$ (in particular the validity of Lemma~\ref{lemma:note4.9}-$(ii)$) guarantee
that
\begin{equation}\label{eq:OS_final}
\car_{\gamma(t)\ne O}\gamma'(t) = - \car_{\gamma(t)\ne O} Du(\gamma(t),t) \qquad \textrm{for }\mu-\textrm{a.e. } \gamma.
\end{equation}
Since the right hand side of \eqref{eq:OS_final} is the limit as $h\to 0$ of $ -\car_{\gamma(t)\ne O}  (u(\gamma(t)+h) - u(\gamma(t))/h $ as $h\to 0$, it is measurable and essentially bounded w.r.t. $\mu$, and so is the function in the left hand side of \eqref{eq:OS_final}. Hence, 
observing also  that $0\notin \textrm{supp}(\psi_\epsilon)$,
differentiation under the integral sign is permitted for  $\ds t\mapsto \int_\Gamma\phi(\gamma(t),t)\psi_\epsilon(\gamma(t))%\car_{(\gamma(t,t))\in Q_m}
\mu(d\gamma)$. We get 
\begin{eqnarray*}
I_i=\lim_{\epsilon\to 0}  \int_0^T\chi(t)
\int_\Gamma 
\left[
\begin{array}[c]{l} \displaystyle
\Bigl(\partial_t \phi(\gamma(t),t)-D\phi(\gamma(t),t)\cdot Du(\gamma(t),t)\Bigr)\psi_\epsilon (\gamma(t)) 
\\ \displaystyle
-\phi(\gamma(t),t) D\psi_\epsilon(\gamma(t))\cdot Du(\gamma(t),t) 
\end{array}
\right]\mu(d\gamma)
%\car_{(\gamma(t),t)\in Q_m}\mu(d\gamma) 
dt.
%\\
%&&\qquad\left.\left. - \phi(\gamma(t),t) D\psi_\epsilon(\gamma(t))\cdot Du(\gamma(t),t)\right] \car_{(\gamma(t),t)\in Q_m}\mu(d\gamma)\right]dt.
\end{eqnarray*}
%by Remark~\ref{rmk:5_13apr} and by the presence of $\psi_\epsilon$, we can write
%\begin{eqnarray*}
%I_i&=&\lim_{\epsilon\to 0}\int_0^T\chi\left[\int_\Gamma\left[(\partial_t \phi+D\phi\gamma')\psi_\epsilon+\phi D\psi_\epsilon\gamma'\right]\car_{(\gamma(t),t)\in Q_m}\mu(d\gamma)\right]dt.
%\end{eqnarray*}
%Proposition~\ref{prp:CCC2_4.1+4.2}-$(b)$ and Lemma~\ref{lemma:reg_u_mfg}-$(b)$ (in particular the validity of the statement of Lemma~\ref{lemma:note4.9}-$(ii)$) guarantee that, for $\mu$-a.e. $\gamma\in \Gamma$ with $(\gamma(t),t)\in Q_m$, there holds
%\begin{equation}\label
%\gamma'(t)=-Du(\gamma(t),t).
%\end{equation}
Hence, 
\begin{eqnarray*}
I_i&=&\lim_{\epsilon\to 0}(I_{i1}+I_{i2})
\end{eqnarray*}
where
\begin{eqnarray*}
I_{i1}&=&
\int_0^T\chi(t) \int_\Gamma(\partial_t \phi (\gamma(t),t)-D\phi(\gamma(t),t) Du(\gamma(t),t))\psi_\epsilon(\gamma(t)) %\car_{(\gamma(t),t)\in Q_m}
\mu(d\gamma) dt\\
I_{i2}&=&-\int_0^T\chi(t)\int_\Gamma\phi(\gamma(t),t) D\psi_\epsilon (\gamma(t)) Du(\gamma(t),t)%\car_{(\gamma(t),t)\in Q_m}
\mu(d\gamma)dt.
\end{eqnarray*}
Dominated convergence theorem yields
\begin{eqnarray}\notag
\lim_{\epsilon\to 0} I_{i1}&=&\int_0^T\chi(t) \int_\Gamma\left[\partial_t \phi(\gamma(t),t)-D\phi (\gamma(t),t)Du(\gamma(t),t)\right]
\car_{\gamma(t)\in J_i\setminus\{O\}}
%\car_{(\gamma(t),t)\in Q_m}
\mu(d\gamma) dt\\ \label{eq:p7_13apr}
&=&\int_0^T\chi(t) \left[\int_\cG\left[\partial_t \phi(x,t)-D\phi(x,t) Du(x,t)\right]\car_{x\in J_i\setminus\{O\}}m(t)(dx)\right]dt.
\end{eqnarray}
On the other hand, there holds
\begin{eqnarray*}
I_{i2}&=&-\int_0^T\chi(t)\left[\int_\Gamma\left[\phi(\gamma(t),t)-\phi(O,t)\right] D\psi_\epsilon 
(\gamma(t))Du(\gamma(t),t)
%\car_{(\gamma(t),t)\in Q_m}
\mu(d\gamma)\right]dt\\
&&\qquad-\int_0^T\chi(t)\phi(O,t)\int_\G D\psi_\epsilon (\gamma(t))Du(\gamma(t))%\car_{(\gamma(t),t)\in Q_m}
\mu(d\gamma)dt\\
&=:& I_{i3}+I_{i4}.
\end{eqnarray*}
We now deal with $I_{i3}$ and $I_{i4}$ separately. First, from the regularity of~$\phi$, 
\begin{eqnarray*}
  I_{i3}&=&-\int_0^T\chi(t)\int_\Gamma\left[D\phi(O,t)\bar \gamma(t)+\epsilon o(1)\right] D\psi_\epsilon(\gamma(t))\ Du(\gamma(t),t)%\car_{(\gamma(t),t)\in Q_m}
  \mu(d\gamma)dt\\
&=&-\int_0^T\chi(t)\left[\int_\Gamma D\phi(O,t)D \beta_\epsilon(\bar\gamma(t)) Du(\gamma(t),t)%\car_{(\gamma(t),t)\in Q_m}
\mu(d\gamma)\right]dt+o(1)
\end{eqnarray*}
where $\gamma(t)=\bar\gamma(t)e_i$, $\beta_\epsilon$ is defined by
\begin{equation*}
\beta_\epsilon(\bar x)=\int_0^{\bar x}\xi D\psi_\epsilon(\xi e_i)d\xi \qquad \forall \bar x\in[0,\infty).
\end{equation*} and 
 $o(1)$ stands for a function of $\epsilon$ such that $\lim_{\epsilon\to 0}o(1)=0$.
Note that $\beta_\epsilon$ is an increasing regular function and fulfills: $\beta_\epsilon(0)=0$, $\beta_\epsilon(\bar x)$ is a constant for $\bar x\geq \epsilon$ with $\beta_\epsilon(\epsilon)=O(\epsilon)$. Therefore, taking into account the regularity of $\phi$ and of $\chi$ 
and arguing as above, %relation~\eqref{eq:OS_final}, 
we get
\begin{eqnarray} \notag
I_{i3}&=&
-\int_0^T\chi(t) \int_\Gamma D\phi(O,t)D \beta_\epsilon(\bar\gamma(t)) Du(\gamma(t),t)\ + \partial_t D\phi(O,t)\beta_\epsilon(\bar\gamma(t))\mu(d\gamma)dt+o(1)\\ \notag
&=&
-\int_0^T\chi(t) \int_\Gamma \partial_t\Bigl(D\phi(O,\cdot)\beta_\epsilon(\bar\gamma(\cdot))\Bigr)\mu(d\gamma) dt\;+\;o(1)\\ \notag
&=&
\int_0^T\chi' (t) \int_\Gamma D\phi(O,\cdot)\beta_\epsilon(\bar\gamma(\cdot))\mu(d\gamma)dt\;+\;o(1)\\ \label{eq:7777}
&=&
o(1)
\end{eqnarray}
where the last line is due to the properties of $\beta_\epsilon$.\\
From~\eqref{eq:7777} and~\eqref{eq:p7_13apr}, we deduce that $\lim_{\epsilon\to 0} I_{i4} = I_i- \lim_{\epsilon\to 0} I_{i1} $.
\\
Because we can choose $\phi(O,\cdot)=1$ on $\textrm{supp}(\chi)$, this in particular implies  that 
\begin{equation*}
t\mapsto -\int_\G D\psi_\epsilon(\gamma(t)) Du(\gamma(t),t)
%\car_{(\gamma(t),t)\in Q_m}
\mu(d\gamma)
\end{equation*} tends to some $q_i$ in ${\mathcal D}'(0,T)$ as $\epsilon\to 0$. Hence,
\begin{equation}\label{eq:def_qi}
t\mapsto -\phi(O,t)\int_\G D\psi_\epsilon(\gamma(t)) Du(\gamma(t),t)
%\car_{(\gamma(t),t)\in Q_m}
\mu(d\gamma)
\end{equation}
tends to $\phi(O,\cdot) q_i$  in ${\mathcal D}'(0,T)$ as $\epsilon\to 0$. We have obtained \eqref{eq:continuity22}.
%TOGLIERE
%where the limit is in the sense of distributions (note that this limit exists because the function $t\mapsto \int_\G D\psi_\epsilon Du\car_{(\gamma(t),t)\in Q_m}\mu(d\gamma)$ is bounded) ?????VERO????? DA CONTROLLARE so we have
%\begin{equation*}
%\lim_{\epsilon\to0} I_{i4}=< \phi(O,\cdot)q_i, \chi>.
%\end{equation*}
Injecting \eqref{eq:continuity22} into \eqref{caim_EC_1} yields \eqref{eq:claim_CE}.

In order to complete the proof, there remains to check 
\eqref{eq:claim2_EC} and \eqref{eq:defqi}.
%that
%\begin{equation}\label{eq:claim2_EC}
%\frac d {dt}\left[m(\cdot)(\{O\})\right]+\sum_{i=1}^N %q_i=0.
%\end{equation}
Clearly, there holds: $\frac  d {dt} \Bigl(m(\cdot)(\{O\})+\sum_i m(\cdot)\left(J_i\setminus\{O\}\right)\Bigr)=0$. Hence, \eqref{eq:claim2_EC} will follow from
\eqref{eq:defqi}.

%\begin{equation}\label{eq:claim2_EC1}
%\partial_t \Bigl( m(\cdot)\left(J_i\setminus\{O\}\right)%\Bigr)=q_i.
%\end{equation}
%{\color{red} Suppress? [Claudio: OK for me.] Choose $\phi$ as in the statement and satisfying furthermore $\phi(\cdot,t)=1$ on $\textrm{supp}(m(t))$. Then \eqref{eq:claim_CE}  yields
%\begin{equation*}
%0=\partial_t\left(\int_\cG m(t)(dx)\right)=\partial_t \Bigl( m(\cdot)%(\{O\})\Bigr) +\sum_{i=1}^Nq_i.
%\end{equation*}}
From the definition of $q_i$ as the limit of 
$-\int_\G D\psi_\epsilon(\gamma(\cdot)) Du(\gamma(\cdot),\cdot)
\mu(d\gamma)$ as $\epsilon$ tends to $0$
and from~\eqref{eq:OS_final},  there holds
\begin{eqnarray*}
\langle q_i,\chi\rangle&=&-\lim_{\epsilon\to0} \int_0^T\left[\int_\G D\psi_\epsilon(\gamma(t)) Du(\gamma(t),t)%\car_{(\gamma(t),t)\in Q_m}
\mu(d\gamma)\right]\chi(t)dt\\
&=&\lim_{\epsilon\to0} \int_0^T\left[\int_\G \partial_t\psi_\epsilon(\gamma(t))%\car_{(\gamma(t),t)\in Q_m}
\mu(d\gamma)\right]\chi(t)dt\\
&=&-\lim_{\epsilon\to0} \int_0^T\left[\int_\G \psi_\epsilon(\gamma(t))%\car_{(\gamma(t),t)\in Q_m}
\mu(d\gamma)\right]\chi'(t)dt\\
&=&-\int_0^T\left(\int_\G \car_{\gamma(t)\in J_i\setminus\{O\}}
\mu(d\gamma)\right)\chi'(t)dt
\end{eqnarray*}
for every test function $\chi\in C^\infty_0(0,T)$,
where the last equality is due to the dominated convergence theorem.
%and is equivalent to~\eqref{eq:claim2_EC1}. 
The proof of \eqref{eq:claim2_EC} 
with \eqref{eq:defqi}  is achieved. 
\end{proof}

\noindent{\bf Acknowledgements.} 
The first author is partially on academic leave at Inria for the years 2021-22 and 2022-23 and acknowledges the hospitality of this institution during this period. The first author is partially supported by the chair Finance and Sustainable Development and FiME Lab (Institut Europlace de Finance).
The first and fourth authors are  partially supported by ANR (Agence Nationale de la Recherche) through project COSS, ANR-22-CE40-0010-01. 
The second and the third authors are partially supported by INDAM-GNAMPA.  The fourth author benefits from the support of the French government “Investissements d’Avenir” program integrated to France 2030, bearing the following reference 
ANR-11-LABX-0020-01 and acknowledges INDAM-GNAMPA for the visiting professor position at the University of Padova.


{\small
\begin{thebibliography}{1}

\bibitem{AHLLM} {\sc Y. Achdou, J. Han, J.M. Lasry, P.L. Lions, B. Moll}
\newblock{\em{Income and wealth distribution in macroeconomics: a continuous-time approach}},  Rev. Econ. Stud. 89 (2022), no. 1, 45--86.

\bibitem{ADLT2020}
  Y. Achdou, M-K.  Dao, O. Ley,  N. Tchou,
  \newblock {\em{Finite Horizon Mean Field Games on Networks}},
 Calculus of Variations  and  P.D.E. 59 (2020), no 5,

 \bibitem{ADLT2018}
  Y. Achdou, M-K.  Dao, O. Ley,  N. Tchou,
  \newblock A class of infinite horizon mean field games on networks.
\newblock {\em Netw. Heterog. Media}  14  (2019),  no. 3, 537--566.




%\bibitem{AC08} {\sc Y. Achdou, I. Capuzzo Dolcetta,}
%\newblock{\em{Approximation of solutions of Hamilton-Jacobi equations on the Heisenberg group}}, \newblock{M2AN Math. Model. Numer. Anal.} 42 (2008), no.4, 565--591.

%\bibitem{ACCT2} {\sc Y. Achdou, F. Camilli, A. Cutr\`i and N. Tchou,}
%\newblock{\em{Hamilton-Jacobi equations on networks}}, In: Proceedings of the 18th IFAC World Congress, Milano, vol. 18, part 1 (2011).

\bibitem{ACCT} {\sc Y. Achdou, F. Camilli, A. Cutr\`i, N. Tchou,}
\newblock{\em{Hamilton-Jacobi equations constrained on networks}},
\newblock{NoDEA Nonlinear Differential Equations Appl.} 20 (2013), 413--445.

%\bibitem{AMMT} {\sc Y. Achdou, P. Mannucci, C. Marchi, N. Tchou,}
%\newblock{\em{Deterministic mean field games with control on the acceleration}},
%\newblock{NoDEA Nonlinear Differential Equations Appl.}, 27 (2020), no. 3, p. 33. 

\bibitem{AMMT2} {\sc Y. Achdou, P. Mannucci, C. Marchi, N. Tchou,}
\newblock{\em{Deterministic mean field games with control on the acceleration and state constraints}},
\newblock{SIAM J. Math. Anal.} 54 (2022), no. 3, 3757–3788.

\bibitem{AOT} {\sc Y. Achdou, S. Oudet, N. Tchou,}
\newblock{\em{Hamilton-Jacobi equations for optimal control on junctions and networks}}, ESAIM Control Optim. Calc. Var. 21 (2015), no. 3, 876--899. 

%\bibitem{AC} {\sc Y. Achdou, I. Capuzzo Dolcetta,}
%\newblock{\em{Mean field games: numerical methods}}, \newblock{SIAM J.
%Numer. Anal.} 48 (2010), 1136--1162.

\bibitem{GRA}{\sc F. Al Saleh, T. Bakaryan, D.A. Gomes, R. Ribeiro}
\newblock{First-order mean-field games on networks and Wardrop equilibrium}, available at https://arxiv.org/abs/2207.01397.

\bibitem{AF}{\sc J.P. Aubin, H. Frankowska}
\newblock{\em{Set-Valued analysis}}, \newblock Systems \& Control: Foundations \& Applications 2, Birkha\"user Boston, Boston, MA, 1990.


\bibitem{AGS} {\sc L. Ambrosio, N. Gigli, G. Savar\'e,}
\newblock{\em{Gradient flows in metric spaces and in the space of probability measures}}, \newblock  Lectures in Mathematics ETH Z\"{u}rich. Birkha\"{u}ser Verlag, Basel 2005.

%\bibitem{BGL} {\sc D. Bakry, I. Gentil, M. Ledoux,}
%\newblock{\em {Analysis and geometry of Markov diffusion operators}}, \newblock Grundlehren der mathematischen Wissenchaften 348, Springer, 2014.

%\bibitem{Baldi} {\sc P. Baldi}
%\newblock{\em{Equazioni differenziali stocastiche e applicazioni}}, second edition, Quaderni della Unione Matematica Italiana 28, Bologna 2000.

%\bibitem{BCP} {\sc Z. Balogh, A. Calogero, R. Pini}\newblock{\em{The Hopf-Lax formula in Carnot groups: a control theoretic approach}}. \newblock Calc. Var. Partial Differential Equations, 49 (2014), no. 3-4, 1379--1414. 

\bibitem{BCD} {\sc M. Bardi, I. Capuzzo Dolcetta,}
\newblock {\em {Optimal control and viscosity solutions of Hamilton-Jacobi Bellman equations}}, \newblock Systems and Control: Foundations and Applications. Birkha\"user, Boston 1997.

%\bibitem{BD} {\sc M. Bardi, F. Dragoni},
%\newblock{\em{Subdifferential and properties of convex functions with respect to vector fields}}, \newblock J. Convex Anal. 21 (2014), no. 3, 785--810.



\bibitem{BC2023}{\sc G. Barles, E. Chasseigne,}
\newblock{\em{An Illustrated Guide of the Modern Approaches of {H}amilton-{J}acobi Equations and Control Problems with Discontinuities}}, \newblock 
arXiv 1812.09197 (2023).


\bibitem{BBC} {\sc G. Barles, A. Briani, E. Chasseigne,}
\newblock{\em{A Bellman approach for regional optimal control problems in $\R^N$}}, \newblock SIAM J. Control Optim. 52 (2014), no. 3, 1712--1744. 

%\bibitem{Bel} {\sc A. Bella\"\i che}
%\newblock{\em{The tangent space in sub-Riemannian geometry}}, in Sub-Riemannian geometry, 1--78, Progr. Math., 144, Birkha\"user, Basel, 1996

\bibitem{BB00}{\sc J.-D. Benamou, Y. Brenier}
\newblock{\em{A computational fluid mechanics solution to the Monge-
Kantorovich mass transfer problem}}, Numerische Mathematik, 84(3):375--393, 2000.

\bibitem{BC15} {\sc J.-D. Benamou, G. Carlier}
\newblock{\em{Augmented Lagrangian methods for transport optimization,
mean field games and degenerate elliptic equations}}, Journal of Optimization Theory and Applications, 167(1):1--26, 2015.

%\bibitem{BCS} {\sc J.-D. Benamou, G. Carlier, F. Santambrogio}
%\newblock{\em{Variational mean field games}}, Active particles. Vol. 1. Advances in theory, models, and applications, 141--171, Birkh\"auser/Springer, Cham, 2017.

%\bibitem{BFY} {\sc A. Bensoussan,  J. Frehse, P. Yam},
%\newblock{\em{Mean field games and mean field type control theory}}, Springer Briefs in Mathematics. Springer, New York 2013.

%\bibitem{BU} {\sc P. Besala, H. Ugowski},
%\newblock{\em{Some uniqueness theorems for solutions of parabolic and elliptic partial differential equations in unbounded regions}}, \newblock  Colloq. Math. 20 (1969), 127--141.

%
%\bibitem{BW} {\sc I. Birindelli, J. Wigniolle},
%\newblock{\em{Homogenization of Hamilton-Jacobi equations in the Heisenberg group}}, \newblock Commun. Pure Appl. Anal. 2 (2003), no. 4, 461--479. 
%

%\bibitem{Bir} {\sc M. Biroli},
%\newblock{\em{Subelliptic Hamilton-Jacobi equations: the coercive evolution case}}, \newblock Appl. Anal. 92 (2013), no. 1, 1--14.

%\bibitem{BMT2} {\sc M. Biroli, U. Mosco, N. Tchou},
%\newblock{\em{Homogenization for degenerate operators with periodical coefficients with respect to the Heisenberg group}}, \newblock C. R. Acad. Sci. Paris S\'er.~I Math. 322 (1996), no. 5, 439--444.
%
%\bibitem{BMT} {\sc M. Biroli, U. Mosco, N. Tchou},
%\newblock{\em{Homogenization by the Heisenberg group}}, \newblock  Adv. Math. Sci. Appl. 7 (1997), no. 2, 809--831.
%

%\bibitem{BK} {\sc V.I. Bogachev, A.V. Kolesnikov},
%\newblock{The Monge-Kantorovich problem: achievements, connections, and perspectives}, \newblock Russ. Math. Surv. 67 (2012), no. 5, 785--890.

%\bibitem{Kr} {\sc  V.I. Bogachev, N.V. Krylov, M. R\" ockner, S.V. Shaposhnikov,}
%\newblock{\em{Fokker-Planck-Kolmogorov equations}},
%\newblock{Mathematical Surveys and Monographs, 207. American Mathematical Society},  Providence, RI, 2015.

%\bibitem{BLU} {\sc A. Bonfiglioli, E. Lanconelli, F. Uguzzoni},
%\newblock{\em{Stratified Lie groups and potential theory for their sub-Laplacians}}, \newblock Springer Monographs in Mathematics, Springer, Berlin, 2007.

%  
%\bibitem{BC} {\sc  B. Bonnard, M. Chyba},
%\newblock{\em{Singular trajectories and their role in control theory}}, Mathematics \& Applications,
%40,  Springer-Verlag, Berlin, 2003.

%\bibitem{BB07} {\sc M. Bramanti, L. Brandolini},
%\newblock{\em{Schauder estimates for parabolic nondivergence operators of H\"ormander type}}, J. Differential Equations 234 (2007), no. 1, 177--245. 
%

%\bibitem{BBLU} {\sc M. Bramanti, L. Brandolini, E. Lanconelli, F. Uguzzoni},
%\newblock{\em{Non-divergence equations structured on H\"ormander vector fields: heat kernels and Harnack inequalities}}. Mem. Amer. Math. Soc. 204 (2010), no. 961.

%\bibitem{BCQ} {\sc  R. Buckdahn, P. Cannarsa and M. Quincampoix},
%\newblock{\em{Lipschitz continuity and semiconcavity
%properties of the value function of a stochastic control problem}}, Nonlinear Differ. Equ. Appl. 17 (2010), 715--728.


\bibitem{cm16}
{\sc F. Camilli, C. Marchi,}
\newblock {\em {Stationary mean field games systems defined on networks}}.
SIAM J. Control Optim., 54(2):1085--1103, 2016.

\bibitem{CC} {\sc P. Cannarsa, R. Capuani,}
\newblock {\em {Existence and uniqueness for Mean Field Games with state constraints}}, \newblock  PDE models for multi-agent phenomena, 49--71,
Springer INdAM Ser., 28, Springer, Cham, 2018.


\bibitem{CCC1} {\sc P. Cannarsa, R. Capuani, P. Cardaliaguet}
\newblock {\em {$C^{1,1}$-smoothness of constrained solutions in the calculus of variations with application to mean field games}}, \newblock  Math. Eng. 1 (2019), no. 1, 174--203. 

\bibitem{CCC2} {\sc P. Cannarsa, R. Capuani, P. Cardaliaguet}
\newblock {\em {Mean field games with state constraints: from mild to pointwise solutions of the PDE system}}, \newblock  Calc. Var. Partial Differential Equations 60 (2021), no. 3, Paper No. 108.

%\bibitem{CM} {\sc P. Cannarsa, C. Mendico,}
%\newblock{\em{ Mild and weak solutions of mean field game problems for linear control systems}},
%\newblock Minimax Theory Appl. 5 (2020), no. 2, 221--250.

%\bibitem{CR} {\sc P. Cannarsa, L. Rifford,}
%\newblock{\em{Semiconcavity results for optimal control problems admitting no singular minimizing controls}},
%\newblock  Ann. Inst. H. Poincar\'e Anal. Non Lin\'eaire 25 (2008), no. 4, 773--802.

\bibitem{CS} {\sc P. Cannarsa, C. Sinestrari,}
\newblock {\em {Semiconcave Functions, Hamilton-Jacobi Equations, and Optimal Control}}, \newblock Progress in Nonlinear Differential Equations and Their Applications, 48, Birkha\"user, Boston 2004.

%\bibitem{CCR} {\sc L. Capogna, G. Citti, G. Rea}
%\newblock{\em{A subelliptic analogue of Aronson-Serrin's Harnack inequality}}, \newblock Math. Ann. 357 (2013), no. 3, 1175--1198. 

\bibitem{C} {\sc P. Cardaliaguet,}
\newblock {\em {Notes on Mean Field Games}}, \newblock from P.L. Lions lectures at College de France (2012), available at
https://www.ceremade.dauphine.fr/ cardalia/MFG20130420.pdf.

\bibitem{CM16}{\sc P. Cardaliaguet, A. R. M\'esz\'aros, F. Santambrogio}
\newblock{\em{First order mean field games with density constraints: pressure equals price}} SIAM J. Control Optim., 54(5):2672--2709, 2016.

%\bibitem{C13} {\sc P. Cardaliaguet,}
%\newblock{\em{Long time average of first order mean field games and weak KAM theory}}, \newblock{Dyn. Games Appl.} 3 (2013), 473--488.

%\bibitem{CGPT} {\sc P. Cardaliaguet, P.J. Graber, A. Porretta, D. Tonon,}
%\newblock{\em{Second order mean field games with degenerate diffusion and local coupling}}, \newblock{NoDEA Nonlinear Differential Equations Appl.} 22 (2015), 1287--1317.

%\bibitem{C} {\sc C. Cinti}
%\newblock{\em{Partial differential equations, uniqueness in the Cauchy problem for a class of hypoelliptic ultraparabolic operators}}, \newblock{Atti Accad. Naz. Lincei Rend. Lincei Mat. Appl.} 20 (2009), no. 2, 145--158. 

%\bibitem{CH} {\sc P. Cardaliaguet, S. Hadikhanloo,}
%\newblock{\em{Learning in mean field games: the fictitious play}},
%\newblock{ ESAIM Control Optim. Calc. Var.} 23 (2017), no. 2, 569--591.
%
%\bibitem{CiSa} {\sc G. Citti, A. Sarti,}
%\newblock{\em{A cortical based model of perceptual completion in the roto-translation space}}, \newblock{J. Math. Imaging Vision} 24 (2006), 307--326.
%
%
%\bibitem{Cla} {\sc F. Clarke,}
%\newblock {\em {Functional Analysis, Calculus of Variations and Optimal Control }}, \newblock Graduate Text in Mathematics 264, Springer-Verlag, London 2013.

%\bibitem{Cla90} {\sc F. Clarke,}
%\newblock {\em {Optimization and nonsmooth analysis}}, \newblock Classics in Applied Mathematics 5, S, Philadelphia, PA, 1990 (2nd edition).

%\bibitem{Cor}  {\sc J.M. Coron,}
%\newblock {\em {Control and nonlinearity}}, \newblock Mathematical Surveys and Monographs, 136. American Mathematical Society, Providence, RI, 2007.

%\bibitem{DL}{\sc F. Da Lio, O. Ley}
%\newblock{\em{Uniqueness results for second-order Bellman-Isaacs equations under quadratic growth assumptions and application}}, \newblock{SIAM J. Control Optim.} 45 (2006), no.1, 74--106.
  
%\bibitem{DFPD}{\sc M. Defoort, T. Floquet, W. Perruquetti, S.V. Drakunov}
%\newblock{\em{Integral sliding mode control of an extended Heisenberg system}}, \newblock{IET Control Theory Appl.} 3 (2009), no. 10, 1409--1424.  
 
%\bibitem{DF} {\sc F. Dragoni, E. Feleqi,}
%\newblock {\em {Ergodic Mean Field Games with H\"ormander diffusions}}, \newblock{Calc. Var. Partial Differential Equations} 57 (2018), no. 5, Art. 116, 22 pp.

%\bibitem{FGT} {\sc E. Feleqi, D. Gomes, T. Tada}
%\newblock {\em {Hypoelliptic mean field games- a case study.}}, \newblock{Minimax Theory Appl.} 5 (2020), no. 2, 305--326.

%\bibitem{GPV} {\sc D. Gomes, E.A. Pimentel, V. Voskanyan}
%\newblock{\em{Regularity theory for mean-field game systems}}, \newblock SpringerBriefs in Mathematics. Springer, Berlin 2016.

%\bibitem{GS} {\sc D. Gomes, J. Saude,}
%\newblock{\em{Mean field games - A brief survey}}, \newblock{Dyn. Games Appl.}
%4 (2014), 110--154.


\bibitem{DM} {\sc S. Dweik, G. Mazanti}
\newblock{Sharp semi-concavity in a non-autonomous control problem and $L^p$ estimates in an optimal-exit MFG},  NoDEA Nonlinear Differential Equations Appl. 27 (2020), no. 2, Paper No. 11.
%\bibitem{HMC} {\sc M. Huang, R.P. Malham\'e, P.E. Caines,}
%\newblock{\em{Large population stochastic dynamic games: closed-loop McKean-Vlasov systems and the Nash certainty equivalence principle}}, \newblock{Commun. Inf. Syst.} 6 (2006), 221--251.

%\bibitem{Ik} {\sc Y. Ikeda,}
%\newblock{\em{The Cauchy problem of linear parabolic equations with discontinuous and unbounded coefficients}}
%\newblock{Nagoya Math. J. } 41 (1971), 33--42.

%\bibitem{I95}{\sc H. Ishii,}
%\newblock{\em{On the equivalence of two notions of weak solutions, viscosity
%   solutions and distribution solutions}}
%\newblock{Funkcial. Ekvac.} 38 (1995),  no. 1, 101--120.

%\bibitem{KS}  {\sc I. Karatzas, S.E. Shreve},
%\newblock{\em{Brownian motion and stochastic calculus}},
%\newblock{Second edition. Graduate Texts in Mathematics}, 113. Springer-Verlag, New York, 1991.

%\bibitem{KRN} {\sc K. Kuratowski, C. Ryll-Nardzewski,}
%\newblock{\em{A general theorem on selectors}}, Bull. Acad. Polon. Sci. S\'er. Sci. Math. Astronom. Phys. 13 (1965), 397--403.

\bibitem{IM}  {\sc C. Imbert, R. Monneau,}
\newblock{\em{Flux-limited solutions for quasi-convex Hamilton-Jacobi equations on networks}}, Ann. Sci. \'Ec. Norm. Sup\'er. (4) 50 (2017), no. 2, 357--448.

\bibitem{IMZ} {\sc C. Imbert, R. Monneau, H. Zidani,}
\newblock{\em{A Hamilton-Jacobi approach to junction problems and application to traffic flows}}, \newblock ESAIM Control Optim. Calc. Var. 19 (2013), no. 1, 129--166. 

%\bibitem{LSU} {\sc O. A. Lady$\check{z}$enskaja, V. A. Solonnikov, N. N. Ural'ceva,}
%\newblock{\em{Linear and quasilinear equations of parabolic type}}. %(Russian)
%\newblock{Translated from the Russian by S. Smith. Translations of Mathematical Monographs}, Vol. 23 American Mathematical Society, Providence, R.I. 1968.

\bibitem{LL1} {\sc J.-M. Lasry, P.-L. Lions,}
\newblock{\em{ Jeux \`a champ moyen. I. Le cas stationnaire}}, \newblock {C. R. Math. Acad. Sci. Paris} 343 (2006), 619--625.

\bibitem{LL2} {\sc J.-M. Lasry, P.-L. Lions,}
\newblock{\em{ Jeux \`a champ moyen. II. Horizon fini et contr\^ole optimal}}, \newblock {C. R. Math. Acad. Sci. Paris} 343 (2006), 679--684.

\bibitem{LL3} {\sc J.-M. Lasry, P.-L. Lions,}
\newblock{\em{Mean field games}}, \newblock {Japan. J. Math. (N.S.)} 2 (2007), 229--260.

%\bibitem{Lie}{\sc  G.M. Lieberman,}
%\newblock{\em{Second order parabolic differential equations}}.
%World Scientific Publishing Co., Inc., River Edge, NJ, 1996.

\bibitem{LS}{\sc P.-L. Lions, P. Souganidis,}
\newblock{\em{Viscosity solutions for junctions: well posedness and stability}}, \newblock Atti Accad. Naz. Lincei Rend. Lincei Mat. Appl. 27 (2016), no. 4, 535--545.

\bibitem{LS17}{\sc P.-L. Lions, P. Souganidis,}
 \newblock{\em{Well-posedness for multi-dimensional junction problems with Kirchhoff-type conditions}},  \newblock Atti Accad. Naz. Lincei Rend. Lincei Mat. Appl.  28  (2017),  no. 4, 807--816.

\bibitem{MS19}{\sc G. Mazanti, F. Santambrogio}
\newblock{\em{Minimal-time mean field games}}, Math. Models Methods Appl. Sci., 29 (2019), no. 8, 1413--1464.
%\bibitem{LM}  {\sc P.L. Lions, M. Musiela,}
%\newblock{\em{Ergodicity of diffusion processes}}, \newblock unpublished.

%\bibitem{MMMT} {\sc P. Mannucci, C. Mariconda, C. Marchi, N. Tchou,}\newblock{\em{Non-coercive first order Mean Field Games}}, \newblock {J. Differential Equations}, 269 (2020), no. 5, 4503--4543

%\bibitem{MS} {\sc J.J. Manfredi, B. Stroffolini} \newblock{\em{A version of the Hopf-Lax formula in the Heisenberg group}}, \newblock {Comm. Partial Differential Equations} 27 (2002), no. 5-6, 1139--1159. 
%

\bibitem{MW}{\sc E.J. McShane, R.B. Warfield Jr},
\newblock{\em{On Filippov's implicit functions lemma}}, Proc. Amer. Math. Soc. 18 (1967), 41--47.
  
%\bibitem{Montg} {\sc R. Montgomery}
%\newblock{\em{A Tour of SubRiemannian Geometries, Their Geodesics and Applications}}, \newblock AMS, Providence, RI, 2002.

\bibitem{Mor}{\sc  P.S. Morfe,}
\newblock{\em{Convergence \& rates for Hamilton-Jacobi equations with Kirchoff junction conditions}}, \newblock NoDEA Nonlinear Differential Equations Appl. 27 (2020), no. 1, Paper No. 10.

%\bibitem{MuSa} {\sc R.M. Murray, S.S. Sastry} \newblock{\em{Nonholonomic motion planning: steering using sinusoids}}, \newblock {IEEE Trans. Automat. Control} 38 (1993), no. 5, 700--716. 

%\bibitem{RS} {\sc L.P. Rothschild, E.M. Stein}
%\newblock{\em{Hypoelliptic differential operators and nilpotent groups}}, Acta Math. 137 (1976), no. 3-4, 247--320.

%\bibitem{STROF} {\sc B. Stroffolini} \newblock{\em{ Homogenization of Hamilton-Jacobi equations in Carnot groups}}, \newblock {ESAIM Control Optim. Calc. Var. 13 (2007), no. 1, 107--119}. 
\end{thebibliography} }
\end{document}